\documentclass{siamltex1213}
\usepackage{color}
\usepackage{xcolor}
\usepackage{amsmath}
\usepackage{amssymb}
\usepackage{mathdots}
\usepackage{graphicx}
\usepackage{subfigure}
\usepackage{url}
\usepackage[utf8]{inputenc}
\usepackage{easybmat}
\newcommand{\e}{\ensuremath\varepsilon}

\newcommand{\la}{\ensuremath\lambda}

\newcommand{\rev}{\ensuremath\mathrm{rev}}

\newcommand{\pdeg}{\ensuremath d}

\newcommand{\lpol}[1]{\ensuremath \mathcal{L}(#1)}
\newcommand{\rowdim}{\ensuremath m}
\newcommand{\coldim}{\ensuremath n}

\newcommand{\EL}[1]{\ensuremath\mathcal{L}_{#1}(\lambda)}

\newcommand{\DEL}[1]{\ensuremath\Delta\mathcal{L}_{#1}(\lambda)}
\newcommand{\FF}{{\mathbb F}}

\newtheorem{remark}[theorem]{Remark}
\newtheorem{example}[theorem]{Example}

\newif\iffinal
\finalfalse

\iffinal

\else

\fi

\title{Block Kronecker Linearizations of Matrix Polynomials and their Backward Errors
  \thanks{This work was partially supported by the ``Ministerio de Econom\'{i}a, Industria y Competitividad of Spain'' and ``Fondo Europeo de Desarrollo Regional (FEDER) of EU'' through grants MTM-2012-32542, MTM-2015-68805-REDT, MTM-2015-65798-P, by the Belgian network DYSCO (Dynamical Systems, Control, and Optimization), funded by the Interuniversity Attraction Poles Programme initiated by the Belgian Science Policy Office, and by the Engineering and Physical Sciences Research Council of UK through grant EP/I005293.}}

\begin{document}

\author{Froil\'{a}n M. Dopico\footnotemark[2] \and Piers W. Lawrence\footnotemark[3]\ \footnotemark[5] \and Javier~P\'{e}rez\footnotemark[4] \and Paul Van Dooren\footnotemark[3] }
\renewcommand{\thefootnote}{\fnsymbol{footnote}}
\footnotetext[2]{Departamento de Matem\'{a}ticas, Universidad Carlos III de Madrid,  Avenida de la Universidad 30, 28911, Legan\'{e}s, Spain. Email: {\tt dopico@math.uc3m.es}.}
\footnotetext[3]{Department of Mathematical Engineering, Universit\'{e} catholique de Louvain,  Avenue Georges Lema\^{i}tre 4, B-1348 Louvain-la-Neuve, Belgium. Email: {\tt paul.vandooren@uclouvain.be}.}
\footnotetext[4]{Department of Computer Science, KU Leuven, Celestijnenlaan 200A bus 2402, B-3001 Leuven, Belgium. Email: {\tt javierpa@gmail.com\emph{}}.}
\footnotetext[5]{Department of Computer Science, KU Leuven, Celestijnenlaan 200A bus 2402, B-3001 Leuven, Belgium. Email: {\tt piers.lawrence@cs.kuleuven.be}.}

\renewcommand{\thefootnote}{\arabic{footnote}}

\maketitle

\begin{abstract}
  We introduce a new family of strong linearizations of matrix polynomials---which we call ``block Kronecker pencils''---and perform a backward stability analysis of complete polynomial eigenproblems. These problems are solved by applying any backward stable algorithm to a block Kronecker pencil, such as the staircase algorithm for singular pencils or the QZ algorithm for regular pencils. This stability analysis allows us to identify those block Kronecker pencils that yield a computed complete eigenstructure which is exactly that of a slightly perturbed matrix polynomial. The global backward error analysis in this work presents for the first time the following key properties: it is a rigurous analysis valid for finite perturbations (i.e., it is not a first order analysis), it provides precise bounds, it is valid simultaneously for a large class of linearizations, and it establishes a framework that may be generalized to other classes of linearizations. These features are related to the fact that block Kronecker pencils are a particular case of the new family of ``strong block minimal bases pencils'', which are robust under certain perturbations and, so, include certain perturbations of block Kronecker pencils. We hope that this robustness property will allow us to extend the results in this paper to other contexts.
\end{abstract}



\pagestyle{myheadings}
\thispagestyle{plain}

\markboth{F.~M. DOPICO, P.~W. LAWRENCE, J. P\'{E}REZ, AND P. VAN DOOREN}{BLOCK KRONECKER LINEARIZATIONS AND BACKWARD ERRORS}


\section{Introduction}
Matrix polynomials appear in many applications in engineering, mechanics, control, linear systems theory, and computer-aided geometric design. They may arise directly or as approximations of highly nonlinear eigenvalue problems. The classical works \cite{LancasterBook,kailath1980linear,rosenbrock1970} and the modern surveys \cite{Mehrmann04nonlineareigenvalue,Tisseur01thequadratic} include discussions of different applications of matrix polynomials. Those readers unfamiliar with matrix polynomials can find in Section \ref{sec:basic} most of the concepts mentioned in this introduction.

Square regular matrix polynomials are related to {\em polynomial eigenvalue problems} (PEPs), i.e., to the computation of all of the eigenvalues of the polynomial, while singular matrix polynomials are related to {\em complete polynomial eigenproblems} (CPEs), i.e., to the computation of all of the eigenvalues and of all of the so-called minimal indices of the polynomial.
Although in the last years the main focus has been on regular matrix polynomials, problems related to singular matrix polynomials are also quite common. Thus, in engineering practice, singular problems allow to add redundancy into the models and, in this way, to regularize ill-conditioned problems \cite{biegler-campbell-mehrmann-eds,Kunkel-Mehrmann-book,mehrmann}. Moreover, singular matrix polynomials are fundamental in the area of systems and control, where they model systems of differential equations whose behavior has to be ``controlled''. This was nicely synthesized in the pioneer work of Rosenbrock \cite{rosenbrock1970}, who introduced quadruples of matrix polynomials $\left\{ T(\lambda), U(\lambda), V(\lambda), W (\lambda) \right\}$ to model such systems. The Smith form \cite{gantmacher1960theory} of the matrix polynomials
\begin{equation*}
 P_p(\lambda)=
 T(\lambda) , \quad P_z(\lambda)=\left[ \begin{array}{cc}
 T(\lambda) & -U(\lambda) \\ V(\lambda) & W (\lambda)
 \end{array} \right],
\end{equation*}
and of the first block row and the first block column of $P_z(\la)$, denoted as $P_c (\la)$ and $P_o (\la)$, respectively,
define the so-called poles and zeros of the transfer function of such systems, as well as the notions of controllability and observability. The matrix polynomial
$P_p(\lambda)$ is square and invertible and defines the poles of the system, which are its natural frequencies. The matrix polynomial $P_z(\lambda)$
may be non-square or singular and describes the zeros of the system, which are the frequencies that are filtered by the system, and the minimal indices that characterize its left and right ``singular'' null space structures. Finally, the Smith form of the non-square matrix polynomials $P_c(\lambda)$ and $P_o(\lambda)$ yields conditions on the controllability and observability of the system.
The importance of computing the finer details of the Smith zeros and minimal indices of a matrix polynomial was already stressed in the eighties \cite{van1981eigenstructure,kailath1980linear}, and was revived later in the behavioral modeling of dynamical systems \cite{markovskybehavior}. It also appears in other problems in this area, as, e.g., in deadbeat control problems \cite{vandeadbeat}. In all of these problems it is very important to have reliable numerical algorithms for computing the relevant structural information of potentially singular matrix polynomials.

The numerical solution of PEPs and CPEs is usually performed by embedding the coefficients of the associated matrix polynomial into a larger linear matrix polynomial, or matrix pencil, called a {\em linearization}, and then applying well-established algorithms for matrix pencils to the linearization, like the QZ algorithm in the regular case \cite{golubvanloan4}, or the staircase algorithm in the singular case \cite{vandoorenstaircase}, potentially enhanced with the stratification of the orbits of pencils \cite{edelelmkags1,edelelmkags2}. This linearization approach for solving PEPs and CPEs was proposed for the first time in \cite{van1981eigenstructure,van1983eigenstructure}, the concept of linearization was formally introduced in \cite{LancasterBook} for regular matrix polynomials, and in \cite{de2008sharp} for singular ones. A thorough treatment of linearizations can be found in \cite{de2014spectral}.

The linearizations used most often to solve PEPs and CPEs are the well known Frobenius companion forms. They are used in \cite{van1983eigenstructure} and in the command {\tt polyeig} of MATLAB. They have many favorable properties; in particular, it was proven in \cite{van1983eigenstructure} that they yield computed solutions of PEPs and CPEs which are exactly those of slightly perturbed matrix polynomials (i.e., from the polynomial point of view they have perfect structured backward stability). However, it is well known that the Frobenius companion forms do not preserve the algebraic structures that are often present in the matrix polynomials arising in applications. Therefore, the rounding errors inherent to numerical computations may destroy qualitative properties of the eigenstructures of such polynomials when they are computed via the Frobenius forms. In addition, it is also known that Frobenius forms do not deliver accurate solutions of PEPs when the matrix coefficients of the polynomial have very different norms; this problem has to date only been addressed in the quadratic case \cite{tisseur-quadeig,su-quadeig}.
These drawbacks have motivated an intense activity in the last few years towards the development and analysis of new classes of linearizations of matrix polynomials, with special emphasis on linearizations that preserve certain structures important in applications (see, as a small sample,  \cite{amiraslani-corless-lancaster,p832,bini-robol,Bueno_structuredstrong,bueno-ter-dop-general,bueno2015large,Bueno_palindromiclinearizations,
DDM2010Fiedler,DeTeran:2011,DDM2012rectangular,q077,Tisseur06symmetriclinearizations,
Mackey_structuredpolynomial,DLP,nakat-nofer-townsed,ChebyFiedler,greeks2011}).

A key open problem in this area is that global backward error analyses of PEPs and CPEs solved by the new classes of linearizations have not yet been developed, and, so, it is not known if their use combined with the QZ or the staircase algorithm is backward stable from the polynomial point of view. The only backward error analyses available in this context are the ``local'' residual analyses valid for each particular computed eigenpair in the case of the linearizations in vector spaces \cite{Higham06backwarderror,Higham05theconditioning,tisseurlaa2000}, and a few first order global backward error analyses valid for particular ``colleague'' linearizations \cite{lawrence-corless-2015,LVBVD,Chebyrootfinding} or for the Frobenius linearizations \cite{van1983eigenstructure}. Two obstacles for extending these global backward error analyses to other classes of linearizations are that these analyses are very particular, since they make use of the highly specific structures of the considered linearizations, and that the new classes of linearizations are very restricted in the sense that they are highly structured and, so, are not robust under the unstructured perturbations coming from the backward errors of the algorithms. Thus, it is not clear if they are still linearizations of some matrix polynomial when they are perturbed, and even less of what polynomial they could be linearizations.

In order to overcome these obstacles, we introduce in this paper two new families of strong linearizations of general matrix polynomials---square or rectangular, regular or singular---whose minimal indices are related to those of the matrix polynomial via constant uniform shifts. We call these families the {\em strong block minimal bases pencils}, and a subfamily of it the {\em block Kronecker pencils}. Strong block minimal bases pencils are defined in an abstract way in terms of the classical concept of {\em dual minimal bases} \cite{doi:10.1137/0313029}. This allows us to prove that they are always strong linearizations of easily described matrix polynomials in a straightforward and general way and that simple relationships exist between their minimal indices and those of the matrix polynomial. These properties are inherited by the block Kronecker pencils, which include---modulo permutations---all of the Fiedler and proper generalized Fiedler pencils as very particular cases (see the extended version of this paper \cite[Section 4]{Dopico:2016:BlockKronecker} and \cite{budopereu-2016}), and which have the property of being easily constructed in terms of the polynomial coefficients.

Strong block minimal bases pencils have, in practice, only one structural feature, that is the presence of a zero block, since the other ingredients of their definition are polynomial minimal bases and ``generically'' all matrix polynomials of proper sizes are minimal bases \cite{vddop-robust-2016}. So, the class of strong block minimal bases pencils is robust under perturbations that preserve that zero block and, in addition, it is easy to describe the matrix polynomials of which they are linearizations. These properties enable us to perform a global backward error analysis of PEPs and CPEs solved via block Kronecker pencils, because arbitrary perturbations of these pencils lead, after some manipulations, to other strong block minimal bases pencils with similar properties. This error analysis has the following novel properties: (1) it is valid for perturbations with finite norms, in contrast to previous analyses which are valid only to first order; (2) it delivers precise bounds, in contrast to other analyses which only provide vague big-O bounds; (3) it is valid simultaneously for a very large class of linearizations, in contrast to other analyses that are specific for particular linearizations; and (4) it may be generalized to other families of strong block minimal bases pencils. As a corollary, this analysis solves the open problem of proving that all Fiedler and proper generalized Fiedler pencils yield computed complete eigenstructures of matrix polynomials that enjoy perfect structured backward stability from the polynomial point of view.

We emphasize that this backward error analysis does not imply that the eigenvalues and/or minimal indices of the matrix polynomial are accurately computed, since they are intrinsically ill-conditioned, or even ill-posed, when the eigenvalues are close to be multiple or the minimal indices are not generic \cite{edelelmkags1,edelelmkags2}. However, note that our results guarantee that if a backward stable stratification-enhanced staircase algorithm \cite{edelelmkags2} is used on a block Kronecker pencil, then, although the computed complete eigenstructure may be quite different from the exact one, it always corresponds (after a fixed constant shift of the minimal indices) to the exact complete eigenstructure of a nearby matrix polynomial.

The paper is organized as follows. Section \ref{sec:basic} presents a summary of basic concepts. In Section \ref{sec:minlinearizations}, the strong block minimal bases pencils are introduced and their properties are established. Section  \ref{sec:linearization} gives the definition of block Kronecker pencils and studies their properties. The global backward error analysis of complete polynomial eigenproblems solved by means of block Kronecker pencils is the subject of Section \ref{sec:expansion}. Some conclusions and lines of future research are discussed in Section \ref{sec:conclusions}. Finally, the Appendices present long technical proofs of some results needed in the paper. For brevity, this paper does not contain recovery procedures of eigenvectors and minimal bases of a matrix polynomial from those of its strong block minimal bases pencils or of its block Kronecker pencils. These results can be found in \cite[Section 7]{Dopico:2016:BlockKronecker}.

\section{Basic concepts, auxiliary results, and notation}\label{sec:basic}
Throughout the paper we use the following notation.
Given an arbitrary field $\mathbb{F}$, we denote by $\FF[\lambda]$ the ring of polynomials in the variable $\lambda$ with coefficients in $\FF$ and by $\FF(\lambda)$ the field of rational functions with coefficients in $\FF$.
The set of $m\times n$ matrices with entries in $\FF[\lambda]$ is denoted by $\mathbb{F}[\lambda]^{m\times n}$ and is also called the set of $m\times n$ {\em matrix polynomials}. In this context, row or column \emph{vector polynomials} are just matrix polynomials with $m=1$ or $n=1$. $\FF(\lambda)^{m\times n}$ denotes the set of $m\times n$ rational matrices. Given two matrices $A$ and $B$, $A\oplus B$ denotes their direct sum, i.e., $A\oplus B = \mbox{diag}(A,B)$, and $A\otimes B$ denotes their Kronecker product \cite{Horn}. The algebraic closure of $\FF$ is denoted by $\overline{\FF}$. The results in Section \ref{sec:expansion} and Subsection \ref{subsec:norms} assume that $\FF = \mathbb{R}$ or $\FF = \mathbb{C}$, while the rest of results remain valid in any field.

A matrix polynomial $P(\lambda)\in \mathbb{F}[\lambda]^{m\times n}$ is said to have \emph{grade} $d$ if it is written as
\begin{equation}\label{eq:polynomial}
P(\lambda) = P_d\lambda^d + \cdots + P_1\lambda+P_0, \quad \mbox{with }P_0,\hdots,P_d\in\FF^{m\times n},
\end{equation}
where any of the coefficient matrices $P_k$, including $P_d$, may be the zero matrix. As usual, the \emph{degree} of $P(\lambda)$, denoted by $\deg(P)$, is the maximum integer $k$ such that $P_k$ is a nonzero matrix. Thus, the degree of $P(\la)$ is fixed while its grade $d$ is a choice that must satisfy $d \geq \deg(P)$. The concept of grade has been used previously in \cite{de2014spectral,MMMMMoebius} and is convenient when the degree of a polynomial is not known in advance. Throughout this paper when the grade of $P(\la)$ is not explicitly stated, we consider its grade equal to its degree.
A matrix polynomial of grade $1$ is called a \emph{matrix pencil}.

For any $d\geq \deg(P)$ the \emph{$d$-reversal  matrix polynomial} of $P(\lambda)$ is defined as
\[
\rev_d P(\lambda) := \lambda^d P(\lambda^{-1}).
\]
Observe that if $P(\la)$ is assumed to have grade $d$, then it is assumed that $\rev_d P(\lambda)$ has also grade $d$, but that the degree of $\rev_d P(\lambda)$ may be different than the degree of $P(\la)$, even in the case $d = \deg(P)$.

We define the \emph{rank} of a matrix polynomial $P(\la) \in \FF[\la]^{m \times n}$ as its rank over the field $\FF(\la)$, i.e., as the size of the largest non-identically zero minor of $P(\lambda)$ \cite{gantmacher1960theory} and is denoted by $\rank(P)$. This is also called the ``normal rank'' of $P(\la)$, but we avoid to use this name for brevity. Note that expressions such as $\rank(P (\la_0))$ denote the rank of the constant matrix $P (\la_0) \in \overline{\FF}^{m \times n}$, i.e., of the polynomial evaluated at $\lambda_0 \in \overline{\FF}$. We will say that $P (\la_0)$ has full row (resp. column) rank if $\rank P (\la_0) = m$ (resp. $\rank P (\la_0) = n$). Observe that if the constant matrix $P (\la_0)$ has full row (resp. column) rank, then also the matrix polynomial $P (\la)$ has full row (resp. column) rank.

A key distinction for matrix polynomials is between regular and singular matrix polynomials. A matrix polynomial $P(\lambda)$ is said to be \emph{regular} if $P(\lambda)$ is square (that is, $m = n$) and $\det P(\lambda)$ is not the identically zero polynomial. Otherwise, $P(\lambda)$ is said to be \emph{singular} (note that this includes all rectangular matrix polynomials $m\neq n$). We refer the reader to \cite[Section 2]{de2014spectral} for the precise definitions of the spectral and the singular structures of a matrix polynomial, as well as for other related concepts that are used in this paper. In addition, as in \cite{FFP2015}, the term {\em complete eigenstructure} of $P(\la)$ stands for the collection of all of the elementary divisors of $P(\la)$, both finite and infinite, and for the collection of all of its minimal indices, both left and right, i.e., for the union of the spectral and singular structures of $P(\la)$. In the next paragraph, we explain in detail the concepts of minimal bases and minimal indices, as they play an essential role in this paper.

If a matrix polynomial $P(\lambda)\in\FF[\lambda]^{m\times n}$ is singular, then it has non-trivial left and/or right {\em rational} null spaces:
\begin{equation} \label{eq:nullspaces}
\begin{split}
\mathcal{N}_\ell(P) & := \{y(\lambda)^T\in\mathbb{F}(\lambda)^{1\times m} \quad \mbox{such that} \quad y(\lambda)^TP(\lambda) = 0\},\\
\mathcal{N}_r(P) & := \{x(\lambda)\in\mathbb{F}(\lambda)^{n \times 1} \quad \mbox{such that} \quad P(\lambda)x(\lambda) = 0\}.
\end{split}
\end{equation}
These null spaces are particular examples of {\em rational} subspaces, i.e., subspaces over the field $\FF (\la)$ formed by $p$-tuplas whose entries are rational functions \cite{doi:10.1137/0313029}. It is not difficult to show that any rational subspace $\mathcal{V}$ has bases consisting entirely of vector polynomials. The \emph{order} of a vector polynomial basis of $\mathcal{V}$ is defined as the sum of the degrees of its vectors  \cite[Definition 2]{doi:10.1137/0313029}. Amongst all of the possible polynomial bases of $\mathcal{V}$, those with least order are called \emph{minimal bases} of $\mathcal{V}$ \cite[Definition 3]{doi:10.1137/0313029}. There are infinitely many minimal bases of $\mathcal{V}$, but the ordered list of degrees of the vector polynomials in any minimal basis of $\mathcal{V}$ is always the same \cite[Remark 4, p. 497]{doi:10.1137/0313029}. This list of degrees is called the list of {\em minimal indices} of $\mathcal{V}$. With these definitions at hand, the left (resp. right) minimal indices and bases of a matrix polynomial $P(\la)$ are defined as those of the rational subspace $\mathcal{N}_\ell(P)$ (resp. $\mathcal{N}_r(P)$).

The following definitions are useful when working with minimal bases in practice. The \emph{$i$th row degree} of a matrix polynomial $Q(\lambda)$ is the degree of the $i$th row of $Q(\lambda)$.

\begin{definition} \label{def:rowreduced}
Let $Q(\lambda)\in\mathbb{F}[\lambda]^{m\times n}$ be a matrix polynomial with row degrees $d_1,d_2,\hdots,d_m$. The {\em highest row degree coefficient matrix} of $Q(\lambda)$, denoted by $Q_h$, is the $m \times n$ constant matrix whose $j$th row is the coefficient of $\lambda^{d_j}$ in the $j$th row of $Q(\lambda)$, for $j=1,2,\hdots,m$. The matrix polynomial $Q(\lambda)$ is called {\em row reduced} if $Q_h$ has full row rank.
\end{definition}

Observe that $Q_h$ is equal to the leading coefficient $Q_d \ne 0$ in the expansion $Q(\la) = \sum_{i=0}^d Q_i \la^i$ if and only if all the row degrees of $Q(\la)$ are equal to $d$.

Theorem \ref{thm:minimal_basis} is the most useful characterization of minimal bases in practice. This classical result was proved in \cite[Main Theorem-2, p. 495]{doi:10.1137/0313029}, where is stated in abstract terms. The statement we present can be found in  \cite[Theorem 2.14]{FFP2015}.

\begin{theorem}
\label{thm:minimal_basis}
The rows of a matrix polynomial $Q(\lambda)\in\FF[\lambda]^{m\times n}$ are a minimal basis of the rational subspace they span if and only if $Q(\lambda_0) \in \overline{\FF}^{m \times n}$ has full row rank for all $\lambda_0 \in \overline{\FF}$ and $Q(\la)$ is row reduced.
\end{theorem}

\begin{remark} {\rm Most of the minimal bases appearing in this work are arranged as the rows of a matrix. Therefore, throughout the paper---and with a slight abuse of notation---we say that an $m\times n$ matrix polynomial (with $m < n$) is a minimal basis if its rows form a minimal basis of the rational subspace they span.

Definition \ref{def:rowreduced} and Theorem \ref{thm:minimal_basis} admit obvious extensions ``for columns'', which are used occasionally in this paper.
}
\end{remark}

\medskip

Corollary \ref{cor:minbasesKron} is a consequence of Theorem \ref{thm:minimal_basis} and the property $\rank (A\otimes B) = \rank (A) \, \rank (B)$ \cite[Theorem 4.2.15]{Horn}. The simple proof is omitted.

\begin{corollary} \label{cor:minbasesKron} If a matrix polynomial $Q(\lambda)$ is a minimal basis and $I_p$ is the $p\times p$ identity matrix, then $Q(\lambda) \otimes I_p$ is also a minimal basis.
\end{corollary}

The concept of \emph{dual minimal bases} is fundamental in this paper and is introduced in Definition \ref{def:dualminimalbases}.
\begin{definition} \label{def:dualminimalbases}
Two matrix polynomials $L(\lambda)\in\FF[\lambda]^{m_1\times n}$ and $N(\lambda)\in\FF[\lambda]^{m_2\times n}$ are called \emph{dual minimal bases} if $L(\lambda)$ and $N(\lambda)$ are both minimal bases and they satisfy
$m_1+m_2 = n$ and $L(\lambda)N(\lambda)^T = 0$.

\end{definition}
The name ``dual minimal bases'' and its definition were introduced in \cite[Definition 2.10]{DDMVzigzag}, but their origins can be traced back to \cite{doi:10.1137/0313029}. We also use the expression ``$N(\la)$ is a minimal basis dual to $L(\la)$'', or vice versa, for referring to matrix polynomials $L(\la)$ and $N(\la)$ as those in Definition \ref{def:dualminimalbases}.

\begin{example} \label{ex-L-Lamb} {\rm We illustrate the concept of dual minimal bases with a simple example that is important in this paper. Consider the following matrix polynomials:
\begin{equation}
\label{eq:Lk}
L_k(\lambda):=\begin{bmatrix}
-1 & \lambda  \\
& -1 & \lambda \\
& & \ddots & \ddots \\
& & & -1 & \lambda  \\
\end{bmatrix}\in\mathbb{F}[\lambda]^{k\times(k+1)},
\end{equation}
and
\begin{equation}
  \label{eq:Lambda}
  \Lambda_k(\lambda)^T :=
\begin{bmatrix}
      \lambda^{k} & \cdots & \lambda & 1
\end{bmatrix} \in \FF[\lambda]^{1\times (k+1)},
\end{equation}
where here and throughout the paper we occasionally omit some, or all, of the zero entries of a matrix. Theorem \ref{thm:minimal_basis} guarantees that $L_k(\lambda)$ and $\Lambda_k(\lambda)^T$ are minimal bases. In addition, $L_k(\lambda)\Lambda_k(\lambda)=0$ holds. Therefore, $L_k(\lambda)$ and $\Lambda_k(\lambda)^T$ are dual minimal bases. From Corollary \ref{cor:minbasesKron} and the properties of the Kronecker product we get that $L_k(\lambda) \otimes I_p$ and $\Lambda_k(\lambda)^T \otimes I_p$ are also dual minimal bases.

The matrix $L_k (\la)$ is very well known since is a right singular block of the Kronecker Canonical Form of pencils \cite[Chapter XII]{gantmacher1960theory}. Also the {\em column} vector polynomial $\Lambda_k(\lambda)$ is very well known and plays an essential role, for instance, in the famous vector spaces of linearizations studied in \cite{Tisseur06symmetriclinearizations,DLP}.
}
\end{example}

\medskip

Theorem  \ref{thm:minbasesdegreeseq} establishes properties of minimal bases whose row degrees are all equal. These are the minimal bases of interest in this work. The proof of Theorem \ref{thm:minbasesdegreeseq} is omitted since follows from results on row-wise reversals of minimal bases \cite{DDMsingular,MMMMMoebius}. For a simpler proof based on Theorem \ref{thm:minimal_basis}, see the extended version of this paper \cite{Dopico:2016:BlockKronecker}.

\begin{theorem} \label{thm:minbasesdegreeseq}
\begin{enumerate}
\item[\rm (a)] Let $K(\la)$ be a minimal basis whose row degrees are all equal to $j$. Then $\rev_j K(\la)$ is also a minimal basis whose row degrees are all equal to $j$.
\item[\rm (b)] Let $K(\la)$ and $N(\la)$ be dual minimal bases. If the row degrees of $K(\la)$ are all equal to $j$ and the row degrees of $N(\la)$ are all equal to $\ell$, then $\rev_j K(\la)$ and $\rev_\ell N(\la)$ are also dual minimal bases.
\end{enumerate}
\end{theorem}

\begin{example} \label{ex-rev-L-Lamb} {\rm Theorem \ref{thm:minbasesdegreeseq}(b) can be applied to the dual minimal bases $L_k (\la)$ and $\Lambda_k (\la)^T$ in Example \ref{ex-L-Lamb} to prove that
\[
 \rev_1 L_k(\lambda) =
\begin{bmatrix}
-\lambda & 1  \\
& -\lambda & 1 \\
& & \ddots & \ddots \\
& & & -\lambda & 1  \\
\end{bmatrix}\in\mathbb{F}[\lambda]^{k\times(k+1)}
 \]
and
\[
\rev_k \Lambda_k (\lambda)^T =
\begin{bmatrix}
1 & \lambda & \cdots & \lambda^k
\end{bmatrix}\in\mathbb{F}[\lambda]^{1\times(k+1)}
\]
are also dual minimal bases. This fact follows also directly from Theorem \ref{thm:minimal_basis} and matrix multiplication.
}
\end{example}

\medskip

Lemma \ref{lemma:Embedding} states that any matrix polynomial $Q(\la)$ such that $Q(\la_0)$ has full row rank for all $\lambda_0\in \overline{\FF}$ can be completed into a {\em unimodular matrix polynomial}, i.e., a matrix polynomial with nonzero constant determinant. This is an old result that can be traced back at least to \cite{kailath1980linear} (a very simple proof appears in \cite[Lemma 2.16(b)]{FFP2015}). Efficient algorithms for computing such completions can be found in \cite{Embedding}.

\begin{lemma}\label{lemma:Embedding}
Let $Q(\lambda)$ be a matrix polynomial over a field $\FF$. If $Q(\la_0)$ has full row rank for all $\lambda_0\in \overline{\FF}$, then there exists a matrix polynomial $\widetilde{Q}(\lambda)$ such that
\[
\widehat{Q}(\lambda) = \begin{bmatrix}
Q(\lambda) \\   \widetilde{Q}(\lambda)
\end{bmatrix}
\]
is unimodular.
\end{lemma}

Lemma \ref{lemma:Embedding} can be applied, in particular, when $Q(\la)$ is a minimal basis, as a consequence of Theorem \ref{thm:minimal_basis}. Moreover, Lemma \ref{lemma:Embedding} can be extended to Theorem \ref{thm:Embedding}, which is one of the main tools employed in Section \ref{sec:minlinearizations}. Observe that Theorem \ref{thm:Embedding} can be applied, in particular, when $L(\la)$ and $N(\la)$ are dual minimal bases.

\begin{theorem}\label{thm:Embedding}
Let $L(\lambda)\in \FF[\la]^{m_1 \times n}$ and $N(\lambda)\in \FF[\la]^{m_2 \times n}$ be matrix polynomials such that $m_1 + m_2 = n$, $L(\la_0)$ and $N(\la_0)$ have both full row rank for all $\lambda_0\in \overline{\FF}$, and $L(\la) N(\la)^T = 0$. Then, there exists a unimodular matrix polynomial $U(\lambda)\in \FF[\la]^{n \times n}$ such that
\[
U(\lambda) =
\begin{bmatrix}
L(\lambda) \\ \widehat{L}(\lambda)
\end{bmatrix}
\quad \mbox{and} \quad
U(\lambda)^{-1}=
\begin{bmatrix}
\widehat{N}(\lambda)^T & N(\lambda)^T
\end{bmatrix}.
\]
\end{theorem}

\begin{proof}
By Lemma \ref{lemma:Embedding}, there exist unimodular embeddings
\[
\begin{bmatrix}
L(\lambda) \\ Z_1(\lambda)
\end{bmatrix}
\quad \mbox{and} \quad
\begin{bmatrix}
Z_2(\lambda)^T & N(\lambda)^T
\end{bmatrix}.
\]
Since the product of two unimodular matrix polynomials is also unimodular, from
\[
\begin{bmatrix}
L(\lambda) \\ Z_1(\lambda)
\end{bmatrix}
\begin{bmatrix}
Z_2(\lambda)^T & N(\lambda)^T
\end{bmatrix} =
\begin{bmatrix}
L(\lambda)Z_2(\lambda)^T & 0 \\
Z_1(\lambda)Z_2(\lambda)^T & Z_1(\la) N(\lambda)^T
\end{bmatrix},
\]
it follows that $L(\lambda)Z_2(\lambda)^T \in \FF[\la]^{m_1 \times m_1}$ and $Z_1(\lambda)N(\lambda)^T \in \FF[\la]^{m_2 \times m_2}$ must also be unimodular matrix polynomials, as well as their inverses. Let us now consider the following unimodular matrix polynomials
\[
U(\lambda) =
\begin{bmatrix}
I_{m_1} & 0 \\ 0 & (Z_1(\lambda)N(\lambda)^T)^{-1}
\end{bmatrix}
\begin{bmatrix}
L(\lambda) \\ Z_1(\lambda)
\end{bmatrix}
\]
and
\[
V(\lambda) =
\begin{bmatrix}
Z_2(\lambda)^T & N(\lambda)^T
\end{bmatrix}
\begin{bmatrix}
(L(\lambda)Z_2(\lambda)^T)^{-1} & 0 \\ 0 & I_{m_2}
\end{bmatrix}
\begin{bmatrix}
I_{m_1} & 0 \\ -X(\lambda) & I_{m_2}
\end{bmatrix},
\]
where $X(\lambda)=(Z_1(\lambda)N(\lambda)^T)^{-1}Z_1(\lambda)
Z_2(\lambda)^T(L(\lambda)Z_2(\lambda)^T)^{-1}$.
The statement of the theorem then follows by verifying that $U(\lambda)V(\lambda)=I_n$.
\end{proof}

\medskip

\begin{example} \label{ex-embeddings} {\rm
We illustrate Theorem \ref{thm:Embedding} with a particular embedding of the dual minimal bases $L_k(\lambda)$ and $\Lambda_k(\lambda)^T$ introduced in Example \ref{ex-L-Lamb}. If $e_{k+1}$ is the last column of $I_{k+1}$, then it is easily verified that
\begin{equation*}\label{eq:Vk}
V_k(\lambda) =
\left[
  \begin{array}{c}
    L_k(\lambda) \\ \hline \! \!\phantom{\Big)}
    e_{k+1}^T
  \end{array}
\right]
=
\left[
  \begin{array}{ccccc}
    -1 & \lambda& & &   \\
    & -1 & \lambda & &  \\
    & & \ddots & \ddots&  \\
    & & & -1 & \lambda \\ \hline
    0 & \cdots  & \cdots  & 0 & 1
  \end{array}
\right]
\in\mathbb{F}[\lambda]^{(k+1)\times(k+1)}
\end{equation*}
is unimodular and that its inverse is
\begin{equation}
\label{eq:invVk}
V_k(\lambda)^{-1} =
\left[\begin{array}{ccccc|c}
 -1 & -\lambda & -\la^2 & \cdots & -\lambda^{k-1}  & \lambda^k \\
  & -1 & -\lambda &  \ddots & \vdots & \lambda^{k-1} \\
  &  & -1 &  \ddots & -\la^2 & \vdots \\
  & & & \ddots & -\la &\lambda^2 \\
  & & & & -1 & \la \\
  & & & & & 1
\end{array}\right]
\in\mathbb{F}[\lambda]^{(k+1)\times(k+1)}.
\end{equation}
Note that the last column of $V_k(\lambda)^{-1}$ is $\Lambda_k(\lambda)$. Therefore, $V_k(\la)$ is a particular instance of a matrix $U(\la)$ in Theorem \ref{thm:Embedding} for $L_k(\lambda)$ and $\Lambda_k(\lambda)^T$. Moreover, $V_k(\la)\otimes I_p$ is a particular instance of $U(\la)$ for the dual minimal bases $L_k(\lambda) \otimes I_p$ and $\Lambda_k(\lambda)^T \otimes I_p$ discussed also in Example \ref{ex-L-Lamb}.
}
\end{example}

\medskip

We now recall the definitions of linearization and strong linearization of a matrix polynomial, which are central in this paper. These definitions were introduced in \cite{GKL:1988:matrixpoly,LancasterBook} for regular matrix polynomials, and extended to the singular case in \cite{de2008sharp}. We refer the reader to \cite{de2014spectral} for a thorough treatment of these concepts and their properties.

\begin{definition}\label{def:linearization}
A matrix pencil $\lpol{\lambda}$ is a \emph{linearization} of a matrix polynomial $P(\lambda)$ of grade $d$ if for some $s \geq 0$ there exist two unimodular matrix polynomials $U(\lambda)$ and $V(\lambda)$ such that
  \begin{equation}\label{eq:linearization_relation}
    U(\lambda)\lpol{\lambda}V(\lambda)=
    \left[
      \begin{array}{cc}
        I_s& \\
        &P(\lambda)
      \end{array}
    \right]\>.
  \end{equation}
Furthermore, a linearization $\lpol{\lambda}$ is called a \emph{strong linearization} of $P(\la)$ if $\rev_1\lpol{\lambda}$ is a linearization of $\rev_d P(\lambda)$.
\end{definition}

The key property of any strong linearization $\mathcal{L}(\lambda)$ of a matrix polynomial $P(\lambda)$ is that $\mathcal{L}(\lambda)$ and $P(\lambda)$ share the same finite and infinite elementary divisors \cite[Theorem 4.1]{de2014spectral}. However, Definition \ref{def:linearization} only guarantees that the number of left (resp. right) minimal indices of $\mathcal{L}(\lambda)$ is equal to the number of left (resp. right) minimal indices of $P(\lambda)$. In fact, except by these constraints on the numbers, $\lpol{\la}$ may have any set of right and left minimal indices \cite[Theorem 4.11]{de2014spectral}. Therefore, in the case of singular matrix polynomials, one needs to consider {\em strong linearizations with the additional property} that their minimal indices allow us to recover the minimal indices of the polynomial via some simple rule. In addition, {\em such rule should be robust under perturbations}, in order to be reliable in numerical computations affected by rounding errors, since minimal indices of matrix polynomials may vary wildly under perturbations \cite{edelelmkags1,edelelmkags2,johkagsvdstrat}. These questions about recovery rules of minimal indices are carefully studied throughout this paper.

Lemma \ref{lemma:antitriglin} is a very simple result that allows us to easily recognize linearizations in certain situations which are of interest in this work.
\begin{lemma} \label{lemma:antitriglin} Let $P(\la)$ be an $m\times n$ matrix polynomial and $\lpol{\lambda}$ be a matrix pencil. If there exist two unimodular matrix polynomials $\widetilde{U}(\lambda)$ and $\widetilde{V}(\lambda)$ such that
\begin{equation} \label{eq:lintriagrelation}
\widetilde{U}(\lambda)\mathcal{L}(\lambda)\widetilde{V}(\lambda) =
\begin{bmatrix}
Z(\lambda) & X(\lambda) & I_{t} \\
Y(\lambda) & P(\lambda) & 0 \\
I_{s} & 0 & 0
\end{bmatrix},
\end{equation}
for some $s \geq 0$ and $t\geq 0$ and for some matrix polynomials $X(\la)$, $Y(\la)$, and $Z(\la)$, then $\lpol{\lambda}$ is a linearization of $P(\la)$.
\end{lemma}

\begin{proof} Define the unimodular matrix polynomials
\[
R(\lambda) = \begin{bmatrix}
I_{t} & 0 & -Z(\lambda) \\
0 & 0 & I_{s} \\
0 & I_m & -Y(\lambda)
\end{bmatrix} \>, \quad
S(\lambda) =
\begin{bmatrix}
0 & I_{s} & 0 \\
0 & 0 & I_n \\
I_{t} & 0 & -X(\lambda)
\end{bmatrix}.
\]
Then equation \eqref{eq:lintriagrelation} implies that
$
R(\lambda) \widetilde{U}(\lambda)\mathcal{L}(\lambda)\widetilde{V}(\lambda) S(\lambda) = \mbox{diag} (I_{t},I_{s},P(\la))
$. This proves that $\mathcal{L}(\lambda)$ is a linearization of $P(\la)$.
\end{proof}

\subsection{Norms of matrix polynomials and their submultiplicative properties} \label{subsec:norms} The study of perturbations and backward errors in Section \ref{sec:expansion} requires the use of norms of matrix polynomials. We have chosen the simple norm in Definition \ref{def:norm}. In this section the polynomials are assumed to have {\em real or complex coefficients}, i.e., $\FF = \mathbb{R}$ or $\FF = \mathbb{C}$. We refer the reader to \cite{stewartsunbook} for the definitions and properties of the Frobenius norm, $\|\cdot\|_F$, and the spectral norm, $\|\cdot\|_2$, of constant matrices.

\begin{definition} \label{def:norm} Let $P(\la) = \sum_{i=0}^d P_i \la^i \in \FF [\la]^{m\times n}$. Then the Frobenius norm of $P(\lambda)$ is
\[
\|P(\la)\|_F := \sqrt{\sum_{i=0}^d \|P_i\|_F^2} \, .
\]
\end{definition}

Obviously $\|P(\la)\|_F$ defines a norm on the vector space of matrix polynomials with arbitrary grade and fixed size $m\times n$. In fact, Definition \ref{def:norm} defines a {\em family of norms}, because we have a different vector space, and, so, a different norm for each particular selection of size $m \times n$. This is important when considering the norm of the product $P(\la) Q(\la)$ of two polynomials $P(\la)$ and $Q(\la)$, since the sizes of the two factors and the product are, in general, different. In this context, it is also important to realize that the value of $\|P(\la)\|_F$ is independent of the grade chosen for $P(\la)$. This property allows us to work with $\|P(\la)\|_F$ without specifying the grade of $P(\la)$.

It is easy to construct examples that show that the norm $\|P(\la)\|_F$ is not submultiplicative, i.e., $\|P(\la) \, Q(\la)\|_F \nleq \|P (\la)\|_F \, \|Q(\la)\|_F$ in general \cite{Dopico:2016:BlockKronecker}. Therefore, since in Section \ref{sec:expansion} we need to bound the norms of certain products of matrix polynomials, we present Lemma \ref{lemma:normsproducts}, whose elementary but somewhat long proof is omitted. The interested reader can find the proof in the extended version of this paper \cite{Dopico:2016:BlockKronecker}.

\begin{lemma} \label{lemma:normsproducts} Let $P(\la) = \sum_{i=0}^d P_i \la^i$, let $Q(\la) = \sum_{i=0}^t Q_i \la^i$, and let $\Lambda_k (\la)^T$ be the vector polynomial defined in \eqref{eq:Lambda}. Then the following inequalities hold:
\begin{enumerate}
\item[\rm (a)] $\displaystyle \|P (\la) \, Q(\la)\|_F \leq \sqrt{d+1} \cdot \sqrt{\sum_{i=0}^d \|P_i\|_2 ^2}  \cdot \|Q(\la)\|_F$  ,

\item[\rm (b)] $\displaystyle \|P(\la) \, Q(\la)\|_F \leq \sqrt{t+1} \cdot \|P(\la)\|_F \cdot \sqrt{\sum_{i=0}^t \|Q_i\|_2 ^2}$ ,

\item[\rm (c)] $\|P (\la) \, Q(\la)\|_F \leq \, \min\{\sqrt{d+1} , \sqrt{t+1} \}\,  \|P (\la)\|_F \, \|Q(\la)\|_F$ ,

\item[\rm (d)] $\|P(\la) \, (\Lambda_k (\la) \otimes I_p) \|_F \leq \, \min\{\sqrt{d+1}, \sqrt{k+1} \} \, \|P(\la) \|_F$,

\item[\rm (e)] $\|(\Lambda_k (\la)^T \otimes I_p) \, Q(\la)\|_F \leq \, \min\{\sqrt{t+1}, \sqrt{k+1} \}  \, \|Q(\la) \|_F$,
\end{enumerate}
where we assume that all the products are defined.
\end{lemma}

\medskip

Finally, in Section \ref{sec:expansion} we need to consider pairs of matrices $(C,D)$ where $C$ and $D$ may have different sizes. Therefore, $(C,D)$ cannot be considered as a matrix pencil. For these pairs, we introduce the corresponding Frobenius norm as:
\begin{equation} \label{eq:defnormpair}
\|(C,D)\|_F := \sqrt{\|C\|_F^2+\|D\|_F^2}.
\end{equation}

\section{Block minimal bases linearizations} \label{sec:minlinearizations} The linearizations considered in this work in Sections \ref{sec:linearization} and \ref{sec:expansion} are particular cases of the new pencils introduced in Definition \ref{def:minlinearizations}. These pencils include all the families of Fiedler-like linearizations of matrix polynomials, which have received considerable attention recently. For more information on this, see the extended version of this paper \cite[Section 4]{Dopico:2016:BlockKronecker} and \cite{budopereu-2016}. Therefore, Definition \ref{def:minlinearizations} seems to be a key concept that unifies and simplifies the theory of many of the linearizations existing in the literature. In this paper, the linearizations in Definition \ref{def:minlinearizations} are of interest because they are generic and robust under perturbations that preserve the zero block, as we discuss at the end of this section.

\begin{definition} \label{def:minlinearizations} A matrix pencil
\begin{equation} \label{eq:minbaspencil}
\EL{} = \begin{bmatrix} M(\la) & K_2 (\la)^T \\ K_1 (\la) & 0\end{bmatrix}
\end{equation}
is called a {\em block minimal bases pencil} if $K_1 (\la)$ and $K_2(\la)$ are both minimal bases.
If, in addition, the row degrees of $K_1 (\la)$ are all equal to $1$, the row degrees of $K_2 (\la)$ are all equal to $1$, the row degrees of a minimal basis dual to $K_1 (\la)$ are all equal, and the row degrees of a minimal basis dual to $K_2 (\la)$ are all equal, then $\EL{}$ is called a {\em strong block minimal bases pencil}.
\end{definition}

\begin{remark} {\rm
Observe in Definition \ref{def:minlinearizations} that the row degrees of any minimal basis dual to $K_1 (\la)$ are always the same, up to permutations, since they are the right minimal indices of $K_1 (\la)$. The same holds for $K_2 (\la)$. Therefore, there are no ambiguities in the definition of strong block minimal bases pencils with respect to the selection of the minimal bases dual to $K_1 (\la)$ and $K_2 (\la)$.}
\end{remark}

\medskip

Next theorem reveals that (strong) block minimal bases pencils are (strong) linearizations of certain matrix polynomials.

\begin{theorem} \label{thm:blockminlin}  Let $K_1 (\la)$ and $N_1 (\la)$ be a pair of dual minimal bases, and let $K_2 (\la)$ and $N_2 (\la)$ be another pair of dual minimal bases. Consider the matrix polynomial
\begin{equation} \label{eq:Qpolinminbaslin}
Q(\la) := N_2(\la) M(\la) N_1(\la)^T,
\end{equation}
and the block minimal bases pencil $\EL{}$ in \eqref{eq:minbaspencil}. Then:
\begin{enumerate}
\item[\rm (a)] $\EL{}$ is a linearization of $Q(\la)$.
\item[\rm (b)] If $\EL{}$ is a strong block minimal bases pencil, then $\EL{}$ is a strong linearization of $Q(\la)$, considered as a polynomial with grade $1 + \deg(N_1 (\la)) + \deg(N_2 (\la))$.
\end{enumerate}
\end{theorem}

\begin{proof} (a) According to Theorem \ref{thm:Embedding}, for $i=1,2$, there exist  unimodular matrix polynomials such that
\begin{equation} \label{eq:twounimodembed}
U_i(\lambda) =
\begin{bmatrix}
K_i(\lambda) \\ \widehat{K}_i(\lambda)
\end{bmatrix}
\quad \mbox{and} \quad
U_i(\lambda)^{-1}=
\begin{bmatrix}
\widehat{N}_i(\lambda)^T & N_i(\lambda)^T
\end{bmatrix}.
\end{equation}
Note that if $m_i$ is the number of rows of $K_i(\lambda)$, for $i=1,2$, then \eqref{eq:twounimodembed} implies $K_i(\lambda) \widehat{N}_i(\lambda)^T = I_{m_i}$ and $K_i(\lambda)  N_i(\lambda)^T = 0$. Keep in mind that these equalities are used in subsequent matrix products. Next, consider the unimodular matrices $U_2(\lambda)^{-T} \oplus I_{m_1}$ and $U_1(\lambda)^{-1} \oplus I_{m_2}$, and form the following matrix product:
\begin{align}
& (U_2(\lambda)^{-T} \oplus I_{m_1}) \, \EL{} \, (U_1(\lambda)^{-1} \oplus I_{m_2})  \nonumber \\
& \phantom{aaaaaaaaaa} =
\begin{bmatrix}
\widehat{N}_2(\lambda) & 0 \\ N_2(\lambda) & 0 \\0 & I_{m_1}
\end{bmatrix} \,
\begin{bmatrix} M(\la) & K_2 (\la)^T \\ K_1 (\la) & 0\end{bmatrix} \,
\begin{bmatrix}
\widehat{N}_1(\lambda)^T & N_1(\lambda)^T & 0 \nonumber \\
0 & 0 & I_{m_2}
\end{bmatrix} \\
&\phantom{aaaaaaaaaa} =  \begin{bmatrix}
Z(\lambda) & X(\lambda) & I_{m_2} \\
Y(\lambda) & Q(\lambda) & 0 \\
I_{m_1} & 0 & 0
\end{bmatrix}, \label{eq:XYZminlin}
\end{align}
where the expressions of the matrix polynomials $ X(\lambda), Y(\lambda)$, and  $Z(\lambda)$ are not of specific interest in this proof. Equation \eqref{eq:XYZminlin} and  Lemma \ref{lemma:antitriglin} prove that $\EL{}$ is a linearization of $Q(\la)$.

\smallskip

(b) Let us denote for brevity $\ell_1 = \deg(N_1 (\la))$ and $\ell_2 = \deg(N_2 (\la))$. Since $\EL{}$ is a strong block minimal bases pencil, Theorem \ref{thm:minbasesdegreeseq}(b) guarantees that $\rev_1 K_1(\la)$ and $\rev_{\ell_1} N_1(\la)$ are dual minimal bases, as well as $\rev_1 K_2(\la)$ and $\rev_{\ell_2} N_2(\la)$. Therefore, $$\rev_1 \EL{} = \begin{bmatrix}\rev_1 M(\la) & \rev_1 K_2 (\la)^T \\ \rev_1 K_1 (\la) & 0 \end{bmatrix}$$ is also a block minimal bases pencil and  Theorem \ref{thm:blockminlin}(a) (just proved) implies that $\rev_1 \EL{}$ is a linearization of
\begin{align*}
(\rev_{\ell_2} N_2(\la)) \, (\rev_1 M(\la)) \, (\rev_{\ell_1} N_1(\la))^T
& =
\la^{\ell_2} N_2 \left( \la^{-1} \right) \, \la  \, M \left( \la^{-1} \right) \, \la^{\ell_1} N_1 \left( \la^{-1} \right)^T \\
\la^{1+\ell_1 + \ell_2} Q(\la^{-1}) = \rev_{1+\ell_1 + \ell_2} Q(\la),
\end{align*}
proving part (b).
\end{proof}

\begin{remark} {\rm Given a {\em strong} block minimal bases pencil $\EL{}$, there are infinitely many minimal bases $N_1(\la)$ and $N_2(\la)$ dual to $K_1 (\la)$ and $K_2 (\la)$, respectively. Therefore, the matrix polynomial $Q (\la)$ is not defined uniquely by $\EL{}$. This is connected to the following remark: the standard scenario when using linearizations is that the matrix polynomial $Q(\la)$ is given and one wants to construct a linearization of $Q(\la)$ as easily as possible, but Theorem \ref{thm:blockminlin} seems to operate in the opposite way. However, if $Q(\la)$ is given and $N_1 (\la)$ and $N_2 (\la)$ are fixed, then \eqref{eq:Qpolinminbaslin} can be viewed as a linear equation for the unknown pencil $M(\la)$. It is possible to prove that this equation is always consistent, as a consequence of the properties of the minimal bases $N_1 (\la)$ and $N_2 (\la)$. Despite its consistency, the equation \eqref{eq:Qpolinminbaslin} may be very difficult to solve for arbitrary minimal bases $N_1 (\la)$ and $N_2 (\la)$. We will see in Section \ref{sec:linearization} that for certain particular choices of $N_1 (\la)$ and $N_2 (\la)$ it is very easy to characterize all possible solutions $M(\la)$ and to define, in this way, a new wide class of linearizations easily constructible from $Q(\la)$. This new class includes, among many others, all Fiedler linearizations, up to permutations, of square or rectangular polynomials \cite{p832,DDM2010Fiedler,DDM2012rectangular,q077}.
}
\end{remark}

\smallskip

\begin{remark} \label{rem:empty} {\rm We include in Definition \ref{def:minlinearizations} the cases in which either $K_1(\la)$ or $K_2 (\la)$ is an empty matrix. This means that $\EL{}$ is either a $1 \times 2$ or a $2\times 1$ block matrix, and, so, the zero block is not present. All of the proofs in this paper remain valid in these border cases with the following convention: if $K_1(\la)$ (resp. $K_2(\la)$) is an empty matrix, then $N_1 (\la) = I_s$ (resp. $N_2 (\la) = I_s$), where $s$ is the number of colums (resp. rows) of $M(\la)$.}
\end{remark}

\smallskip

Next, we investigate, for strong block minimal bases pencils, the relationship of the minimal indices of $Q(\la)$ in \eqref{eq:Qpolinminbaslin} with those of its strong linearization $\EL{}$ in \eqref{eq:minbaspencil}. This result is a corollary of a technical lemma presented in Appendix \ref{sec:appendixminbases}.

\begin{theorem} \label{thm:indicesminbaseslin}
Let $\EL{}$ be a strong block minimal bases pencil as in \eqref{eq:minbaspencil}, let $N_1(\la)$ be a minimal basis dual to $K_1 (\la)$, let $N_2(\la)$ be a minimal basis dual to $K_2 (\la)$, and let $Q(\la)$ be the matrix polynomial defined in \eqref{eq:Qpolinminbaslin}. Then the following hold:
\begin{enumerate}
\item[\rm (a)] If $0 \leq \e_1 \leq \e_2 \leq \cdots \leq \e_p$ are the right minimal indices of $Q(\la)$, then
\[
\e_1 + \deg(N_1(\la)) \leq \e_2 + \deg(N_1(\la)) \leq \cdots \leq \e_p +  \deg(N_1(\la))
\]
are the right minimal indices of $\EL{}$.
\item[\rm (b)] If $0 \leq \eta_1 \leq \eta_2 \leq \cdots \leq \eta_q$ are the left minimal indices of $Q(\la)$, then
\[
\eta_1 + \deg(N_2(\la)) \leq \eta_2 + \deg(N_2(\la)) \leq \cdots \leq \eta_q +  \deg(N_2(\la))
\]
are the left minimal indices of $\EL{}$.
\end{enumerate}
\end{theorem}

\begin{proof} Part (a) follows immediately from Lemma \ref{lemm:techindminbaslin}(b) and equation \eqref{eq:degreeshift1}. Part (b) follows simply from applying part (a) to $\EL{}^T$ and $Q(\la)^T$ after taking into account that: (i) $\EL{}^T$ is also a strong block minimal bases pencil with the roles of $(K_1(\la), N_1(\la))$ and $(K_2(\la),N_2 (\la))$ interchanged, (ii) so $\EL{}^T$ is a strong linearization of $Q(\la)^T$, and (iii) for any matrix polynomial its left minimal indices are the right minimal indices of its transpose.
\end{proof}

In order to concisely refer to results like those in Theorem \ref{thm:indicesminbaseslin} we use in this paper expressions as ``the right minimal indices of $\EL{}$ are those of $Q(\la)$ shifted by $\deg (N_1(\la))$'', whose rigorous meaning is precisely the statement of Theorem \ref{thm:indicesminbaseslin}(a).

Finally, we emphasize that {\em ``generically'' any pencil partitioned into $2\times 2$ blocks with a $(2,2)$-zero block as in \eqref{eq:minbaspencil} is a strong block minimal bases pencil} if the sizes of the blocks are adequate. This follows from the recent results in \cite[Section 5]{vddop-robust-2016} when the pencils $K_1(\la)$ and $K_2(\la)$ have both more columns than rows and the excess number of columns is a divisor of the number of rows. This makes the pencils in Definition \ref{def:minlinearizations} a very large family of strong linearizations very convenient for analyzing perturbations of the highly structured strong linearizations used in computational practice, as for instance the Frobenious companion forms \cite{LancasterBook}, because although the perturbations destroy the particular structures, as long as they are sufficiently small and the $(2,2)$-zero block is preserved, the perturbed linearization is still a strong linearization (in fact, a strong block minimal bases pencil) of a nearby polynomial obtained by \eqref{eq:Qpolinminbaslin} applied to the perturbed pencil. Note that the $(2,2)$-zero block is not present in the border cases discussed in Remark \ref{rem:empty}. These ideas are fundamental for the error analysis in Section \ref{sec:expansion}.

\section{Block Kronecker linearizations}\label{sec:linearization}
In this section we study those strong block minimal bases pencils with off-diagonal blocks equal to the pencils in Example \ref{ex-L-Lamb}. They are called {\em block Kronecker pencils}. Thus, these pencils have the structure in \eqref{eq:minbaspencil} with $K_1 (\la) = L_{\e} (\la) \otimes I_n$ and $K_2 (\la) = L_{\eta} (\la) \otimes I_m$. Since, according to Example \ref{ex-L-Lamb}, $N_1(\la) = \Lambda_{\e} (\la)^T \otimes I_n$ and $N_2 (\la) = \Lambda_{\eta} (\la)^T \otimes I_m$ are minimal bases dual to these $K_1 (\la)$ and $K_2 (\la)$, respectively, most properties of block Kronecker pencils follow immediately from the general and simple theory in Section \ref{sec:minlinearizations}, for these particular $K_i (\la)$ and $N_i (\la)$, $i=1,2$. Nonetheless, we emphasize that block Kronecker pencils have an essential advantage over general strong block minimal bases pencils that is key in applications: given a matrix polynomial $P(\la)$ it is very easy to characterize an infinite set of $(1,1)$-blocks $M (\la)$ that make $\EL{}$ in \eqref{eq:minbaspencil} a strong linearization of $P(\la)$. Moreover, as we discuss below, block Kronecker pencils include, as particular cases, the classical Frobenius companion forms and the Fiedler pencils \cite{DDM2010Fiedler,DDM2012rectangular} modulo permutations. Block Kronecker pencils are formally introduced in Definition \ref{def:blockKronlin}.

\begin{definition} \label{def:blockKronlin} Let $L_k (\la)$ be the matrix pencil defined in \eqref{eq:Lk} and let $\la M_1 + M_0$ be an arbitrary pencil. Then any matrix pencil of the form
\begin{equation}
  \label{eq:linearization_general}
  \begin{array}{cl}
  \EL{}=
  \left[
    \begin{array}{c|c}
      \lambda M_1+M_0&L_{\eta}(\lambda)^{T}\otimes I_{\rowdim}\\\hline
      L_{\e}(\lambda)\otimes I_{\coldim}&0
      \end{array}
    \right]&
    \begin{array}{l}
      \left. \vphantom{L_{\mu}^{T}(\lambda)\otimes I_{\rowdim}} \right\} {\scriptstyle (\eta+1)\rowdim}\\
      \left. \vphantom{L_{\e}(\lambda)\otimes I_{\coldim}}\right\} {\scriptstyle\e\coldim}
    \end{array}\\
    \hphantom{\EL{}=}
    \begin{array}{cc}
      \underbrace{\hphantom{L_{\e}(\lambda)\otimes I_{\coldim}}}_{(\e+1)\coldim}&\underbrace{\hphantom{L_{\mu}^{T}(\lambda)\otimes I_{\rowdim}}}_{\eta\rowdim}
    \end{array}
  \end{array}
  \>,
\end{equation}
is called an $(\e,n,\eta,m)$-block Kronecker pencil or, simply, a {\em block Kronecker pencil}. The partition of $\EL{}$ into $2 \times 2$ blocks in \eqref{eq:linearization_general} is called the natural partition of a block Kronecker pencil.
\end{definition}

The name ``block Kronecker pencil'' is motivated by the fact that the anti-diagonal blocks of $\EL{}$ in \eqref{eq:linearization_general} are Kronecker products of singular blocks of the Kronecker canonical form of pencils \cite[Chapter XII]{gantmacher1960theory} with identity matrices.

Since block Kronecker pencils are particular cases of strong block minimal bases pencils, we obtain the following result for block Kronecker pencils as an immediate corollary of Theorems \ref{thm:blockminlin} and \ref{thm:indicesminbaseslin} and the results in Example \ref{ex-L-Lamb}.

\begin{theorem}\label{thm:strong}
Let $\mathcal{L}(\lambda)$ be an $(\e , n, \eta , m)$-block Kronecker pencil as in \eqref{eq:linearization_general}.
Then $\mathcal{L}(\lambda)$ is a strong linearization of the  matrix polynomial
\begin{equation}
\label{eq:condition}
Q(\la) := (\Lambda_\eta(\lambda)^T\otimes I_m)(\lambda M_1+M_0)(\Lambda_{\e}(\lambda)\otimes I_n) \in \FF[\lambda]^{m\times n}
\end{equation}
of  grade $\e+\eta+1$, the right minimal indices of $\EL{}$ are those of $Q(\la)$ shifted by $\e$, and the left minimal indices of $\EL{}$ are those of $Q(\la)$ shifted by $\eta$.
\end{theorem}

\begin{remark} \label{rem:explicitunimodular} {\rm Explicit unimodular matrices that transform any block Kronecker pencil as in \eqref{eq:linearization_general} into a block anti-triangular form \eqref{eq:XYZminlin} can be described via the matrices $V_k (\la)^{-1}$ in Example \ref{ex-embeddings}. In fact, an immediate corollary of Example \ref{ex-embeddings} and the block matrix multiplications yielding \eqref{eq:XYZminlin} in the proof of Theorem \ref{thm:blockminlin} is that
\begin{equation} \label{eq:explicitunimodular}
((V_\eta (\la)^{-T} \otimes I_m) \oplus I_{\e n}) \, \EL{} \,
((V_\e (\la)^{-1} \otimes I_n) \oplus I_{\eta m})
\end{equation}
has the block anti-triangular structure in \eqref{eq:XYZminlin}. This can also be checked via a direct multiplication, which proves in a simple way that block Kronecker pencils are linearizations of $Q(\la)$ as a consequence of Lemma \ref{lemma:antitriglin}. A similar approach can be used to prove that $\EL{}$ is a strong linearization of $Q(\la)$.
}
\end{remark}

\medskip

The most transparent examples of block Kronecker pencils are the classical first and second Frobenius companion forms of a matrix polynomial $P(\la) = \sum_{i=0}^{d} P_i \la^i \in \FF[\la]^{m \times n}$ \cite[Section 5.1]{de2014spectral}.  The first Frobenius companion form is just $\EL{}$ in \eqref{eq:linearization_general} with $M(\la) = [\la P_d + P_{d-1}, P_{d-2}, \ldots, P_0]$, $\e = d-1$, and $\eta =0$, while the second Frobenius companion form corresponds to $M(\la) = [\la P_d^T + P_{d-1}^T, P_{d-2}^T, \ldots, P_0^T]^T$, $\e = 0$, and $\eta =d-1$. Note that the application of Theorem \ref{thm:strong} in these two cases proves in a very simple way that the first and the second Frobenius companion forms are strong linearizations of $P(\la)$ with the well-known shifting relationships between the minimal indices (compare with the proofs in \cite[Section 5.1]{de2014spectral}).

It is also possible to prove with more effort that after performing some row and column permutations all Fiedler pencils of $P(\la) = \sum_{i=0}^{d} P_i \la^i \in \FF[\la]^{m \times n}$ \cite{DDM2010Fiedler,DDM2012rectangular} become block Kronecker pencils with the pencil $\la M_1 + M_0$  having a very simple structure that can be explicitly described in terms of the coefficients of $P(\la)$. This result can be found in the extended version of this paper \cite[Section 4]{Dopico:2016:BlockKronecker}, where it is proved that the only nonzero block entries of $\la M_1 + M_0$ are $\la P_d + P_{d-1}, P_{d-2}, \ldots, P_1, P_0$ distributed along what is called a ``staircase pattern'' \cite[Section 5]{condnumbFiedler} and \cite{eastman2}. Once this is established, Theorem \ref{thm:strong} proves again in a very simple way that all Fiedler pencils are strong linearizations of $P(\la)$ with the well-known shifting relationships between the minimal indices (compare with the cumbersome proofs in \cite{DDM2010Fiedler} and the very complicated ones in \cite{DDM2012rectangular}).

Next, we show what conditions on $\lambda M_1+M_0$ are needed for a block Kronecker pencil \eqref{eq:linearization_general} to be a strong linearization of a {\em prescribed} matrix polynomial $P(\lambda)$.

\begin{theorem} \label{thm:givenPblockKron}
Let $P(\lambda) = \sum_{k=0}^d P_k \la^k \in \FF[\la]^{m \times n}$, let $\mathcal{L}(\lambda)$ be an $(\e,n,\eta,m)$-block Kronecker pencil as in \eqref{eq:linearization_general} with $\e + \eta + 1 =d$, let us consider $M_0$ and $M_1$ partitioned into $(\eta+1)\times (\e +1)$ blocks each of size $m\times n$, and let us denote these blocks by $[M_0]_{ij}, [M_1]_{ij} \in\FF^{m\times n}$  for $i=1,\hdots,\eta+1$ and $j=1,\hdots,\e +1$.
If
 \begin{equation}\label{eq:condition_coeff}
   \sum_{i+j=\pdeg+2-k} [M_1]_{ij} + \sum_{i+j=\pdeg+1-k} [M_0]_{ij}=P_k, \quad \mbox{for $k=0,1,\hdots,d$,}
 \end{equation}
then $\mathcal{L}(\lambda)$ is a strong linearization of $P(\lambda)$, the right minimal indices of $\EL{}$ are those of $P(\la)$ shifted by $\e$, and the left minimal indices of $\EL{}$ are those of $P(\la)$ shifted by $\eta$.
\end{theorem}

\begin{proof} A direct multiplication, the condition $\e + \eta + 1 =d$, and some elementary manipulations of summations allow us to express $Q(\la)$ in \eqref{eq:condition} as
\[
Q(\la) = \sum_{k=0}^d \la^k \left( \sum_{i+j=\pdeg+2-k} [M_1]_{ij} +\sum_{i+j=\pdeg+1-k} [M_0]_{ij} \right).
\]
Then \eqref{eq:condition_coeff} implies that $Q(\la) = P(\la)$ and the result follows from Theorem \ref{thm:strong}.
\end{proof}

Theorem \ref{thm:givenPblockKron} admits a revealing interpretation in terms of block antidiagonals of $M_0$ and $M_1$. To see this, note that equation \eqref{eq:condition_coeff} tells us that the sum of the blocks on the $(d-k)$th block antidiagonal of $M_0$ plus the sum of the blocks on the $(d-k+1)$th block antidiagonal of $M_1$ must be equal to the coefficient $P_k$ of $P(\la)$. This implies that the upper-left block of $M_1$ must be equal to $P_d$, and that the lower-right block of $M_0$ must be equal to $P_0$, that is, the pencil $\lambda M_1+M_0$ has the form
\begin{equation}
  \lambda M_1 +M_0=
  \left[
    \begin{array}{ccc}
      \lambda P_{\pdeg}+[M_0]_{11} &\iddots&\iddots\\
      \iddots&\iddots&\iddots\\
      \iddots&\iddots &\lambda [M_1]_{\eta+1,\e+1} +P_0
    \end{array}
  \right]\>.
\end{equation}
There are infinitely many ways to select the remaining block entries of $M_1$ and $M_0$ to {\it synthesize} $P(\lambda)$ in the pencil $\lambda M_1+M_0$.

In Example \ref{ex-blockKron} we show three different block Kronecker pencils that are all strong linearizations of a grade $5$ matrix polynomial $P(\lambda)$. These three pencils have parameters $\e = \eta = 2$.
Moreover, the corresponding pencils $\lambda M_1+M_0$ in these block Kronecker pencils do not follow a staircase pattern for $\la P_5 + P_4, P_3, \ldots , P_0$, that is, they are not  permuted Fiedler pencils \cite[Theorem 4.5]{Dopico:2016:BlockKronecker}.
\begin{example} \label{ex-blockKron} {\rm
Let $P(\lambda) =\sum_{k=0}^5 P_k\lambda^k \in \FF [\la]^{m \times n}$ and let $A,B \in \FF^{m\times n}$ be arbitrary constant matrices. The following block Kronecker pencils
\begin{align*}
&\left[\begin{array}{ccc|cc}
\lambda P_5+P_4 & 0 & 0 & -I_m & 0 \\
0 & \lambda P_3+P_2& 0 & \lambda I_m & -I_m \\
0 & 0 & \lambda P_1+P_0 & 0 & \lambda I_m \\ \hline
-I_n & \lambda I_n & 0 & 0 & 0 \\
0 & -I_n & \lambda I_n & 0 & 0
\end{array}\right], \\
&\left[\begin{array}{ccc|cc}
\lambda P_5 & \lambda P_4 & \lambda P_3 & -I_m & 0 \\
0 & 0& \lambda P_2 & \lambda I_m & -I_m \\
0 & 0 & \lambda P_1+P_0 & 0 & \lambda I_m \\ \hline
-I_n & \lambda I_n & 0 & 0 & 0 \\
0 & -I_n & \lambda I_n & 0 & 0
\end{array}\right], \quad \mbox{and} \\
&\left[\begin{array}{ccc|cc}
\lambda P_5 & A & P_2 & -I_m & 0 \\
\lambda P_4 & -\lambda A& \lambda B+P_1 & \lambda I_m & -I_m \\
\lambda P_3 & -\lambda B & P_0 & 0 & \lambda I_m \\ \hline
-I_n & \lambda I_n & 0 & 0 & 0 \\
0 & -I_n & \lambda I_n & 0 & 0
\end{array}\right]
\end{align*}
are all strong linearizations of $P(\lambda)$.}
\end{example}

\begin{remark} \label{rem.lastsecblockKron} {\rm As discussed above, equation \eqref{eq:condition_coeff} allows us to construct infinitely many block Kronecker pencils that are strong linearizations of a prescribed matrix polynomial $P(\la)$. Therefore, a natural question is which ones can be reliably used for computing all the eigenvalues of $P(\la)$, when $P(\la)$ is regular, or all the eigenvalues and minimal indices of $P(\la)$, when $P(\la)$ is singular, via either the QZ algorithm \cite{golubvanloan4} or the staircase algorithm \cite{demmelkagstrom1,demmelkagstrom2,vandoorenstaircase}. From the point of view of backward errors, this is clearly stated in Corollary \ref{cor:FINperturbation} and carefully analyzed in the paragraphs before that corollary, but we advance here the main conclusions for impatient readers. First, the use of block Kronecker pencils \eqref{eq:linearization_general} is reliable only if $\|\la M_1 + M_0\|_F \approx \|P(\la)\|_F$. This is intuitively natural, because, according to \eqref{eq:condition}, if $\|\la M_1 + M_0\|_F \gg \|P(\la)\|_F$, then small relative perturbations in $\la M_1+M_0$ might produce huge perturbations in $P(\la)$, and $\|\la M_1 + M_0\|_F \ll \|P(\la)\|_F$ cannot happen as a consequence of \eqref{eq:condition} and Lemma \ref{lemma:normsproducts}. In addition,
$\|\la M_1 + M_0\|_F \approx \|P(\la)\|_F \approx 1$ must also hold, which is also natural since either $\|P(\la)\|_F \ll 1$ or $\|P(\la)\|_F \gg 1$ would lead to highly unbalanced block Kronecker pencils \eqref{eq:linearization_general}, with the norms of the antidiagonal blocks either much larger or much smaller than the norm of the $(1,1)$ block. In fact, it is proved in Corollary \ref{cor:FINperturbation} that any block Kronecker pencil with $\|\la M_1 + M_0\|_F \approx \|P(\la)\|_F \approx 1$ leads to small relative backward errors from the polynomial point of view. This condition still allows us to use infinitely many pencils that might have additional advantages as preservation of structures.
}
\end{remark}

\section{Backward error analysis of complete polynomial eigenproblems sol\-ved via block Kronecker pencils}\label{sec:expansion}
The problem of computing in floating point arithmetic the complete eigenstructure of a matrix polynomial $P(\la)$ is called in this paper the {\em complete polynomial eigenproblem}. The complete eigenstructure consists of all of the eigenvalues, finite and infinite, and all of the minimal indices, left and right, of $P(\la)$. This eigenstructure can be efficiently computed via the {\em staircase algorithm for matrix pencils} applied to any strong linearization $\EL{}$ of the polynomial that allows us to recover the minimal indices of the polynomial from those of the linearizations via constant shifts (like those of Theorem \ref{thm:givenPblockKron} for block Kronecker pencils). The staircase algorithm for pencils was introduced for the first time in \cite{vandoorenstaircase} and was further developed in \cite{demmelkagstrom1,demmelkagstrom2}, where reliable software for computing such a staircase form was presented. Though problems involving singular polynomials arise very often in control theory, the matrix polynomials arising in many other applications are normally square and regular. In this case the complete eigenstructure does not include minimal indices and the algorithm of choice is the simpler {\em QZ algorithm} \cite{golubvanloan4}.

The staircase and the QZ algorithms have been shown to be {\em backward stable}, but it ought to be stressed that the backward stability of these two algorithms does {\em not} imply that the computed eigenstructure is the exact one of the given linearization: in general this problem is ill-posed, which implies that even an arbitrarily small perturbation may yield a different eigenstructure. Since this is not the subject of this paper, we refer to \cite{edelelmkags1,edelelmkags2} for a more elaborate discussion on these aspects. Nonetheless, the standard backward error results guarantee that if the staircase algorithm or the QZ algorithm are applied to a strong linearization $\EL{}$ in a computer with unit roundoff $\mathbf{u}$, then {\em the computed complete eigenstructure of $\EL{}$ is the exact complete eigenstructure of a nearby matrix pencil  $\mathcal{L}(\lambda) + \Delta \mathcal{L}(\lambda)$ such that}
\begin{equation} \label{eq:initialbackwarderror}
\frac{\|\Delta \mathcal{L}(\lambda)\|_F}{\|\mathcal{L}(\lambda)\|_F} = O(\mathbf{u}),
\end{equation}
where $\|\cdot\|_F$ denotes the Frobenius norm introduced in Definition \ref{def:norm}. However, \eqref{eq:initialbackwarderror} is not the desired ideal result for the  original problem of computing the complete eigenstructure of the matrix polynomial $P(\la)$ of given grade $d$. The desired backward error result would be that {\em the computed complete eigenstructure of $P(\la)$ is the exact complete eigenstructure of a nearby matrix polynomial  $P(\lambda) + \Delta P(\lambda)$ also of grade $d$ and such that}
\begin{equation} \label{eq:idealbackwarderror}
\frac{\|\Delta P(\lambda)\|_F}{\|P(\lambda)\|_F} = O(\mathbf{u}).
\end{equation}

In order to establish \eqref{eq:idealbackwarderror}, if possible, starting from \eqref{eq:initialbackwarderror}, two results must be proved: {\em (i)} that the perturbed pencil  $\mathcal{L}(\lambda) + \Delta \mathcal{L}(\lambda)$ is a strong linearization for some matrix polynomial $P(\lambda) + \Delta P(\lambda)$ of grade $d$ with the shifting relations between the minimal indices of $\mathcal{L}(\lambda) + \Delta \mathcal{L}(\lambda)$ and $P(\lambda) + \Delta P(\lambda)$ equal to the shifting relations between the minimal indices of $\mathcal{L}(\lambda)$ and $P(\lambda)$; and {\em (ii)} to prove a {\em perturbation bound} of the type
\begin{equation} \label{eq:perturbationbound}
\frac{\|\Delta P(\lambda)\|_F}{\|P(\lambda)\|_F} \leq C_{P,\mathcal{L}} \, \frac{\|\Delta \mathcal{L}(\lambda)\|_F}{\|\mathcal{L}(\lambda)\|_F},
\end{equation}
with $C_{P,\mathcal{L}}$ a moderate number depending, in principle, on $P(\la)$ and $\EL{}$. We emphasize that to prove {\em (i)} is much easier for regular than for singular polynomials, because in the former case there are no minimal indices involved in the computations. Observe also that the minimal indices of $P(\la)+\Delta P(\lambda)$ are computed via the recovery rules valid for the unperturbed linearization $\EL{}$ applied to the computed minimal indices of $\EL{}$, that is, to the exact minimal indices of $\EL{} + \Delta \EL{}$. Therefore, if the recovery rules for the minimal indices of $\EL{} + \Delta \EL{}$ were different than those of $\EL{}$, such a method for computing the minimal indices of $P(\lambda)$ would not make any sense because the minimal indices are integer numbers. We repeat that thereby, we do not claim that the exact eigenstructure of $\EL{}$ was computed, or, even more, that the computed eigenstructure is close to that of $\EL{}$, but rather that the exact eigenstructure of a {\em nearby} pencil $\EL{} +\Delta \EL{}$ was computed, which may be quite different than the one of $\EL{}$ for ill-conditioned problems \cite{edelelmkags1,edelelmkags2}.

The goal of this section is to study these questions for any block Kronecker pencil $\EL{}$ as in \eqref{eq:linearization_general} of a given polynomial $P(\la)$ of grade $d$ and size $m\times n$. In plain words, we will prove that {\em if the block Kronecker pencil satisfies $\|\la M_1 + M_0\|_F  \approx \|P(\la)\|_F$ and $P(\la)$ is scaled to satisfy $\|P(\la)\|_F = 1$}, then \eqref{eq:perturbationbound} holds with $C_{P,\mathcal{L}} \approx d^{3} \sqrt{m+n}$. Therefore, {\em under these two conditions, we get perfect structured backward stability from the polynomial point of view} when the block Kronecker pencils are combined with the staircase or QZ algorithms for computing the complete eigenstructure of $P(\la)$. We emphasize that this is no longer true if  $\|\la M_1 + M_0\|_F  \gg \|P(\la)\|_F$, because in this case we will prove that $C_{P,\mathcal{L}}$ in \eqref{eq:perturbationbound} is huge. Note that $\|\la M_1 + M_0\|_F  \gg \|P(\la)\|_F$ may happen, for instance, if in the last block Kronecker pencil of Example \ref{ex-blockKron} the arbitrary matrices $A$ or $B$ have very large norms. Observe that the permuted Fiedler pencils in \cite[Theorem 4.5]{Dopico:2016:BlockKronecker} satisfy $\|\la M_1 + M_0\|_F  = \|P(\la)\|_F$ and, so, our analysis guarantees perfect structured polynomial backward stability for all Fiedler pencils.

Backward error analyses valid simultaneously for the complete eigenstructure, i.e., global analyses, of complete polynomial eigenproblems (and complete scalar rootfinding problems) solved by linearizations are not new in the literature. They appeared for the first time in the seminal paper \cite{van1983eigenstructure}, were studied in the influential work \cite{Edelman1995}, and have received considerable attention in recent years \cite{deTeran2015imajna,lawrence-corless-2014,lawrence-corless-2015,LVBVD,nakatsukasa-noferini,Chebyrootfinding}.
However, we stress that the analysis developed in this paper has a number of key features which are not present in any of the other analyses published so far: first, it is not a first order analysis since it holds for perturbations $\Delta \EL{}$ of finite norm; second, it provides very detailed bounds, and not just vague big-O bounds as other analyses do; third, it is valid simultaneously for a very large class of linearizations for which backward error analyses are not yet known; and, fourth, it establishes a framework that may be generalized to other classes of linearizations.

Before proceeding, we remark that our analysis is of a completely different nature than the ``local'' residual backward error analyses presented in \cite{Higham06backwarderror,tisseurlaa2000}, which are only valid for regular matrix polynomials, are based on the residual of a particular computed eigenvalue-vector pair, and find a nearby polynomial to the original one that has as exact eigenpair the particular computed one. A key difference with our analysis is that in these local analyses the nearby polynomial is different for each computed eigenpair, while in our case it is the same for the complete eigenstructure.

The main result in this section is Theorem \ref{thm:perturbation}, whose proof requires considerable efforts. The proof is split into three main steps that are briefly described in the next paragraphs in such a way that the reader may follow easily the main flow of the proof. We emphasize that the complete eigenstructure of the initial perturbed pencil $\EL{} + \Delta \EL{}$ does not change in the three steps except for the constant shifts of the minimal indices in the third step. In this section we assume that $\FF = \mathbb{R}$ or $\FF = \mathbb{C}$.

\smallskip
\noindent
{\bf Initial data.} A matrix polynomial $P(\lambda)=\sum_{k=0}^d P_k \la^k \in \FF[\la]^{m\times n}$ and a block Kronecker pencil $\mathcal{L}(\lambda)$ as in \eqref{eq:linearization_general} such that
\begin{equation} \label{eq:Pinsectionerrors}
P(\lambda) = (\Lambda_\eta(\lambda)^T\otimes I_m)(\lambda M_1+M_0)(\Lambda_{\e}(\lambda)\otimes I_n), \qquad \mbox{with $\e + \eta + 1 = d$,}
\end{equation}
are given. A perturbation pencil $\Delta \EL{}$ of $\EL{}$ is also given and is partitioned conformably to the natural partition of $\EL{}$, that is,
\begin{equation}\label{eq:perturbed_pencil}
\mathcal{L}(\lambda) + \Delta \mathcal{L}(\lambda) =
\left[\begin{array}{c|c}
\lambda M_1+M_0 + \Delta \mathcal{L}_{11}(\lambda) & L_\eta(\lambda)^T\otimes I_m + \Delta \mathcal{L}_{12}(\lambda) \\ \hline
L_{\e}(\lambda)\otimes I_n + \Delta \mathcal{L}_{21}(\lambda) & \Delta \mathcal{L}_{22} (\lambda)
\end{array}\right].
\end{equation}

\smallskip
\noindent
{\bf First step.}  We establish a bound on $\|\Delta \mathcal{L}(\lambda)\|_F$ that allows us to construct a strict equivalence transformation that returns the $(2,2)$-block of the perturbed pencil \eqref{eq:perturbed_pencil} back to zero as in $\EL{}$:
\begin{align}
  \label{eq:constant_reduction_big}
    &\left[
    \begin{array}{cc}
      I_{(\eta+1)\rowdim}&0\\
        C&I_{\e\coldim}
    \end{array}
  \right]
  \left(
 \mathcal{L}(\lambda)+\Delta \mathcal{L}(\lambda)
  \right)
  \left[
    \begin{array}{cc}
      I_{(\e+1)\coldim}&D\\
      0&I_{\eta\rowdim}
    \end{array}
  \right]\\
  & \nonumber \phantom{aaaa} =
  \left[
    \begin{array}{cc}
      \lambda M_1+M_0+\DEL{11}& L_\eta (\lambda)^T\otimes I_m+ \Delta \widetilde{\mathcal{L}}_{12}(\lambda) \\
      L_{\e}(\lambda)\otimes I_n+\Delta \widetilde{\mathcal{L}}_{21}(\lambda)&0
    \end{array}
  \right] =: \mathcal{L}(\lambda)+\Delta \widetilde{\mathcal{L}}(\lambda).
\end{align}
This construction is equivalent to solving a nonlinear system of matrix equations whose unknowns are the constant matrices $C$ and $D$.  Moreover, we prove detailed bounds on $\|(C,D)\|_F$, $\|\Delta \widetilde{\mathcal{L}}_{12}(\lambda)\|_F$, and $\|\Delta \widetilde{\mathcal{L}}_{21}(\lambda) \|_F$ in terms of $\|\Delta \mathcal{L}(\lambda)\|_F$. It is important to remark that $\mathcal{L}(\lambda)+\Delta \mathcal{L}(\lambda)$ and the pencil $\mathcal{L}(\lambda)+\Delta \widetilde{\mathcal{L}}(\lambda)$ in \eqref{eq:constant_reduction_big} have the same complete eigenstructures (including minimal indices), since they are strictly equivalent \cite[Definition 3.1]{de2014spectral}.

\begin{remark} \label{rem:2empty} {\rm This first step is not needed if either $\e =0$ or $\eta =0$, which means that one of the anti-diagonal blocks and the zero block in \eqref{eq:linearization_general} are not present. These cases are important since include the first and second Frobenius companion pencils.
}
\end{remark}

\smallskip
\noindent
{\bf Second step.}
The second step consists of establishing bounds on $\|\Delta \widetilde{\mathcal{L}}_{12}(\lambda)\|_F$ and $\|\Delta \widetilde{\mathcal{L}}_{21}(\lambda) \|_F$ that guarantee that {\em $\mathcal{L}(\lambda)+\Delta \widetilde{\mathcal{L}}(\lambda)$ in \eqref{eq:constant_reduction_big} is a strong block minimal bases pencil}. This requires two substeps: (a) to prove that $K_1(\la) := L_{\e}(\lambda)\otimes I_n+\Delta \widetilde{\mathcal{L}}_{21}(\lambda)$ and $K_2(\la) := L_\eta(\lambda)\otimes I_m+ \Delta \widetilde{\mathcal{L}}_{12}(\lambda)^T$ are both minimal bases with their row degrees all equal to $1$, and (b) to prove that there exist minimal bases
$$\Lambda_{\e}(\lambda)^T \otimes I_n +\Delta R_\e(\lambda)^T \quad \mbox{and} \quad \Lambda_\eta(\lambda)^T \otimes I_m + \Delta R_\eta (\lambda)^{T}$$
dual, respectively, to $K_1(\la)$ and $K_2(\la)$ with their row degrees all equal, respectively, to $\e$ and $\eta$. In addition, we prove detailed bounds on $\|\Delta R_\e(\lambda)\|_F$ in terms of $\|\Delta \widetilde{\mathcal{L}}_{21}(\lambda) \|_F$, and on $\|\Delta R_\eta(\lambda)\|_F$ in terms of $\|\Delta \widetilde{\mathcal{L}}_{12}(\lambda) \|_F$.

\begin{remark} \label{rem:onlyoneproof} {\rm It is only needed to prove the results in the substeps (a) and (b) for $K_1(\la) = L_{\e}(\lambda)\otimes I_n+\Delta \widetilde{\mathcal{L}}_{21}(\lambda)$ and $\Lambda_{\e}(\lambda)^T \otimes I_n +\Delta R_\e(\lambda)^T$, since, then, the ones for $K_2(\la) = L_\eta(\lambda)\otimes I_m+ \Delta \widetilde{\mathcal{L}}_{12}(\lambda)^T$ and $\Lambda_\eta(\lambda)^T \otimes I_m + \Delta R_\eta (\lambda)^{T}$ follow as corollaries.}
\end{remark}

\smallskip
\noindent
{\bf Third step.} Combining the first and second steps and Theorems \ref{thm:blockminlin} and \ref{thm:indicesminbaseslin}, we get that $\mathcal{L}(\lambda)+\Delta \mathcal{L}(\lambda)$ is a strong linearization of the matrix polynomial
\begin{align} \label{eq:polyperturbed3step}
& P(\lambda) + \Delta P(\lambda) \\ & \phantom{aaaa} := \left( \Lambda_\eta(\lambda)^T\otimes I_m+\Delta R_\eta (\lambda)^{T}\right)
\left( \la M_1 +M_0 +\Delta \mathcal{L}_{11}(\lambda)  \right)
\left( \Lambda_{\e}(\lambda)\otimes I_n+\Delta R_\e (\lambda)\right), \nonumber
\end{align}
that the right minimal indices of $\mathcal{L}(\lambda)+\Delta \mathcal{L}(\lambda)$ are those of $P(\lambda) + \Delta P(\lambda)$ shifted by $\e$, and that the left minimal indices of $\mathcal{L}(\lambda)+\Delta \mathcal{L}(\lambda)$ are those of $P(\lambda) + \Delta P(\lambda)$ shifted by $\eta$, i.e., the shifting relations between the minimal indices are the same as those between the minimal indices of $\mathcal{L}(\lambda)$ and $P(\lambda)$. The rest of the proof consists of bounding $\|\Delta P (\la)\|_F / \|P (\la)\|_F$ in terms of $\|\Delta \EL {}\|_F / \|\EL{}\|_F$ using the bounds obtained in the first and second steps.

\medskip

In the rest of this section, the three steps described above are developed in detail. We use very often, without explicitly referring to, the properties of the Frobenius norm of matrix polynomials in Lemma \ref{lemma:normsproducts} and, also, that for any matrix polynomial $P(\la)$ and any submatrix $B(\la)$ of $P(\la)$, the inequality $\|B(\la)\|_F \leq  \|P(\la)\|_F$ holds.

\subsection{First step: solving a system of quadratic Sylvester-like matrix equations for constructing the strict equivalence \eqref{eq:constant_reduction_big}} \label{sec:firststep}
For brevity, hereafter we use the following notation for the anti-diagonal blocks of block Kronecker pencils, which are constructed from the pencil \eqref{eq:Lk}: $L_{k}(\lambda)\otimes I_{\ell} =:(\lambda F_{k} - E_{k})\otimes I_{\ell}=: \lambda F_{k\ell}-E_{k\ell}$, where
\begin{equation}\label{eq:bases_definitions}
  E_{k\ell}=
  \left[
    \begin{array}{cc}
      I_{k}&0_{k\times 1}
    \end{array}
  \right]\otimes I_\ell\>,\quad \mbox{and} \quad
  F_{k\ell}=
  \left[
    \begin{array}{cc}
      0_{k\times 1}&I_{k}
    \end{array}
  \right]\otimes I_\ell\>.
\end{equation}
In addition, the natural blocks of the perturbation $\Delta \EL{}$ in \eqref{eq:perturbed_pencil} are denoted by
\begin{equation} \label{eq:blocksofdeltaL}
 \Delta \mathcal{L}(\lambda) =
 \left[ \begin{array}{c|c}
\Delta \mathcal{L}_{11}(\lambda) & \Delta \mathcal{L}_{12}(\lambda) \\ \hline
\Delta \mathcal{L}_{21}(\lambda) & \Delta \mathcal{L}_{22}(\lambda)
\end{array} \right] =:
\left[\begin{array}{c|c}
\lambda \Delta B_{11}+\Delta A_{11} & \lambda \Delta B_{12}+ \Delta A_{12} \\ \hline
\lambda \Delta B_{21}+\Delta A_{21} & \lambda \Delta B_{22}+\Delta A_{22}
\end{array}\right] \, .
\end{equation}
According to Remark \ref{rem:2empty}, we assume that $\e \ne 0$ and $\eta \ne 0$ throughout this subsection.

The main result of this subsection is Theorem \ref{thm:finalofstep1} and the starting point is the trivial Lemma \ref{lem:trivial}, which follows from elementary matrix operations applied to the lower-right block in \eqref{eq:constant_reduction_big}.

\begin{lemma} \label{lem:trivial} There exist constant matrices $C\in\FF^{\e\coldim \times (\eta+1)\rowdim}$ and  $D\in\FF^{(\e+1)\coldim\times \eta\rowdim}$ satisfying
\eqref{eq:constant_reduction_big} if and only if
\begin{equation}\label{eq:gen_sylv2}
   \left[
    \begin{array}{cc}
        C&I_{\e\coldim}
    \end{array}
  \right]
  (\EL{}+\DEL{})
  \left[
    \begin{array}{c}
      D\\
      I_{\eta\rowdim}
    \end{array}
  \right]
  =0\>.
\end{equation}
Moreover, with the notation introduced in \eqref{eq:bases_definitions} and \eqref{eq:blocksofdeltaL}, the equation \eqref{eq:gen_sylv2} is equivalent to the following system of quadratic Sylvester-like matrix equations
\begin{equation}\label{eq:gen_sylv}
  \left\{
    \begin{array}{l}
      C (E_{\eta\rowdim}^{T} - \Delta A_{12} )+(E_{\e\coldim} - \Delta A_{21})D= \Delta A_{22}+C(M_0+\Delta A_{11})D\\
      C (F_{\eta\rowdim}^{T} + \Delta B_{12})+(F_{\e\coldim} + \Delta B_{21})D=-\Delta B_{22}-C(M_1+\Delta B_{11})D
    \end{array}
  \right.\> ,
\end{equation}
for the unknown matrices $C$ and $D$.
\end{lemma}

The system of matrix equations \eqref{eq:gen_sylv} is equivalent to a system of $2\e\eta\coldim\rowdim$ quadra\-tic scalar equations in the $2\e\eta\coldim\rowdim+(\e+\eta)\rowdim\coldim$ unknown entries of $C$ and $D$. Therefore, \eqref{eq:gen_sylv} is an underdetermined system of equations that may have infinitely many solutions.
Our aim is to establish conditions on $\|\Delta \EL{}\|_F$ that guarantee the existence of a solution $(C,D)$ to \eqref{eq:gen_sylv} with $\|(C,D)\|_F \lesssim \|\Delta \EL{}\|_F$, where the norm $\|(C,D)\|_F$ was defined in \eqref{eq:defnormpair}. This is done in Theorem \ref{thm:gen_sylvester_solution}, whose proof follows that of Stewart \cite[Theorem 5.1]{doi:10.1137/0709056} (see also \cite[Theorem 2.11, p. 242]{stewartsunbook} for a more general and more accesible result) and is based on a fixed point iteration argument. However, we emphasize that the result by Stewart is valid only for certain nonlinear matrix equations having a unique solution, while in our case there may be infinitely many solutions.

The solution of \eqref{eq:gen_sylv} relies upon first solving the system of linear Sylvester equations obtained by removing the quadratic terms in $C$ and $D$ of \eqref{eq:gen_sylv}. Such a system is:
\begin{equation}\label{eq:gen_sylv_linear}
  \left\{
    \begin{array}{l}
      C (E_{\eta\rowdim}^{T} -\Delta A_{12})+(E_{\e\coldim} - \Delta A_{21})D=\Delta A_{22}\\
      C (F_{\eta\rowdim}^{T} + \Delta B_{12})+(F_{\e\coldim} + \Delta B_{21} )D=-\Delta B_{22}
    \end{array}
  \right.\> ,
\end{equation}
which is equivalent to the underdetermined standard linear system $(T+\Delta T)x=b$ given by
\begin{align}
\nonumber
  &\left(
    \underbrace{\left[
      \begin{array}{c|c}
        E_{\eta\rowdim}\otimes I_{\e\coldim} &I_{\eta \rowdim}\otimes E_{\e\coldim}\\\hline
        F_{\eta\rowdim}\otimes I_{\e\coldim}&I_{\eta \rowdim}\otimes F_{\e\coldim}
      \end{array}
    \right]}_{=: T}
    + \underbrace{
    \left[
      \begin{array}{c|c}
        -\Delta A_{12}^{T}\otimes I_{\e\coldim} &-I_{\eta \rowdim}\otimes \Delta A_{21}\\\hline \\[-2.2 ex]
        \Delta B_{12}^{T}\otimes I_{\e\coldim}&I_{\eta \rowdim}\otimes \Delta B_{21}
      \end{array}
    \right]}_{=: \Delta T}
    \right)
    \underbrace{\left[
    \begin{array}{c}
      \mathrm{vec}(C)\\\hline
      \mathrm{vec}(D)
    \end{array}
  \right]}_{=:x }\\ \label{eq:linear_operator_equation}
  &\phantom{aaaa} =
  \underbrace{\left[
    \begin{array}{c}
      \mathrm{vec}(\Delta A_{22})\\\hline
      -\mathrm{vec}(\Delta B_{22})
    \end{array}
  \right]}_{=:b}
  \>,
\end{align}
where, for any $m\times n$ matrix $M=\left[m_{ij}\right]$, the column vector ${\rm vec} (M)$ is the {\em vectorization} of $M$, namely, $\mbox{\rm vec}(M): = [ m_{11}\,\hdots\, m_{m1}\, m_{12}\, \hdots\, m_{m2}\, \hdots\, m_{1n}\,\hdots\, m_{mn}]^T$ (see Horn and Johnson \cite[Def. 4.2.9]{Horn}, for instance). For brevity, and with an abuse of notation we use expressions such as ``$(C,D)$ is a solution of \eqref{eq:linear_operator_equation}''.

Lemma \ref{lemma_min_singular_value} proves that the matrix $T$ in \eqref{eq:linear_operator_equation} has full row rank and provides an expression for its minimum singular value.
This implies that if $\|\Delta T\|_2$ is small enough, then $T+ \Delta T$ has also full row rank and the linear system \eqref{eq:linear_operator_equation} is consistent, as well as the equivalent system of matrix equations \eqref{eq:gen_sylv_linear}. The proof of Lemma \ref{lemma_min_singular_value} is long and can be found in Appendix \ref{sec:proof-sigmamin}.
Here and in the rest of the paper the {\em minimum singular value} of any matrix $M$ is denoted by $\sigma_{\rm min}(M)$.

\begin{lemma}\label{lemma_min_singular_value} The matrix $T$ in \eqref{eq:linear_operator_equation} has full row rank and its minimum singular value is given by
  \begin{equation}\label{eq:smallest_singvalue}
    \sigma_{\rm min}(T)=
    \begin{cases}
      2\sin{\frac{\pi}{4\min{(\eta,\e)}+2}},& \e\neq \eta\\
      2\sin{\frac{\pi}{4\eta}},& \e=\eta
    \end{cases} \, .
  \end{equation}
\end{lemma}

The following simple lower bound on $\sigma_{\min} (T)$ is useful to get bounds that can be easily handled and are related to the grade of the original matrix polynomial.
\begin{corollary} \label{cor:boundonsigmamin} Let $T$ be the matrix in \eqref{eq:linear_operator_equation} and $d = \e + \eta + 1$. Then
\[
\sigma_{\rm min}(T) \geq \frac{2\sqrt{2}}{d} \,.
\]
\end{corollary}

\begin{proof} It follows from \eqref{eq:smallest_singvalue} and the inequality $\sin(x)\ge 2\sqrt{2}x/\pi$ for $0\le x \le \pi/4$.
\end{proof}

Lemma \ref{lemma:linear_bound} bounds the norm of the minimum 2-norm solution of \eqref{eq:linear_operator_equation} or, equivalently, of the minimum Frobenius norm solution of the matrix equation \eqref{eq:gen_sylv_linear}, since $\|[\mathrm{vec}(C)^T, \mathrm{vec}(D)^T ]^T \|_2 \allowbreak = \|(C,D) \|_F$.

\begin{lemma}\label{lemma:linear_bound}
Let $(T+\Delta T)x= b$ be the underdetermined linear system  \eqref{eq:linear_operator_equation}, and let us assume that $\sigma_{\min}(T)>\|\Delta T\|_2$.
Then $(T+\Delta T)x= b$ is consistent and its minimum norm solution $(C_0, D_0)$ satisfies
  \begin{equation}\label{eq:bound_CD}
    \|(C_0,D_0)\|_F\leq \frac{1}{\delta} \|(\Delta A_{22},\Delta B_{22})\|_F,
  \end{equation}
 where $ \delta:=\sigma_{\rm min}(T)-\|\Delta T\|_2$.
\end{lemma}
\begin{proof}
From Weyl's perturbation theorem for singular values \cite[Theorem 3.3.16]{Horn}, we get $\sigma_{\rm min}(T+\Delta T)\geq \sigma_{\rm min}(T)-\|\Delta T\|_2>0$.
Therefore, $T+\Delta T$ has full row rank and the linear system \eqref{eq:linear_operator_equation} is consistent. Its minimum norm solution, $(C_0, D_0)$, is given by $(T+\Delta T)^\dagger b$, where $(T+\Delta T)^\dagger$ denotes the Moore-Penrose pseudoinverse of $T+\Delta T$.
Then,
  \begin{align*}
    \|(C_0,D_0)\|_F \leq &\|(T+\Delta T)^\dagger \|_2\|(\Delta A_{22},\Delta B_{22})\|_F = \frac{1}{\sigma_{\rm min}(T+\Delta T)}\|(\Delta A_{22},\Delta B_{22})\|_F \\
   \leq &  \frac{1}{\sigma_{\rm min}(T)-\|\Delta T\|_2}\|(\Delta A_{22},\Delta B_{22})\|_F.
  \end{align*}
\end{proof}

From Lemma \ref{lemma:linear_bound}, it is clear that the quantity $\delta=\sigma_{\rm min}(T)-\|\Delta T\|_2$ will play a relevant role in our analysis. Therefore, we establish a tractable lower bound on $\delta$.

\begin{lemma} \label{lemm:deltabound} Let $T$ and $\Delta T$ be the matrices in \eqref{eq:linear_operator_equation}, let $\Delta \EL{}$ be the pencil in \eqref{eq:blocksofdeltaL}, and $d = \e + \eta + 1$. If $\|\Delta \EL{}\|_F < 1/d$, then
\[
\sigma_{\rm min}(T)-\|\Delta T\|_2 \geq \frac{2(\sqrt{2}-1)}{d} >0 \, .
\]
\end{lemma}

\begin{proof} Using standard properties of norms and Kronecker products \cite[Chapter 4]{Horn} (pay particular attention to \cite[p. 247, paragraph 1]{Horn}), we get
\begin{align*}
\|\Delta T\|_2 & \leq \left\| \left[
      \begin{array}{c}
        -\Delta A_{12}^{T}\otimes I_{\e\coldim} \\\hline \\[-2.2 ex]
        \Delta B_{12}^{T}\otimes I_{\e\coldim}
      \end{array}
    \right] \right\|_2 + \left\| \left[
      \begin{array}{c}
        -I_{\eta \rowdim}\otimes \Delta A_{21}\\\hline \\[-2.2 ex]
        I_{\eta \rowdim}\otimes \Delta B_{21}
      \end{array}
    \right] \right\|_2 \\
    & = \left\| \left[
      \begin{array}{c}
        -\Delta A_{12}^{T} \\\hline \\[-2.2 ex]
        \Delta B_{12}^{T}
      \end{array}
    \right] \right\|_2 + \left\| \left[
      \begin{array}{c}
        - \Delta A_{21}\\\hline \\[-2.2 ex]
         \Delta B_{21}
      \end{array}
    \right] \right\|_2
    \leq
    \left\| \left[
      \begin{array}{c}
        -\Delta A_{12}^{T} \\\hline \\[-2.2 ex]
        \Delta B_{12}^{T}
      \end{array}
    \right] \right\|_F + \left\| \left[
      \begin{array}{c}
        - \Delta A_{21}\\\hline \\[-2.2 ex]
         \Delta B_{21}
      \end{array}
    \right] \right\|_F  \\
    & \leq 2 \|\Delta \EL{}\|_F \,.
\end{align*}
From this inequality and Corollary \ref{cor:boundonsigmamin}, the result is obtained as follows:
\[
\sigma_{\rm min}(T)-\|\Delta T\|_2 \geq \! \frac{2\sqrt{2}}{d} -  2 \|\Delta \EL{}\|_F > \frac{2(\sqrt{2}-1)}{d}  .
\]
\end{proof}

Theorem \ref{thm:gen_sylvester_solution} is the key technical result of this section. It proves that the system of quadratic Sylvester-like matrix equations \eqref{eq:gen_sylv} has a solution $(C,D)$ such that $\|(C,D)\|_F \lesssim \|\Delta \EL{}\|_F$, whenever $\|\Delta \EL{}\|_F$ is properly upper bounded. As mentioned before, this theorem extends to underdetermined quadratic matrix equations results proved by Stewart for equations with a unique solution \cite[Theorem 5.1]{doi:10.1137/0709056}, \cite[Theorem 2.11, p. 242]{stewartsunbook}. The proof of Theorem \ref{thm:gen_sylvester_solution} follows those by Stewart.
\begin{theorem}\label{thm:gen_sylvester_solution}
  There exists a solution $(C,D)$ of the quadratic system of Sylvester-like matrix equations \eqref{eq:gen_sylv} satisfying
\begin{equation}\label{eq:norm_solution}
  \|(C,D)\|_F\leq 2\frac{\theta}{\delta}\>,
\end{equation}
whenever
\begin{equation}\label{eq:condition_convergence}
 \delta>0 \quad \mbox{and} \quad \frac{\theta\omega}{\delta^2}<\frac{1}{4}\>,
\end{equation}
where $\delta = \sigma_{\rm min}(T)-\|\Delta T\|_2$, $\theta :=\|(\Delta A_{22},\Delta B_{22})\|_F$, and $\omega:=\|(M_0+\Delta A_{11},M_1+\Delta B_{11})\|_F$.	
\end{theorem}

\begin{proof} Lemma \ref{lemma:linear_bound} and the hypothesis $\delta >0$ guarantee that the linear system of matrix equations \eqref{eq:gen_sylv_linear} is consistent, and, even more, that is consistent for any right-hand side.
Let the minimum norm solution of \eqref{eq:gen_sylv_linear} be denoted by $(C_0,D_0)$. It satisfies
\[
    \|(C_0,D_0)\|_F\leq \frac{1}{\delta}\|(\Delta A_{22}, \Delta B_{22})\|_F = \frac{\theta}{\delta}=:\rho_0,
\]
according to Lemma \ref{lemma:linear_bound}. Then, let us define a sequence $\{(C_i,D_i)\}_{i=0}^\infty$ of pairs of matrices as follows: for $i > 0$ the pair $(C_i, D_i)$ is the minimum norm solution of
  \begin{align}\label{eq:iteration}
\left\{\begin{array}{l}    C_{i}(E_{\eta\rowdim}^{T}-\Delta A_{12})+(E_{\e\coldim} - \Delta A_{21})D_{i}=\Delta A_{22}+C_{i-1}(M_0+\Delta A_{11})D_{i-1}\\
    C_{i}(F_{\eta\rowdim}^{T}+\Delta B_{12})+(F_{\e\coldim}+\Delta B_{21})D_{i}= -\Delta B_{22}-C_{i-1}(M_1+\Delta B_{11})D_{i-1}\end{array}\right. \, .
  \end{align}
 Therefore, vectorizing \eqref{eq:iteration} and using the matrix $T + \Delta T$ defined in \eqref{eq:linear_operator_equation}, we get
\begin{equation} \label{eq:iteration2}
    \left[
      \begin{array}{c}
        \mathrm{vec}(C_{i})\\
        \mathrm{vec}(D_{i})
      \end{array}
    \right]
    =
    \left[
      \begin{array}{c}
        \mathrm{vec} (C_0)\\
        \mathrm{vec} (D_0)
      \end{array}
    \right]
    +(T+\Delta T)^\dagger
    \left(
      \left[
        \begin{array}{c}
          \mathrm{vec}(C_{i-1}(M_0+\Delta A_{11})D_{i-1})\\
          -\mathrm{vec}(C_{i-1}(M_1+\Delta B_{11})D_{i-1})
        \end{array}
      \right]
    \right)\>.
\end{equation}
We claim that the sequence $\{(C_i,D_i)\}_{i=0}^\infty$ converges to a solution $(C,D)$ of \eqref{eq:gen_sylv}  satisfying \eqref{eq:norm_solution}.
To prove this, we first show that the sequence  $\{\|(C_i,D_i)\|_F\}_{i=0}^\infty$ is a bounded sequence.
If $\|(C_{i-1},D_{i-1})\|_F \leq \rho_{i-1}$, then we have from \eqref{eq:iteration2} that
  \begin{align*}
    \|(C_{i},D_{i})\|_F\nonumber \leq  &\|(C_0,D_0)\|_F \\
     &+\|(T+\Delta T)^\dagger \|_2\|\|(C_{i-1},D_{i-1})\|_F^{2}\|(M_0+\Delta A_{11},M_1+\Delta B_{11})\|_F\nonumber\\
    \leq &\rho_0+\rho_{i-1}^2\omega\delta^{-1}=:\rho_{i}\>.
    \label{eq:recurrecne_rho_i_1}
  \end{align*}
We may write the quantity $\rho_{i}$ in the equation above as $\rho_{i}=\rho_0(1+\kappa_i)$, where $\kappa_i$ satisfies the recursion
\begin{equation} \label{eq:fixedstewart}
  \left\{ \begin{array}{l}
    \kappa_1=\rho_0\omega\delta^{-1}=\theta \omega \delta^{-2},\\
    \kappa_{i+1}=\kappa_1(1+\kappa_i)^2\>.
    \end{array}\right.
\end{equation}
An induction argument proves that $0 < \kappa_1 < \kappa_2 < \cdots$, i.e., that the sequence is strictly increasing. In addition, if $\kappa_1<1/4$, then the function $g(x) := \kappa_1 (1 + x)^2$, which defines the fixed point iteration in \eqref{eq:fixedstewart}, has two positive fixed points, one smaller than one and another larger, and satisfies $0 < g(x) < 1$ and $0 < g'(x) < 1$ for $0 < x < 1$. Therefore, standard results on fixed point iterations imply that $\lim_{i \rightarrow \infty} \kappa_i = \kappa$, where $\kappa$ is the smallest fixed point of $g(x)$, i.e.,
\[
\kappa =\lim_{i\rightarrow \infty} \kappa_i = \frac{2\kappa_1}{1-2\kappa_1+\sqrt{1-4\kappa_1}} < 1,
  \]
and $\kappa_i <\kappa$ for all $i\geq 1$.
Thus, the norms of the elements of the sequence $\{(C_i,D_i)\}_{i=0}^\infty$ are bounded as
  \begin{equation} \label{eq:bound_CD2}
    \|(C_{i},D_{i})\|_F\leq \rho :=\lim_{i\rightarrow \infty}{\rho_i}=\rho_0(1+\kappa)\>.
  \end{equation}

We now show that the sequence $\{ (C_i,D_i)\}_{i=0}^\infty$ converges provided that $2\delta^{-1}\omega \rho<1$, which is ensured by \eqref{eq:condition_convergence}.
For this purpose, let $S_i=(S_i^{(C)},S_i^{(D)}):=(C_{i+1}-C_{i},D_{i+1}-D_{i})$.
Then \eqref{eq:iteration2} implies
\begin{align*}
\|S_i\|_F & \leq \|(T+\Delta T)^\dagger\|_2 \,
\left\| \begin{bmatrix}
{\rm vec}\,\left( C_i(M_0+\Delta A_{11})D_i - C_{i-1}(M_0+\Delta A_{11})D_{i-1} \right) \\
{\rm vec}\,\left( C_i(M_1+\Delta B_{11})D_i - C_{i-1}(M_1+\Delta B_{11})D_{i-1} \right)
\end{bmatrix} \right\|_2 \\
&\leq \delta^{-1} \,
\left\| \begin{bmatrix}
{\rm vec}\,\left(S_{i-1}^{(C)}(M_0+\Delta A_{11})D_i + C_{i-1}(M_0+\Delta A_{11})S_{i-1}^{(D)} \right) \\
{\rm vec}\,\left( S_{i-1}^{(C)}(M_1+\Delta B_{11})D_i + C_{i-1}(M_1+\Delta B_{11})S_{i-1}^{(D)} \right)
\end{bmatrix} \right\|_2 \\
&\leq 2\delta^{-1}\omega \rho \|S_{i-1}\|_F.
\end{align*}
Therefore, the sequence $\{ (C_i,D_i)\}_{i=0}^\infty$ is a Cauchy sequence, since $2\delta^{-1}\omega \rho<1$, and must have a limit
$
(C,D):=\lim_{i\rightarrow \infty} (C_i,D_i).
$
Taking limits of both sides of \eqref{eq:iteration}, we get that $(C,D)$ is a solution of \eqref{eq:gen_sylv}. Finally, from \eqref{eq:bound_CD2}, $\|(C,D)\|_F \leq \rho_0(1+\kappa) < 2\rho_0 = 2\delta^{-1}\theta$, which concludes the proof.
\end{proof}

Theorem \ref{thm:finalofstep1} completes the first step of the backward error analysis. Its proof follows from Theorem \ref{thm:gen_sylvester_solution} and norm inequalites. The numerical constants appearing in Theorem \ref{thm:finalofstep1} are not optimal and have been chosen to keep the analysis and the bounds simple.

\begin{theorem} \label{thm:finalofstep1}
Let $\EL{}$ be an $(\e,n,\eta , m)$-block Kronecker pencil as in \eqref{eq:linearization_general}, let $\e + \eta + 1 = d$, and let $\Delta \EL{}$ be any pencil with the same size as $\EL{}$ and such that
\begin{equation} \label{eq:boundL1}
\|\Delta \EL{}\|_F < \left(\frac{(\sqrt{2}-1)}{d} \right)^2 \, \frac{1}{1 + \|\la M_1 + M_0\|_F}.
\end{equation}
Then, there exist matrices $C\in\FF^{\e\coldim \times (\eta+1)\rowdim}$ and  $D\in\FF^{(\e+1)\coldim\times \eta\rowdim}$ that satisfy
\begin{equation} \label{eq:easyboundCD}
\|(C,D)\|_F \leq \frac{d \|\Delta \EL{}\|_F}{(\sqrt{2}-1)} ,
\end{equation}
and the equality \eqref{eq:constant_reduction_big} with
\begin{align}
\nonumber
& \max\{\|\Delta \widetilde{\mathcal{L}}_{21}(\lambda)\|_{F} ,  \|\Delta \widetilde{\mathcal{L}}_{12}(\lambda)\|_{F} \}
\\
&\phantom{aaaa} \leq \|\Delta \EL{}\|_F \left(1 + \frac{d}{(\sqrt{2}-1)} \, (\|\la M_1 + M_0\|_F + \|\Delta \EL{}\|_F) \right) \, . \label{eq:boundL12tilde}
\end{align}
\end{theorem}

\begin{proof} The notation in \eqref{eq:blocksofdeltaL} for the blocks of $\Delta \EL{}$ is used throughout the proof. We first prove that \eqref{eq:boundL1} implies  \eqref{eq:condition_convergence} and, so, the existence of $C$ and $D$ satisfying \eqref{eq:constant_reduction_big}. For this purpose, note that  \eqref{eq:boundL1} implies $\|\Delta \EL{}\|_F < 1/d < 1$ and, therefore, that Lemma \ref{lemm:deltabound} holds and that $\delta >0$. With this, and the notation in Theorem \ref{thm:gen_sylvester_solution}, we get
\begin{align*}
\frac{\theta \omega}{\delta^2} & \leq \frac{\|\Delta \EL{}\|_F \, (\|\la M_1 + M_0\|_F + \|\Delta \EL{}\|_F)}{\frac{4(\sqrt{2}-1)^2}{d^2} } < \frac{1}{4},
\end{align*}
and \eqref{eq:condition_convergence} indeed holds. Then, Theorem \ref{thm:gen_sylvester_solution} implies that there exist matrices $C$ and $D$ satisfying \eqref{eq:constant_reduction_big} and
\[
  \|(C,D)\|_F\leq 2\frac{\theta}{\delta} \leq  \frac{d \|\Delta \EL{}\|_F}{(\sqrt{2}-1)}  \, ,
\]
which proves \eqref{eq:easyboundCD}. Finally, from \eqref{eq:perturbed_pencil} and \eqref{eq:constant_reduction_big}, we obtain that
\begin{align*}
\Delta \widetilde{\mathcal{L}}_{12}(\lambda) & = (\la M_1 + M_0 + \Delta \mathcal{L}_{11}(\lambda)) D + \Delta \mathcal{L}_{12}(\lambda), \\
\Delta \widetilde{\mathcal{L}}_{21}(\lambda) & = C (\la M_1 + M_0 + \Delta \mathcal{L}_{11}(\lambda))  + \Delta \mathcal{L}_{21}(\lambda) \, ,
\end{align*}
which combined with \eqref{eq:easyboundCD} leads to \eqref{eq:boundL12tilde}.

\end{proof}

\subsection{Second step: proving that $\EL{} + \Delta \widetilde{\mathcal{L}} (\la)$ in \eqref{eq:constant_reduction_big} is a strong block minimal bases pencil} \label{sec:secondstep}
The main result of this section is Theorem \ref{thm:finalofstep2}. From the definition of strong block minimal bases pencils, it is not surprising that part of the proof of Theorem \ref{thm:finalofstep2} relies on algebraic results that characterize when a matrix polynomial is a minimal basis with all its row degrees equal and such that any minimal basis dual to it has also all its row degrees equal. In the first part of this section, we establish such characterizations. In this process, we use the complete eigenstructure of a matrix polynomial. Since it may include infinite eigenvalues, whose definition depends on which grade is chosen for the polynomial \cite[Section 2]{de2014spectral}, we adopt the convention in this section that anytime a complete eigenstructure is mentioned, the grade of the corresponding polynomial is equal to its degree.

A simple result that is used in this section is the next lemma.

\begin{lemma} \label{lemm:eigenstminbasis} Let $Q(\la) \in \FF[\la]^{m\times n}$ with $m<n$. Then, $Q(\la)$ is a minimal basis with all its row degrees equal if and only if the complete eigenstructure of $Q(\la)$ consists of only $n-m$ right minimal indices.
\end{lemma}

\begin{proof} It is a simple consequence of Theorem \ref{thm:minimal_basis} and the fact that if all the row degrees of $Q(\la) = \sum_{i=0}^q Q_i \lambda^i$ (where $Q_q \ne 0$) are equal, then its highest row degree coefficient matrix is equal to its leading coefficient $Q_q$. So, if $Q(\la)$ is a minimal basis with all its row degrees equal, then Theorem \ref{thm:minimal_basis} guarantees that $Q(\la)$ has no finite eigenvalues, since $Q(\la_0)$ has full row rank for all $\la_0 \in \overline{\FF}$, and that $Q(\la)$ has no infinite eigenvalues, since it is row reduced. In addition, $Q(\la)$ has no left minimal indices, since it has full row rank. Therefore, the complete eigenstructure of $Q(\la)$ consists of only $n-m$ right minimal indices.

Conversely, if the complete eigenstructure of $Q(\la)$ consists of only $n-m$ right minimal indices, then $\rank Q_q = \rank Q(\la) = m$, because $Q(\la)$ has neither infinite eigenvalues nor left minimal indices. This implies that all the row degrees of $Q(\la)$ are equal, since otherwise $\rank Q_q < m$, and that $Q(\la)$ is row reduced. Moreover, $\rank Q(\la_0) = m$ for all $\la_0 \in \overline{\FF}$ because $Q(\la)$ has no finite eigenvalues, and we get from Theorem \ref{thm:minimal_basis} that $Q(\la)$ is a minimal basis with all its row degrees equal.
\end{proof}

{\em Convolution matrices} will be useful in our characterizations of minimal bases and in a number of perturbation results. For any matrix polynomial $Q(\la) = \sum_{i=0}^q Q_i \lambda^i$ of grade $q$ and arbitrary size, we define in the spirit of Gantmacher \cite[Chapter XII]{gantmacher1960theory} the sequence of its {\em convolution matrices} as follows
\begin{equation} \label{eq:defconvolution}
C_j (Q(\la)) =
\underbrace{\left[
\begin{array}{llll}
Q_q \\
Q_{q-1} & Q_{q} \\
\vdots &  Q_{q-1} & \ddots \\
Q_0 & \vdots & \ddots & Q_{q} \\
    & Q_0    &        & Q_{q-1} \\
    &        & \ddots &   \vdots  \\
    &        &         & Q_0
\end{array}
\right]}_{\displaystyle j+1 \; \mbox{block columns}} , \quad \mbox{for $j=0,1,2,\ldots$.}
\end{equation}
Observe that for every $j$ the matrix $C_j (Q(\la))$ is a constant matrix. In particular for $j=0$, the matrix $C_0 (Q(\la))$ is a block column matrix whose block entries are the matrix coefficients of the polynomial. The fundamental property of these convolution matrices is that if $Z(\la)$ is any matrix polynomial of grade $j$ for which the product $Q(\la) Z(\la)$ is defined and is considered to have grade $q+j$, then
\begin{equation} \label{eq:fundpropconvolution}
C_0 (Q(\la) Z(\la)) = C_j (Q(\la)) \, C_0 (Z(\la)) \, .
\end{equation}
Another easy property of convolution matrices that we often use is that $\|C_j (Q(\la))\|_F \allowbreak = \sqrt{j+1} \, \|Q(\la)\|_F$. Note also that if $S(\la)$ is another matrix polynomial with the same grade as $Q(\la)$, then $C_j (Q(\la) + S(\la))= C_j (Q(\la)) + C_j (S(\la))$, for all $j$. The convolution matrices for pencils are particularly simple. For instance, for the pencil $L_\e (\la) \otimes I_n$ in the $(2,1)$-block of \eqref{eq:linearization_general}, we have with the notation in \eqref{eq:bases_definitions} that
\begin{equation} \label{eq:convolforpencils}
C_{\e - 1} (L_\e (\la) \otimes I_n) =
\underbrace{
\begin{bmatrix}
F_{\e n} \\
-E_{\e n} & \ddots \\
& \ddots & F_{\e n} \\
& & -E_{\e n}
\end{bmatrix}}_{\displaystyle \e \; \mbox{block columns}}
\left. \phantom{\begin{array}{l} l\\l\\l\\l\\l \end{array}} \! \! \!\!\! \!\!\!\right\}
\e+1 \mbox{ block rows} \,.
\end{equation}

Lemma \ref{lemm:eigenstminbasis} motivates us to look deeper into the right minimal indices of a matrix polynomial $Q(\la)$ and into the rational right null subspace $\mathcal{N}_r (Q)$ defined in \eqref{eq:nullspaces}. This is the purpose of Lemma \ref{lemm:nullspaconv}, which proves with precision for general matrix polynomials ideas that can be found in \cite[Chapter XII]{gantmacher1960theory} only for matrix pencils.

\begin{lemma} \label{lemm:nullspaconv} Let $Q(\la) \in \FF[\la]^{m\times n}$ and let $C_s(Q(\la))$, $s=0,1,2,...$, be the sequence of convolution matrices of $Q(\la)$. Then, the following statements hold.
\begin{enumerate}
\item[\rm (a)] Let $v(\la) \in \FF[\la]^{n\times 1}$ be a polynomial vector of grade $j$. Then, $v(\la) \in  \mathcal{N}_r (Q)$ if and only if $C_0 (v(\la)) \in \mathcal{N}_r (C_j (Q(\la)))$.
\item[\rm (b)] The smallest right minimal index of $Q(\la)$ is $j$ if and only if $C_{j-1} (Q(\la))$ has full column rank and $C_j (Q(\la))$ does not have full column rank.
\item[\rm (c)] If $j$ is the smallest right minimal index of $Q(\la)$ and $\dim  \mathcal{N}_r (C_j (Q(\la))) =p$, then $Q(\la)$ has at least $p$ minimal indices equal to $j$.
\end{enumerate}
\end{lemma}

\begin{proof} Part (a) follows immediately from \eqref{eq:fundpropconvolution}.
Before proceeding, note that part (a) establishes the natural bijection\footnote{We emphasize that this bijection is not a linear map since the fields of the linear spaces corresponding to the domain and the codomain are different. Nevertheless, it has some obvious linear properties that can be used.}
$v(\la) \longmapsto C_0 (v(\la))$ between the set of polynomial vectors of grade $j$ in
$\mathcal{N}_r (Q) \subseteq \FF(\la)^n$ and $\mathcal{N}_r (C_j (Q(\la))) \subseteq \FF^{(j+1)n \times 1}$. Indeed $v(\la) \mapsto C_0 (v(\la))$ is a bijection, since its inverse can be trivially constructed as follows: partition any $x \in \mathcal{N}_r (C_j (Q(\la))) \subseteq \FF^{(j+1)n \times 1}$ as $x = [x_j^T , \ldots, x_1^T , x_0^T]^T$, where $x_i \in \FF^{n \times 1}$, and note that
\begin{equation} \label{eq:inversebijection}
x \longmapsto \sum_{i=0}^j x_i \la^i =: \mathcal{P}(x;\la) \in \mathcal{N}_r (Q)
\end{equation}
is the inverse of $v(\la) \mapsto C_0 (v(\la))$.

Part (b). From part (a), it is obvious that if the smallest right minimal index of $Q(\la)$ is $j$, then $C_{j-1} (Q(\la))$ has full column rank but $C_{j} (Q(\la))$ does not. The converse also follows from part (a) by taking into account that if $C_{j-1} (Q(\la))$ has full column rank, then $C_{j-2} (Q(\la)), \ldots, \allowbreak C_{0} (Q(\la))$ have also full column ranks. Therefore, $\mathcal{N}_r(C_{j-1} (Q(\la))) = \{0\}, \ldots , \mathcal{N}_r(C_{1} (Q(\la))) = \{0\}, \mathcal{N}_r(C_{0} (Q(\la))) = \{0\}$ and part (a) implies that $\mathcal{N}_r (Q)$ does not include vectors different from $0$ of degree less than $j$, but does include vectors of degree $j$ because $C_j (Q(\la))$ does not have full column rank and so $\mathcal{N}_r(C_{j} (Q(\la))) \ne \{0\}$.

The proof of part (c) requires more work. Let $\{v^{(1)}, \ldots , v^{(p)} \}$ be a basis of $\mathcal{N}_r (C_j (Q(\la)))$ and consider, according to \eqref{eq:inversebijection}, the vector polynomials $\mathcal{P}(v^{(k)};\la) = \sum_{i=0}^j v^{(k)}_i \la^i \in \mathcal{N}_r (Q)$ for $k=1,\ldots ,p$. Note that
$\mathcal{P}(v^{(k)};\la) \ne 0$, because $v^{(k)} \ne 0$, and that $\deg (\mathcal{P}(v^{(k)};\la)) = j$, because otherwise $Q(\la)$ would have right minimal indices smaller than $j$. The result follows from proving that $\mathcal{P}(v^{(1)};\la), \ldots, \mathcal{P}(v^{(p)};\la)$ are linearly independent. We prove this by contradiction.
Assume that there exists a linear combination
\[
a_1(\lambda)\mathcal{P}(v^{(1)};\la)+a_2(\lambda)\mathcal{P}(v^{(2)};\la)+
\cdots+a_p(\lambda) \mathcal{P}(v^{(p)};\la) = 0,
\]
where, without loss of generality, we assume that $a_1(\lambda),\hdots,a_p(\lambda)$ are scalar polynomials not all equal to zero (if they were rational functions we may multiply the equation above by their least common denominator).
The coefficient of the highest power in the equation above satisfies
\[
c_1 v^{(1)}_j + c_2 v^{(2)}_j + \cdots + c_p v^{(p)}_j = 0,
\]
for some constants $c_1,c_2,\hdots,c_p$, where at least one of them  is nonzero.
Then, let us define the polynomial vector $q(\lambda):= \sum_{k=1}^p c_k \mathcal{P}(v^{(k)};\la)$.
Notice that $q(\lambda)\in\mathcal{N}_r(Q)$ and that $\deg (q(\lambda))< j$.
Then $q(\lambda)=0$, because otherwise $Q(\la)$ would have right minimal indices smaller than $j$, which implies $\sum_{k=1}^p c_k v^{(k)} =0$.
This is a contradiction since $\{v^{(1)},v^{(2)},\hdots,v^{(p)}\}$ is a linearly independent set of vectors.
\end{proof}

Next, we study when arbitrary pencils with the same size as the $(2,1)$-block of $\EL{} + \Delta \widetilde{\mathcal{L}} (\la)$ in \eqref{eq:constant_reduction_big} are the corresponding block of a strong block minimal bases pencil.

\begin{theorem} \label{thm:pencilminconv} Let $A + \la B \in \FF[\la]^{\e n \times (\e + 1) n}$ and let $C_s(A+\la B)$, $s=0,1,2,...$, be the sequence of convolution matrices of $A + \la B$. Then, $A+\la B$ is a minimal basis with all its row degrees equal to $1$ and with all the row degrees of any minimal basis dual to it equal to $\e$ if and only if $C_{\e -1}(A + \la B) \in \FF^{\e (\e + 1) n \times \e (\e + 1) n}$ is nonsingular and $C_{\e}(A + \la B) \in \FF^{\e (\e + 2) n \times (\e + 1)^2 n}$ has full row rank.
\end{theorem}

\begin{proof} Bear in mind that the right minimal indices of a minimal basis are the row degrees of any minimal basis dual to it. First, assume that $A+\la B$ is a minimal basis with all its row degrees equal to $1$ and with all the row degrees of any minimal basis dual to it equal to $\e$. Then, the complete eigenstructure of $A + \la B$ consists of only $n$ right minimal indices equal to $\e$, by Lemma \ref{lemm:eigenstminbasis}. From Lemma \ref{lemm:nullspaconv}(b), we get that $C_{\e -1}(A + \la B)$ has full column rank and, since it is square, it must be nonsingular. From Lemma \ref{lemm:nullspaconv}(c), we get that $n \geq \dim \mathcal{N}_r (C_{\e}(A + \la B)) = (\e + 1)^2 n - \rank (C_{\e}(A + \la B))$, which implies that
$\rank (C_{\e}(A + \la B)) \geq (\e + 1)^2 n -n = \e (\e + 2) n$ and, finally, that
$\rank (C_{\e}(A + \la B)) = \e (\e + 2) n$, because $C_{\e}(A + \la B)$ has $\e (\e + 2) n$ rows.

Next, assume that $C_{\e -1}(A + \la B)$ is nonsingular and $C_{\e}(A + \la B)$ has full row rank. Therefore, $\dim \mathcal{N}_r (C_{\e}(A + \la B)) = (\e + 1)^2 n - \rank (C_{\e}(A + \la B)) = (\e + 1)^2 n - \e (\e + 2) n = n$. From Lemma \ref{lemm:nullspaconv}(b), we get that the smallest right minimal index of $A+ \la B$ is $\e$, and from Lemma \ref{lemm:nullspaconv}(c), we get that $A+ \la B$ has at least $n$ right minimal indices equal to $\e$. Also note that the degree of $A + \la B$ must be $1$, since otherwise its minimal indices would be all equal to zero. Combining this information with the index sum theorem \cite[Theorem 6.5]{de2014spectral} applied to $A+ \la B$ and with the obvious bound $\e n \geq \rank (A+ \la B)$, we get
\begin{equation} \label{eq:indexsuminequalities}
n \e \geq \rank (A+ \la B) \geq n \e + \delta(A+\la B) + \mu_{left}(A+\la B),
\end{equation}
where $\delta(A+\la B)$ is the sum of the degrees of all the elementary divisors (finite and infinite) of $A+ \la B$ and $\mu_{left}(A+\la B)$ is the sum of the left minimal indices of $A+ \la B$.  The inequalities \eqref{eq:indexsuminequalities} imply that $\rank (A+ \la B) = n \e$ and that $A+\la B$ has no elementary divisors at all. Moreover, $\rank (A+ \la B) = n \e$ implies that $A+ \la B$ has no left minimal indices and that it has exactly $n$ right minimal indices. Therefore, the complete eigenstructure of $A+ \la B$ consists of only $n$ right minimal indices equal to $\e$, which implies, by Lemma \ref{lemm:eigenstminbasis}, that  $A+\la B$ is a minimal basis with all its row degrees equal to $1$ and with all the row degrees of any minimal basis dual to it equal to $\e$.
\end{proof}

Theorem \ref{thm:polyminconv} is a counterpart of the previous result which is valid for matrix polynomials that may be minimal bases dual to the pencils considered in Theorem \ref{thm:pencilminconv}. The proof of Theorem \ref{thm:polyminconv} is omitted, since it is very similar to that of Theorem \ref{thm:pencilminconv} and is based again on Lemmas \ref{lemm:eigenstminbasis} and \ref{lemm:nullspaconv}.

\begin{theorem} \label{thm:polyminconv} Let $Q(\la) = \sum_{i=0}^\e Q_i \la^i  \in \FF[\la]^{n \times (\e + 1) n}$ and let $C_s(Q(\la))$, $s=0,1,2,...$, be the sequence of convolution matrices of $Q(\la)$. Then, $Q(\la)$ is a minimal basis with all its row degrees equal to $\e$ and with all the row degrees of any minimal basis dual to it equal to $1$ if and only if $C_{0}(Q(\la)) \in \FF^{(\e + 1) n \times (\e + 1) n}$ is nonsingular and $C_{1}(Q(\la)) \in \FF^{(\e + 2) n \times 2(\e + 1) n}$ has full row rank.
\end{theorem}

Theorems \ref{thm:pencilminconv} and \ref{thm:polyminconv} have established the characterizations of a minimal basis with all its row degrees equal and with all the row degrees of any minimal basis dual to it also equal that are needed in this paper. We now return to our perturbation problem for $\EL{} + \Delta \widetilde{\mathcal{L}} (\la)$ in \eqref{eq:constant_reduction_big}. In Theorem \ref{thm:unperturbed21block}, we give some properties of the unperturbed $(2,1)$-block of $\EL{}$, that is, $L_\e (\la) \otimes I_n$, and its dual minimal basis $\Lambda_\e (\la)^T \otimes I_n$.  The proof is given in  Appendix \ref{sec:proof-sigmamin2}.

\begin{theorem} \label{thm:unperturbed21block} Let $L_\e (\la)$ and $\Lambda_\e (\la)^T$ be the pencil and the row vector polynomial defined in \eqref{eq:Lk} and \eqref{eq:Lambda}, respectively. Then the following statements hold.
\begin{enumerate}
\item[\rm (a)] $C_{\e -1}(L_\e (\la) \otimes I_n) \in \FF^{\e (\e + 1) n \times \e (\e + 1) n}$ is nonsingular and $C_{\e}(L_\e (\la) \otimes I_n) \in \FF^{\e (\e + 2) n \times (\e + 1)^2 n}$ has full row rank.
\item[\rm (b)] $C_{0}(\Lambda_\e (\la)^T \otimes I_n) = I_{(\e + 1) n}$ and, therefore, is nonsingular, and $C_{1}(\Lambda_\e (\la)^T \otimes I_n) \in \FF^{(\e + 2) n \times 2(\e + 1) n}$ has full row rank.
\item[\rm (c)] $\displaystyle  \sigma_{\min} (C_{\e -1}(L_\e (\la) \otimes I_n)) = \sigma_{\min}(C_{\e}(L_\e (\la) \otimes I_n)) = 2 \sin \frac{\pi}{(4\e +2)} \geq \frac{3}{2(\e+1)}$.
\item[\rm (d)] $\sigma_{\min}(C_{0}(\Lambda_\e (\la)^T \otimes I_n)) =  \sigma_{\min}(C_{1}(\Lambda_\e (\la)^T \otimes I_n)) = 1.$
\end{enumerate}
\end{theorem}

\smallskip

As a corollary of Theorem \ref{thm:pencilminconv} and Theorem \ref{thm:unperturbed21block}(a)-(c), we obtain the following perturbation result for the $(2,1)$-block of $\EL{} + \Delta \widetilde{\mathcal{L}} (\la)$ in \eqref{eq:constant_reduction_big}.

\begin{corollary} \label{cor:21blockpert} Let $\Delta \widetilde{\mathcal{L}}_{21} (\la)$ be any pencil of size $\e n \times (\e + 1) n$ such that
\begin{equation} \label{eq:21pertbound}
\|\Delta \widetilde{\mathcal{L}}_{21} (\la) \|_F <  \frac{3}{2(\e +1)^\frac32}.
\end{equation}
Then, $L_{\e}(\lambda)\otimes I_n+\Delta \widetilde{\mathcal{L}}_{21}(\lambda)$ is a minimal basis with all its row degrees equal to $1$ and with all the row degrees of any minimal basis dual to it equal to $\e$.
\end{corollary}

\begin{proof} Observe that \eqref{eq:21pertbound} implies that $\|C_{\e -1} (\Delta \widetilde{\mathcal{L}}_{21} (\la))\|_2 \leq \|C_{\e -1} (\Delta \widetilde{\mathcal{L}}_{21} (\la))\|_F \allowbreak = \sqrt{\e} \, \|\Delta \widetilde{\mathcal{L}}_{21} (\la)\|_F <
	\frac{3}{2(\e +1)} \leq \sigma_{\min} (C_{\e -1}(L_\e (\la) \otimes I_n))$, where we have used Theorem \ref{thm:unperturbed21block}(c). Therefore, $C_{\e -1}(L_\e (\la) \otimes I_n + \Delta \widetilde{\mathcal{L}}_{21} (\la)) = C_{\e -1}(L_\e (\la) \otimes I_n) + C_{\e -1}(\Delta \widetilde{\mathcal{L}}_{21} (\la))$ is nonnsingular, as a consequence of Theorem \ref{thm:unperturbed21block}(a) and Weyl's perturbation theorem for singular values \cite[Theorem 3.3.16]{Horn}. An analogous argument proves that $C_{\e}(L_\e (\la) \otimes I_n + \Delta \widetilde{\mathcal{L}}_{21} (\la))$ has full row rank. The result follows from Theorem \ref{thm:pencilminconv}.
\end{proof}

As a corollary of Theorem \ref{thm:polyminconv} and Theorem \ref{thm:unperturbed21block}(b)-(d), we obtain the following perturbation result for the minimal basis dual to $L_\e (\la) \otimes I_n$.

\begin{corollary} \label{cor:dual21blockpert} Let $\Delta R_\e (\la)^T$ be a matrix polynomial of size $n \times (\e + 1) n$, grade $\e$, and such that
\begin{equation} \label{eq:dual21pertbound}
\|\Delta R_\e (\la) \|_F < \frac{1}{\sqrt{2}}.
\end{equation}
Then, $\Lambda_{\e}(\lambda)^T\otimes I_n+\Delta R_\e (\lambda)^T$ is a minimal basis with all its row degrees equal to $\e$ and with all the row degrees of any minimal basis dual to it equal to $1$.
\end{corollary}

\begin{proof} Observe that \eqref{eq:dual21pertbound} implies that $\|C_{1} (\Delta R_\e (\la)^T)\|_2 \leq \|C_{1} (\Delta R_\e (\la)^T)\|_F = \allowbreak \sqrt{2}\, \|\Delta R_\e (\la)^T\|_F < 1 = \sigma_{\min} (C_{1}(\Lambda_{\e}(\lambda)^T\otimes I_n))$, where we have used Theorem \ref{thm:unperturbed21block}(d). Therefore, $C_{1}(\Lambda_\e (\la)^T \otimes I_n + \Delta R_\e (\la)^T) = C_{1}(\Lambda_\e (\la)^T \otimes I_n) + C_{1}(\Delta R_\e (\la)^T)$ has full row rank, as a consequence of Theorem \ref{thm:unperturbed21block}(b) and Weyl's perturbation theorem for singular values. An analogous argument proves that $C_{0}(\Lambda_\e (\la)^T \otimes I_n + \Delta R_\e (\la)^T)$ is nonsingular. The result follows from Theorem \ref{thm:polyminconv}.
\end{proof}

Now, we are in the position of proving the main result of this section.

\begin{theorem} \label{thm:finalofstep2} Let $L_\e (\la)$ and $\Lambda_\e (\la)^T$ be the pencil and the row vector polynomial defined in \eqref{eq:Lk} and \eqref{eq:Lambda}, respectively, and let $\Delta \widetilde{\mathcal{L}}_{21} (\la)$ be any pencil of size $\e n \times (\e + 1) n$ such that
\begin{equation} \label{eq:final21pertbound}
\|\Delta \widetilde{\mathcal{L}}_{21} (\la) \|_F <   \frac{1}{2(\e +1)^{3/2}}.
\end{equation}
Then, there exists a matrix polynomial $\Delta R_\e (\lambda)^T$ with size $n \times (\e + 1) n$ and grade $\e$ such that
\begin{enumerate}
\item[\rm (a)]  $L_{\e}(\lambda)\otimes I_n+\Delta \widetilde{\mathcal{L}}_{21}(\lambda)$ and $\Lambda_{\e}(\lambda)^T\otimes I_n+\Delta R_\e (\lambda)^T$ are dual minimal bases, with all the row degrees of the former equal to $1$ and with all the row degrees of the latter equal to $\e$, and
\smallskip
\item[\rm (b)] $\displaystyle \|\Delta R_\e (\la) \|_F \leq
 \sqrt{2} \, (\e+1) \, \|\Delta \widetilde{\mathcal{L}}_{21} (\la) \|_F \,
<  \frac{1}{\sqrt{2}}$.
\end{enumerate}
\end{theorem}

\begin{proof} The hypothesis \eqref{eq:final21pertbound} implies
$\|\Delta \widetilde{\mathcal{L}}_{21} (\la) \|_F < 3 / (2 (\e +1)^{3/2})$. Therefore, from Corollary \ref{cor:21blockpert}, we get that $L_{\e}(\lambda)\otimes I_n+\Delta \widetilde{\mathcal{L}}_{21}(\lambda)$ is a minimal basis with all its row degrees equal to $1$ and with all the row degrees of any minimal basis dual to it equal to $\e$, and, according to Theorem \ref{thm:pencilminconv}, we also have that $C_\e (L_{\e}(\lambda)\otimes I_n+\Delta \widetilde{\mathcal{L}}_{21}(\lambda))$ has full row rank. Using this fact, the goal of the rest of the proof is to show that there exists a matrix polynomial $\Delta R_\e (\lambda)^T$ with grade $\e$, that satisfies the bound in Theorem \ref{thm:finalofstep2}(b), and such that
\begin{equation} \label{eq:1prfinalstep2}
(L_{\e}(\lambda)\otimes I_n+\Delta \widetilde{\mathcal{L}}_{21}(\lambda)) \, (\Lambda_{\e}(\lambda)\otimes I_n+\Delta R_\e (\lambda) ) = 0 \, .
\end{equation}
Once this is proved, the proof of Theorem \ref{thm:finalofstep2} concludes by the application of Corollary \ref{cor:dual21blockpert}.

Since $(L_{\e}(\lambda)\otimes I_n) \, (\Lambda_{\e}(\lambda)\otimes I_n) = 0$, the equation \eqref{eq:1prfinalstep2} is equivalent to the following linear equation for $\Delta R_\e (\lambda)$
\begin{equation} \label{eq:2prfinalstep2}
(L_{\e}(\lambda)\otimes I_n+\Delta \widetilde{\mathcal{L}}_{21}(\lambda)) \, (\Delta R_\e (\lambda) ) = -\Delta \widetilde{\mathcal{L}}_{21}(\lambda) \, (\Lambda_{\e}(\lambda)\otimes I_n)  \, .
\end{equation}
Both sides of \eqref{eq:2prfinalstep2} have grade $\e + 1$, therefore, by using convolution matrices, \eqref{eq:2prfinalstep2} is equivalent to $C_0((L_{\e}(\lambda)\otimes I_n+\Delta \widetilde{\mathcal{L}}_{21}(\lambda)) \, (\Delta R_\e (\lambda) )) = -C_0(\Delta \widetilde{\mathcal{L}}_{21}(\lambda) \, (\Lambda_{\e}(\lambda)\otimes I_n))$, which in turn, by using \eqref{eq:fundpropconvolution}, is equivalent to
\begin{equation} \label{eq:3prfinalstep2}
C_\e (L_{\e}(\lambda)\otimes I_n+\Delta \widetilde{\mathcal{L}}_{21}(\lambda)) \, \, C_0 (\Delta R_\e (\lambda) ) = -C_0 (\Delta \widetilde{\mathcal{L}}_{21}(\lambda) \, (\Lambda_{\e}(\lambda)\otimes I_n))  \, .
\end{equation}
Observe that \eqref{eq:3prfinalstep2} is a consistent linear system for the unknown $ C_0 (\Delta R_\e (\lambda) )$, since $C_\e (L_{\e}(\lambda)\otimes I_n+\Delta \widetilde{\mathcal{L}}_{21}(\lambda))$ has full row rank, with minimum Frobenius norm solution
\begin{equation} \label{eq:4prfinalstep2}
 C_0 (\Delta R_\e (\lambda) ) = - C_\e (L_{\e}(\lambda)\otimes I_n+\Delta \widetilde{\mathcal{L}}_{21}(\lambda))^\dagger \, \,C_0 (\Delta \widetilde{\mathcal{L}}_{21}(\lambda) \, (\Lambda_{\e}(\lambda)\otimes I_n))  \, .
\end{equation}
From \eqref{eq:4prfinalstep2}, we get the bound
\begin{align} \nonumber
\|C_0 (\Delta R_\e (\lambda) ) \|_F & \leq  \| C_\e (L_{\e}(\lambda)\otimes I_n+\Delta \widetilde{\mathcal{L}}_{21}(\lambda))^\dagger \|_2 \, \, \|C_0 (\Delta \widetilde{\mathcal{L}}_{21}(\lambda) \, (\Lambda_{\e}(\lambda)\otimes I_n))\|_F \\
& = \frac{1}{\sigma_{\min} (C_\e (L_{\e}(\lambda)\otimes I_n+\Delta \widetilde{\mathcal{L}}_{21}(\lambda)))}\, \, \|C_0 (\Delta \widetilde{\mathcal{L}}_{21}(\lambda) \, (\Lambda_{\e}(\lambda)\otimes I_n))\|_F \, . \label{eq:5prfinalstep2}
\end{align}
In the rest of the proof, the two factors in the right-hand side of \eqref{eq:5prfinalstep2} are bounded. For bounding the first factor, we use Theorem \ref{thm:unperturbed21block}(c) and \eqref{eq:final21pertbound} as follows:
\begin{align} \nonumber
\frac{1}{\sigma_{\min} (C_\e (L_{\e}(\lambda)\otimes I_n+\Delta \widetilde{\mathcal{L}}_{21}(\lambda)))} & \leq \frac{1}{\sigma_{\min} (C_\e (L_{\e}(\lambda)\otimes I_n)) - \|C_\e (\Delta \widetilde{\mathcal{L}}_{21}(\lambda))\|_2} \\ \nonumber
& \leq \frac{1}{\sigma_{\min} (C_\e (L_{\e}(\lambda)\otimes I_n)) - \|C_\e (\Delta \widetilde{\mathcal{L}}_{21}(\lambda))\|_F} \\ \nonumber
& \leq \frac{1}{\frac{3}{2 (\e +1)} - \sqrt{\e+1} \, \| \Delta \widetilde{\mathcal{L}}_{21}(\lambda))\|_F} \\ \label{eq:6prfinalstep2}
& \leq \frac{1}{\frac{3}{2(\e +1)} - \frac{1}{2(\e +1)}} = (\e +1) \,.
\end{align}
For bounding the second factor of \eqref{eq:5prfinalstep2}, we use Lemma \ref{lemma:normsproducts}(d) with $d=1$ as follows:
\begin{equation}
\|C_0 (\Delta \widetilde{\mathcal{L}}_{21}(\lambda) \, (\Lambda_{\e}(\lambda)\otimes I_n))\|_F  = \|\Delta \widetilde{\mathcal{L}}_{21}(\lambda) \, (\Lambda_{\e}(\lambda)\otimes I_n)\|_F  \leq \sqrt{2} \, \| \Delta \widetilde{\mathcal{L}}_{21}(\lambda) \|_F
\, . \label{eq:7prfinalstep2}
\end{equation}
Finally, by combining (\ref{eq:5prfinalstep2}, \ref{eq:6prfinalstep2}, \ref{eq:7prfinalstep2}), the following bound is obtained
\[
\|\Delta R_\e (\lambda)\|_F =  \|C_0 (\Delta R_\e (\lambda) )\|_F \leq \sqrt{2} (\e+1)
 \, \|\Delta \widetilde{\mathcal{L}}_{21} (\la) \|_F \le \frac{1}{\sqrt{2(\e+1})},
\]
and the proof is finished.
\end{proof}

Theorem \ref{thm:finalofstep2} can be applied with $\e$ replaced by $\eta$ and $I_n$ replaced by $I_m$, i.e., to the transpose of the $(1,2)$-block of $\mathcal{L}(\lambda)+\Delta \widetilde{\mathcal{L}}(\lambda)$ in \eqref{eq:constant_reduction_big}. This allows us to state, as a corollary of Theorem \ref{thm:finalofstep2}, the final conclusion of this section in Theorem \ref{thm:corfinalofstep2}.

\begin{theorem} \label{thm:corfinalofstep2}  Let $\mathcal{L}(\lambda)+\Delta \widetilde{\mathcal{L}}(\lambda)$ be the pencil in \eqref{eq:constant_reduction_big} and let $d = \e + \eta + 1$.
If
\[
\max\{\|\Delta \widetilde{\mathcal{L}}_{21} (\la) \|_F , \|\Delta \widetilde{\mathcal{L}}_{12} (\la) \|_F \} < \frac{1}{2 \, d^{3/2}},
\]
then $\mathcal{L}(\lambda)+\Delta \widetilde{\mathcal{L}}(\lambda)$ is a strong block minimal bases pencil. Moreover, there exist matrix polynomials $\Delta R_\e (\lambda)^T$ and $\Delta R_\eta (\lambda)^T$ of grades $\e$ and $\eta$, respectively, such that $\Lambda_{\e}(\lambda)^T\otimes I_n+\Delta R_\e (\lambda)^T$ is a minimal basis dual to the $(2,1)$-block of $\mathcal{L}(\lambda)+\Delta \widetilde{\mathcal{L}}(\lambda)$ with all its row degrees equal to $\e$, $\Lambda_{\eta}(\lambda)^T\otimes I_m+\Delta R_\eta (\lambda)^T$ is a minimal basis dual to the transpose of the $(1,2)$-block of $\mathcal{L}(\lambda)+\Delta \widetilde{\mathcal{L}}(\lambda)$ with all its row degrees equal to $\eta$, and
\[
\max \{ \|\Delta R_\e (\la) \|_F , \|\Delta R_\eta (\la) \|_F \} \leq
\sqrt{2} \, d \, \max\{\|\Delta \widetilde{\mathcal{L}}_{21} (\la) \|_F , \|\Delta \widetilde{\mathcal{L}}_{12} (\la) \|_F \} \,
<  \frac{1}{\sqrt{2}} \, .
\]
\end{theorem}

The bound $\max \{ \|\Delta R_\e (\la) \|_F , \|\Delta R_\eta (\la) \|_F \} < 1/\sqrt{2}$ in the equation above has the main purpose to emphasize that the hypotheses of Corollary \ref{cor:dual21blockpert} hold. In addition, it motivates the assumptions in Lemmas \ref{lemm:Pdelta1} and \ref{lemm:Pdelta2} that allow us to get rid of nonlinear terms in bounding $\|\Delta P(\la)\|_F$.

\subsection{Third step: Mapping perturbations to a block Kronecker pencil onto the matrix polynomial} In this section, we combine the results in Sections \ref{sec:firststep} and \ref{sec:secondstep} to obtain our main backward error (or perturbation) results, that is, Theorem \ref{thm:perturbation} for general block Kronecker pencils as in \eqref{eq:linearization_general} and Theorem \ref{thm:perturbation2} for degenerate block Kronecker pencils in which either $\e = 0$ or $\eta =0$, that is, in which one of the anti-diagonal blocks and the zero block are not present. According to Remark \ref{rem:2empty} both cases require somewhat different treatments which makes the discussion longer.

The proofs of Theorems \ref{thm:perturbation} and \ref{thm:perturbation2} are direct consequences of previous results, but require some delicate (although elementary) norm manipulations which are simplified if the technical Lemmas \ref{lemm:Pdelta1} and \ref{lemm:Pdelta2} are stated in advance. The relevance of these lemmas comes from the fact that the strong block minimal bases pencil $\mathcal{L}(\lambda)+\Delta \widetilde{\mathcal{L}}(\lambda)$ in Theorem \ref{thm:corfinalofstep2} is a strong linearization of the matrix polynomial in \eqref{eq:polyperturbed3step}, as a consequence of Theorem \ref{thm:blockminlin}. The numerical constants appearing in Lemmas \ref{lemm:Pdelta1} and \ref{lemm:Pdelta2}, and in the rest of the analysis, are not optimal but allow us to keep the analysis simple.

\begin{lemma} \label{lemm:Pdelta1} Let $P(\la)$ and $P(\la) + \Delta P(\la)$ be the matrix polynomials in \eqref{eq:Pinsectionerrors} and \eqref{eq:polyperturbed3step}, respectively. If the matrix polynomials $\Delta R_\e (\lambda)$ and $\Delta R_\eta (\lambda)$ of grades $\e$ and $\eta$, respectively, satisfy $\|\Delta R_\e (\lambda)\|_F < 1/\sqrt{2}$ and $\|\Delta R_\eta (\lambda)\|_F < 1/\sqrt{2}$, then
\[
\|\Delta P(\la)\|_F \leq \sqrt{d} \left(5 \|\Delta \EL{11}\|_F + 4  \|\la M_1 + M_0\|_F \max\{\|\Delta R_\e (\lambda)\|_F,\|\Delta R_\eta (\lambda)\|_F\}\right) \, ,
\]
where $d = \e + \eta + 1$.
\end{lemma}
\begin{proof} For brevity, we use in this proof the notation $\Lambda_{\e n}^T := \Lambda_{\e}(\lambda)^T\otimes I_n$ and omit the dependence on $\la$ of some matrix polynomials. From \eqref{eq:Pinsectionerrors} and \eqref{eq:polyperturbed3step}, we get that
\begin{align} \nonumber
\Delta P (\la) = & \Delta R_\eta^T(\la M_1 + M_0) \Lambda_{\e n} + \Lambda_{\eta m}^T \Delta \mathcal{L}_{11} \Lambda_{\e n} + \Delta R_\eta^T  \Delta\mathcal{L}_{11}  \Lambda_{\e n} \\ \nonumber & + \Lambda_{\eta m}^T (\la M_1 + M_0) \Delta R_\e +
\Delta R_{\eta}^T (\la M_1 + M_0) \Delta R_\e \\ & + \Lambda_{\eta m}^T \Delta\mathcal{L}_{11}  \Delta R_\e + \Delta R_\eta^T \Delta\mathcal{L}_{11}  \Delta R_\e \, . \label{eq:longDeltaP}
\end{align}
The result follows from bounding the Frobenius norm of each of the terms in the right-hand side of \eqref{eq:longDeltaP}. For this purpose, Lemma \ref{lemma:normsproducts} is used and, in addition, the inequalities $\|\Delta R_\e (\lambda)\|_F < 1/\sqrt{2}$ and $\|\Delta R_\eta (\lambda)\|_F < 1/\sqrt{2}$ are used in those terms that are not linear in $\Delta \mathcal{L}_{11} (\la)$, $\Delta R_\e (\lambda)$, and $\Delta R_\eta (\lambda)$ for bounding them with linear terms. Let us show how to bound only one of the terms in \eqref{eq:longDeltaP}, since the rest are bounded via similar procedures,
\begin{align*}
\|\Delta R_{\eta}^T (\la M_1 + M_0) \Delta R_\e\|_F &\leq \sqrt{d} \, \|\Delta R_{\eta}\|_F \|(\la M_1 + M_0) \Delta R_\e\|_F \\ &\leq \sqrt{2d} \, \|\Delta R_{\eta}\|_F \|\la M_1 + M_0\|_F \|\Delta R_\e\|_F \\ & \leq \sqrt{d} \, \|\la M_1 + M_0\|_F \|\Delta R_\e\|_F  \, .
\end{align*}
\end{proof}

Lemma \ref{lemm:Pdelta2} is the counterpart of Lemma \ref{lemm:Pdelta1} that is needed to deal with perturbations of degenerate block Kronecker pencils. The proof of Lemma \ref{lemm:Pdelta2} is omitted because it is similar to, and simpler than, the one of Lemma \ref{lemm:Pdelta1}.
\begin{lemma} \label{lemm:Pdelta2}
\begin{enumerate}
\item[\rm (a)] Let us consider the matrix polynomials
\begin{align*} P(\lambda) &= (\lambda M_1+M_0)(\Lambda_{\e}(\lambda)\otimes I_n), \\
P(\lambda) + \Delta P(\lambda) & = \left( \la M_1 +M_0 +\Delta \mathcal{L}_{11}(\lambda)  \right) \left( \Lambda_{\e}(\lambda)\otimes I_n+\Delta R_\e (\lambda)\right) .\end{align*}
If the matrix polynomial $\Delta R_\e (\lambda)$ satisfies $\|\Delta R_\e (\lambda)\|_F \allowbreak < 1/\sqrt{2}$, then
$$
\|\Delta P(\la)\|_F \leq 3 \, \|\Delta \EL{11}\|_F + \sqrt{2} \, \|\la M_1 + M_0\|_F \, \|\Delta R_\e (\lambda)\|_F \,.
$$

\item[\rm (b)] Let us consider the matrix polynomials
\begin{align*}
P(\lambda) & = (\Lambda_\eta(\lambda)^T\otimes I_m)(\lambda M_1+M_0), \\
P(\lambda) + \Delta P(\lambda) & = \left( \Lambda_\eta(\lambda)^T\otimes I_m+\Delta R_\eta (\lambda)^{T}\right) \left( \la M_1 +M_0 +\Delta \mathcal{L}_{11}(\lambda)  \right).
\end{align*}
If the matrix polynomial $\Delta R_\eta (\lambda)$ satisfies $\|\Delta R_\eta (\lambda)\|_F \allowbreak < 1/\sqrt{2}$, then
$$
\|\Delta P(\la)\|_F \leq 3 \,  \|\Delta \EL{11}\|_F + \sqrt{2} \, \|\la M_1 + M_0\|_F \, \|\Delta R_\eta (\lambda)\|_F \,.
$$
\end{enumerate}
\end{lemma}

Next, we state and prove the main results of Section \ref{sec:expansion} concerning perturbations of the block Kronecker pencils defined and studied in Section \ref{sec:linearization}. Recall that these pencils are strong linearizations of prescribed matrix polynomials enjoying constant shifting recovery properties for the minimal indices (see Theorems \ref{thm:strong} and \ref{thm:givenPblockKron}).

\begin{theorem}\label{thm:perturbation}
Let $P(\lambda) = \sum_{i=0}^d P_i \la^i \in \FF[\la]^{m\times n}$ and let $\EL{}$ be an $(\e,n,\eta,m)$-block Kronecker pencil with $d = \e + \eta + 1$ such that
$P(\lambda) = (\Lambda_\eta(\lambda)^T\otimes I_m)(\lambda M_1+M_0)(\Lambda_{\e}(\lambda)\otimes I_n)$, where $\la M_1 + M_0$ is the $(1,1)$-block in the natural partition of $\EL{}$ and $\Lambda_k (\la)$ is the vector polynomial in \eqref{eq:Lambda}. If $\Delta \EL{}$ is any pencil with the same size as $\EL{}$ and such that
\begin{equation} \label{eq:Lfinalbound}
\|\Delta \EL{}\|_F < (\sqrt{2}-1)^2 \, \frac{1}{d^{5/2}} \, \frac{1}{1 + \|\la M_1 + M_0\|_F},
\end{equation}
then $\EL{} + \Delta \EL{}$ is a strong linearization of a matrix polynomial $P(\la) + \Delta P(\la)$ with grade $d$ and such that
\[
\frac{\|\Delta P(\la)\|_F}{\|P(\la)\|_F} \leq 14\,  d^{5/2} \frac{\|\EL{} \|_F}{\|P(\la)\|_F} \, (1+ \|\la M_1 + M_0\|_F + \|\la M_1 + M_0\|_F^2) \,
\frac{\|\Delta \EL{} \|_F}{\|\EL{} \|_F} \, .
\]
In addition, the right minimal indices of $\EL{} + \Delta \EL{}$ are those of $P(\la) + \Delta P(\la)$ shifted by $\e$, and the left minimal indices of $\EL{} + \Delta \EL{}$ are those of $P(\la) + \Delta P(\la)$ shifted by $\eta$.
\end{theorem}

\begin{proof}
Observe that the condition \eqref{eq:Lfinalbound} implies that \eqref{eq:boundL1} holds. Therefore, we can apply Theorem \ref{thm:finalofstep1} to $\EL{} + \Delta \EL{}$ for proving that it is strictly equivalent to the pencil $\EL{} + \Delta \widetilde{\mathcal{L}} (\la)$ in \eqref{eq:constant_reduction_big} and thus both pencils have the same complete eigenstructures. By combining \eqref{eq:Lfinalbound}, which implies $d \|\Delta \EL{}\|_F < (\sqrt{2}-1)$, with \eqref{eq:boundL12tilde}, we get
the following bound
\begin{align} \label{eq:aux1boundfinal}
\max\{\|\Delta \widetilde{\mathcal{L}}_{21}(\lambda)\|_{F} ,  \|\Delta \widetilde{\mathcal{L}}_{12}(\lambda)\|_{F} \}
&\leq \|\Delta \EL{}\|_F \left(2 + \frac{d}{\sqrt{2}-1} \, \|\la M_1 + M_0\|_F  \right) \, \\
& \leq (\sqrt{2}-1) \, \frac{1}{d^{3/2}} < \frac12  \, \frac{1}{d^{3/2}} \,, \nonumber
\end{align}
which allows us to apply Theorem \ref{thm:corfinalofstep2} to $\EL{} + \Delta \widetilde{\mathcal{L}} (\la)$. Then, $\EL{} + \Delta \widetilde{\mathcal{L}} (\la)$ is a strong block minimal bases pencil which, according to Theorem \ref{thm:blockminlin}, is a strong linearization of the matrix polynomial $P(\la) + \Delta P(\la)$ in \eqref{eq:polyperturbed3step}. Moreover, Theorem \ref{thm:indicesminbaseslin} guarantees that the right minimal indices of $\EL{} + \Delta \widetilde{\mathcal{L}} (\la)$ are those of $P(\la) + \Delta P(\la)$ shifted by $\e$, and that the left minimal indices of $\EL{} + \Delta \widetilde{\mathcal{L}} (\la)$ are those of $P(\la) + \Delta P(\la)$ shifted by $\eta$.
The same holds for $\EL{} + \Delta \EL{}$, since it is strictly equivalent to $\EL{} + \Delta \widetilde{\mathcal{L}} (\la)$.
It only remains to bound $\|\Delta P (\la) \|_F$. For this purpose, we combine Lemma \ref{lemm:Pdelta1} and the bound on
$\max \{ \|\Delta R_\e (\la) \|_F , \|\Delta R_\eta (\la) \|_F \}$ in Theorem \ref{thm:corfinalofstep2}. By using Theorem \ref{thm:corfinalofstep2} and \eqref{eq:aux1boundfinal}, the inequality
\[
\max \{ \|\Delta R_\e (\la) \|_F , \|\Delta R_\eta (\la) \|_F \} \leq \frac{\sqrt{2}}{(\sqrt{2}-1)} \, d^2 \,
\|\Delta \EL{}\|_F \left(1 +  \|\la M_1 + M_0\|_F  \right),
\]
is proved. If this inequality is introduced in the bound of Lemma \ref{lemm:Pdelta1}, then we obtain
\[
\|\Delta P(\la)\|_F \leq 14 \,  d^{5/2} \, \|\Delta \EL{}\|_F \, (1+ \|\la M_1 + M_0\|_F + \|\la M_1 + M_0\|_F^2) \, ,
\]
and the proof concludes.
\end{proof}

Next, we state and prove Theorem \ref{thm:perturbation2}, which is the counterpart of Theorem \ref{thm:perturbation} for degenerate block Kronecker pencils. For brevity, degenerate block Kronecker pencils are called either $(0,n,\eta,m)$-block Kronecker pencils when the second block row in \eqref{eq:linearization_general} is missing or $(\e,n,0,m)$-block Kronecker pencils when the second block column in \eqref{eq:linearization_general} is missing, i.e., they correspond to taking either $\e=0$ or $\eta = 0$. We emphasize that the perturbation bound in Theorem \ref{thm:perturbation2} is smaller than the one in Theorem \ref{thm:perturbation} because performing the strict equivalence \eqref{eq:constant_reduction_big} is not needed in the degenerate case. The most relevant difference in Theorem \ref{thm:perturbation2} with respect to the bound in Theorem \ref{thm:perturbation} is that the term $\|\la M_1 + M_0\|_F^2$ is not present, which is in agreement with the first order results obtained in \cite{deTeran2015imajna} for Fiedler matrices (not pencils) of scalar monic polynomials.

\begin{theorem} \label{thm:perturbation2}
Let $P(\lambda) = \sum_{i=0}^d P_i \la^i \in \FF[\la]^{m\times n}$ and let $\EL{}$ be either a $(0,n,\eta,m)$-block Kronecker pencil with $d = \eta + 1$ such that
$P(\lambda) = (\Lambda_\eta(\lambda)^T\otimes I_m)(\lambda M_1+M_0)$ or an $(\e,n,0,m)$-block Kronecker pencil with $d = \e + 1$ such that
$P(\lambda) = (\lambda M_1+M_0)(\Lambda_\e (\lambda) \otimes I_n)$, where $\la M_1 + M_0$ is the $(1,1)$-block in the natural partition of $\EL{}$ and $\Lambda_k (\la)$ is the vector polynomial in \eqref{eq:Lambda}. If $\Delta \EL{}$ is any pencil with the same size as $\EL{}$ and such that
\begin{equation} \label{eq:Lfinalbound2}
\|\Delta \EL{}\|_F < \frac{1}{2\, d^{3/2}}   \,  ,
\end{equation}
then $\EL{} + \Delta \EL{}$ is a strong linearization of a matrix polynomial $P(\la) + \Delta P(\la)$ with grade $d$ and such that
\[
\frac{\|\Delta P(\la)\|_F}{\|P(\la)\|_F} \leq 2\,  d \, \frac{\|\EL{} \|_F}{\|P(\la)\|_F} \, (1+ \|\la M_1 + M_0\|_F) \,
\frac{\|\Delta \EL{} \|_F}{\|\EL{} \|_F} \, .
\]
In addition, the right minimal indices of $\EL{} + \Delta \EL{}$ are those of $P(\la) + \Delta P(\la)$ shifted by $\e$, and the left minimal indices of $\EL{} + \Delta \EL{}$ are those of $P(\la) + \Delta P(\la)$ shifted by $\eta$, where either $\e =0$ or $\eta =0$.
\end{theorem}

\begin{proof} We simply sketch the proof, since it follows the same ideas as the proof of Theorem \ref{thm:perturbation}. The reader should bear in mind Remark \ref{rem:empty}. In the degenerate case, we can apply Theorem \ref{thm:corfinalofstep2} directly to $\EL{} + \Delta \EL{} = \EL{} + \Delta \widetilde{\mathcal{L}} (\la)$. After that, it only remains to prove the bound on $\|\Delta P(\la)\|_F$. For this purpose, we combine Lemma \ref{lemm:Pdelta2} and the bound on
$\max \{ \|\Delta R_\e (\la) \|_F , \|\Delta R_\eta (\la) \|_F \}$ in Theorem \ref{thm:corfinalofstep2} for obtaining
\[
\|\Delta P (\la) \|_F \leq 2 \, d  \, \|\Delta \EL{} \|_F \, (1 + \|\la M_1 + M_0 \|_F )\, .
\]
This ends the proof.
\end{proof}

Finally, we discuss when Theorems \ref{thm:perturbation} and \ref{thm:perturbation2} guarantee backward stability of complete polynomial eigenproblems solved via the staircase or the QZ algorithms applied to a block Kronecker pencil. We restrict the discussion to nondegenerate block Kronecker pencils, since the obtained conclusions are also valid for the degenerate case. According to our discussion at the beginning of Section \ref{sec:expansion}, to equation \eqref{eq:perturbationbound}, and to Theorem \ref{thm:perturbation}, if
\begin{equation}  \label{eq:C_pl}
C_{P,\mathcal{L}} := 14\,  d^{5/2} \frac{\|\EL{} \|_F}{\|P(\la)\|_F} \, (1+ \|\la M_1 + M_0\|_F + \|\la M_1 + M_0\|_F^2)
\end{equation}
is a moderate number, then the backward stability is guaranteed. From \eqref{eq:C_pl}, it is clear that the following elementary lemma is useful for our discussion.
\begin{lemma} \label{lemm:quotient} Let $P(\lambda) = \sum_{k=0}^d P_k \la^k \in \FF[\la]^{m\times n}$ and let $\EL{}$ be an $(\e,n,\eta,m)$-block Kronecker pencil with $d = \e + \eta + 1$ such that $P(\lambda) = (\Lambda_\eta(\lambda)^T\otimes I_m)(\lambda M_1+M_0)(\Lambda_{\e}(\lambda)\otimes I_n)$. Then:
\begin{enumerate}
\item[\rm (a)] $\displaystyle \frac{\|\EL{} \|_F}{\|P(\la)\|_F} = \sqrt{\displaystyle \left( \frac{\|\la M_1 + M_0\|_F}{\|P(\la)\|_F }\right)^2 +
    \frac{2(n\e + m \eta)}{\|P(\la)\|_F^2}} \geq \frac{1}{\sqrt{2\, d} }$.
\item[\rm (b)] $\displaystyle \|\la M_1 + M_0\|_F \, \geq \, \|P (\la) \|_F /\sqrt{2\, d}$.
\end{enumerate}
\end{lemma}

\begin{proof} The equality in part (a) follows from \eqref{eq:linearization_general} and Definition \ref{def:norm}. The inequality follows from \eqref{eq:condition_coeff}, which implies, for $k=0,1,\ldots, d$,
\begin{align*}
\|P_k \|_F & \leq  \sum_{i+j=\pdeg+2-k} \|[M_1]_{ij}\|_F + \sum_{i+j=\pdeg+1-k} \|[M_0]_{ij}\|_F \\
& \leq \sqrt{2d} \sqrt{\sum_{i+j=\pdeg+2-k} \|[M_1]_{ij}\|_F^2 + \sum_{i+j=\pdeg+1-k} \|[M_0]_{ij}\|_F^2} \, .
\end{align*}
 This in turn implies $\|P (\la) \|_F  \leq \sqrt{2d} \, \|\la M_1 + M_0\|_F$, which is the result in part (b), and gives the inequality in part (a).
\end{proof}

From \eqref{eq:C_pl} and Lemma \ref{lemm:quotient}(a), we see that if $\|P(\la)\|_F \ll 1$, then $C_{P,\mathcal{L}}$ is huge, since $2(n\e + m \eta)/\|P(\la)\|_F^2$ is huge. Moreover, from \eqref{eq:C_pl} and Lemma \ref{lemm:quotient}(b), we see that if $\|P(\la)\|_F \gg 1$, then $C_{P,\mathcal{L}}$ is also huge, since $\|\la M_1 + M_0 \|_F$ is huge and $\|\EL{} \|_F/\|P(\la)\|_F \geq 1/\sqrt{2\, d}$. Therefore, one should scale $P(\la)$ in advance in such a way that $\|P(\la)\|_F = 1$ to have a chance of $C_{P,\mathcal{L}}$ is moderate. But even in this case, $C_{P,\mathcal{L}}$ is large if $\|\la M_1 + M_0\|_F$ is large. This happens, for instance, in the last pencil in Example \ref{ex-blockKron} if the arbitrary matrices $A$ and/or $B$ have huge norms.

As a consequence of the discussion above and Theorems \ref{thm:perturbation} and \ref{thm:perturbation2}, we can state the informal Corollary \ref{cor:FINperturbation}, which establishes sufficient conditions for the backward stability  of the solution of complete polynomial eigenproblems via block Kronecker pencils  (degenerate or not). For the sake of clarity and simplicity any nonessential numerical constant is omitted in Corollary \ref{cor:FINperturbation}.

\begin{corollary}\label{cor:FINperturbation}
Let $P(\lambda) = \sum_{i=0}^d P_i \la^i \in \FF[\la]^{m\times n}$ with $\|P(\la)\|_F = 1$. Let $\EL{}$ be an $(\e,n,\eta,m)$-block Kronecker pencil as in \eqref{eq:linearization_general} with $d = \e + \eta + 1$ and such that
$P(\lambda) = (\Lambda_\eta(\lambda)^T\otimes I_m)(\lambda M_1+M_0)(\Lambda_{\e}(\lambda)\otimes I_n)$. Let $\Delta \EL{}$ be any pencil with the same size as $\EL{}$ and with $\|\Delta \EL{}\|_F$ sufficiently small. If $\|\la M_1 + M_0\|_F \approx \|P(\la)\|_F$, then $\EL{} + \Delta \EL{}$ is a strong linearization of a matrix polynomial $P(\la) + \Delta P(\la)$ with grade $d$ and such that
\begin{equation} \label{eq:informalcor}
\frac{\|\Delta P(\la)\|_F}{\|P(\la)\|_F} \lesssim \,  d^{3} \,\sqrt{m+n} \, \,
\frac{\|\Delta \EL{} \|_F}{\|\EL{} \|_F} \, .
\end{equation}
In addition, the right minimal indices of $\EL{} + \Delta \EL{}$ are those of $P(\la) + \Delta P(\la)$ shifted by $\e$, and the left minimal indices of $\EL{} + \Delta \EL{}$ are those of $P(\la) + \Delta P(\la)$ shifted by $\eta$. In particular, this corollary holds for all permuted Fiedler pencils presented in \cite[Theorem 4.5]{Dopico:2016:BlockKronecker}, since for them $\|\la M_1 + M_0\|_F = \|P(\la)\|_F$.
\end{corollary}

For degenerate block Kronecker pencils, the bound \eqref{eq:informalcor} can be improved as follows: the factor $d^3$ can be replaced by $d^{3/2}$, as a consequence of Theorem \ref{thm:perturbation2}, and $\sqrt{m+n}$ by $\sqrt{m}$ if $\e =0$ or by $\sqrt{n}$ if $\eta =0$, as a consequence of Lemma \ref{lemm:quotient}(a).

\begin{remark} {\rm We emphasize that Corollary \ref{cor:FINperturbation} can be applied also to non-permuted Fiedler pencils, since the Frobenius norm is invariant under permutations and permutations preserve strong linearizations and minimal indices. Therefore, given a Fiedler pencil and a perturbation of it, we can permute both and transform the corresponding perturbation problem into the problem we have solved in this section.
}
\end{remark}

\begin{remark} {\rm Consider that each block-entry of the $(1,1)$-block $\la M_1 + M_0$ of the block Kronecker pencil $\EL{}$ in
Theorems \ref{thm:perturbation} and \ref{thm:perturbation2}, and in Corollary \ref{cor:FINperturbation}, is a linear combination of the coefficients $P_d, \ldots , P_0$ of $P(\la)$ and of some arbitrary matrices. Then, the pencil $\EL{} + \Delta \EL{}$ in
Theorems \ref{thm:perturbation} and \ref{thm:perturbation2}, and in Corollary \ref{cor:FINperturbation}, is strictly equivalent to a block Kronecker pencil $\widehat{\mathcal{L}}(\la)$ {\em with exactly the same structure as $\EL{}$} but for the polynomial $P(\la) + \Delta P(\la)$ instead of $P(\la)$. This means that each block-entry of the $(1,1)$-block of the block Kronecker pencil $\widehat{\mathcal{L}}(\la)$ is the same linear combination of the coefficients $P_d +\Delta P_d, \ldots , P_0 + \Delta P_0$ of $P(\la) + \Delta P(\la)$ and of the same arbitrary matrices as the corresponding block entry of $\EL{}$ is for the coefficients $P_d, \ldots , P_0$ and the same arbitrary matrices.
In particular, if $\EL{}$ is a given permuted Fiedler pencil of $P(\la)$ (see \cite[Theorem 4.5]{Dopico:2016:BlockKronecker}), then $\EL{} + \Delta \EL{}$ is strictly equivalent to the same permuted Fiedler pencil of $P(\la) + \Delta P(\la)$. This result follows from the fact that Theorem \ref{thm:givenPblockKron} guarantees that $\widehat{\mathcal{L}}(\la)$ has the same complete eigenstructure as $\EL{} + \Delta \EL{}$, and so both pencils must be strictly equivalent \cite[Chapter XII]{gantmacher1960theory}. This remark by itself does not prove that the strict equivalence transformations connecting $\widehat{\mathcal{L}}(\la)$  and $\EL{} + \Delta \EL{}$ are small perturbations of identity matrices, despite the fact that $\widehat{\mathcal{L}}(\la)$ and $\EL{} + \Delta \EL{}$ are indeed very close each other. However, it is clear that this remark opens the possibility of proving directly that $\widehat{\mathcal{L}}(\la)$ and $\EL{} + \Delta \EL{}$ are strictly equivalent via nonsingular matrices that are very close to the identity, as it was done in \cite{van1983eigenstructure} for the Frobenius companion linearizations.
}
\end{remark}

\section{Conclusions and future work}\label{sec:conclusions} The new family of strong block minimal bases pencils has been introduced and analyzed.  We have proven in a simple and general way that these pencils are always strong linearizations of matrix polynomials and that their minimal indices and those of the polynomials satisfy constant uniform shifting relationships. These proofs are based on the properties of dual minimal bases---classical tools in multivariable linear system theory that have been used recently in different matrix polynomial eigenproblems. As an immediate corollary of this general theory, we obtain that the same results hold for the subfamily of block Kronecker pencils, which form a wide subclass of block minimal bases pencils easily constructible from the coefficients of a given but general matrix polynomial (general in the sense that it may be square or rectangular, regular or singular). The fundamental property that strong block minimal bases pencils are robust under arbitrary perturbations that are sufficiently small and that preserve the $(2,2)$-zero block allows us to develop a rigorous global backward error analysis of complete polynomial eigenproblems solved via block Kronecker pencils. The key point of the analysis is that although perturbations of block Kronecker pencils destroy the delicate block Kronecker structure, they lead, after some manipulations, to strong block minimal bases pencils with similar properties. The backward error bounds delivered by this analysis enjoy a number of novel features not present so far in the literature as, for instance, the fact that they are finite precise bounds instead of first order big-O bounds.

The results in this work have already motivated considerable research in the area. For instance, they have clarified many of the results that have been published in the last few years on linearizations of matrix polynomials, since it has been proved in \cite{budopereu-2016} that all generalized Fiedler linearizations \cite{p832,bueno-ter-dop-general,DeTeran:2011}, all Fiedler linearizations with repetition \cite{Bueno_structuredstrong,Bueno_palindromiclinearizations,greeks2011}, and all generalized Fiedler linearizations with repetition \cite{bueno2015large} may be transformed through proper permutations into particular strong block minimal bases pencils that can be described very easily; structured versions of the backward error analysis in this paper have been developed for many classes of structured strong block minimal bases linearizations of structured matrix polynomials in \cite{dopevdoor-2017}; in \cite{robol-vandebril-vandooren} particular block minimal bases linearizations have been used to compute efficiently and in a  stable way the zeros of a polynomial that is the sum of two polynomials expressed in two different bases, as well as for solving other challenging numerical problems; extensions of block Kronecker pencils that linearize matrix polynomials expressed in Chebyshev bases have been developed in
\cite{lawrence-perez-cheby}; it has been shown that each strong block minimal bases pencil can be used to construct strong linearizations of rational matrices with non-constant polynomial part \cite{amdomarza-2016}; etc. In addition to these publications, several other ongoing research projects related to block minimal bases pencils are being currently developed by different researchers. They include the extension of the error analysis to other strong block minimal bases linearizations and the generalization of the ideas presented in this work to the context of $\ell$-ifications of matrix polynomials \cite{de2014spectral,DDV-lifications}.

\appendix

\section{The minimal bases of strong block minimal bases pencils} \label{sec:appendixminbases} In this appendix, we state and prove Lemma \ref{lemm:techindminbaslin}, which establishes, first, the relationship between the vectors in the rational right null spaces of any of the strong block minimal bases pencils $\EL{}$ introduced in Definition \ref{def:minlinearizations} and of the corresponding matrix polynomial $Q(\la)$ in \eqref{eq:Qpolinminbaslin}, and, second, the relationship between the right minimal bases of $\EL{}$ and $Q(\la)$. In this paper Lemma \ref{lemm:techindminbaslin} is only used in the proof of Theorem \ref{thm:indicesminbaseslin}, but we emphasize that is very useful for proving the recovery procedures of eigenvectors and minimal bases of block Kronecker pencils in \cite[Section 7]{Dopico:2016:BlockKronecker} and that is a fundamental result in the theory of strong block minimal bases linearizations.

\begin{lemma} \label{lemm:techindminbaslin}
Let $\EL{}$ be a strong block minimal bases pencil as in \eqref{eq:minbaspencil}, let $N_1(\la)$ be a minimal basis dual to $K_1 (\la)$, let $N_2(\la)$ be a minimal basis dual to $K_2 (\la)$, let $Q(\la)$ be the matrix polynomial defined in \eqref{eq:Qpolinminbaslin}, and let $\widehat{N}_2 (\la)$ be the matrix appearing in \eqref{eq:twounimodembed}. Then the following hold:
\begin{enumerate}
\item[\rm (a)] If $h(\la) \in \mathcal{N}_r (Q)$, then
\begin{equation} \label{eq:defz}
z(\la) :=
\begin{bmatrix}
N_1 (\la)^T \\-\widehat{N}_2 (\la) M(\la) N_1 (\la)^T
\end{bmatrix} h(\la) \, \in \mathcal{N}_r (\mathcal{L})\, .
\end{equation}
Moreover, if $0 \ne h(\la) \in \mathcal{N}_r (Q)$ is a vector polynomial, then $z(\la)$ is also a vector polynomial and
\begin{equation} \label{eq:degreeshift1}
\deg (z (\la)) = \deg( N_1 (\la)^T \, h(\la)) = \deg( N_1 (\la)) +  \deg(h(\la)).
\end{equation}
\item[\rm (b)] If $\{h_1(\la), \ldots, h_p(\la)\}$ is a right minimal basis of $Q(\la)$, then
\[
\left\{\begin{bmatrix}
N_1 (\la)^T \\-\widehat{N}_2 (\la) M(\la) N_1 (\la)^T
\end{bmatrix} h_1(\la), \ldots , \begin{bmatrix}
N_1 (\la)^T \\-\widehat{N}_2 (\la) M(\la) N_1 (\la)^T
\end{bmatrix} h_p(\la) \right\}
\]
is a right minimal basis of $\EL{}$.
\end{enumerate}
\end{lemma}

\begin{proof}
(a) It can be checked, via a direct multiplication, that the matrix $X(\la)$ in \eqref{eq:XYZminlin} satisfies $X(\la) = \widehat{N}_2 (\la) M(\la) N_1 (\la)^T$. Then, from \eqref{eq:XYZminlin}, we get that
\[
(U_2(\lambda)^{-T} \oplus I_{m_1}) \, \EL{} \, (U_1(\lambda)^{-1} \oplus I_{m_2})
\begin{bmatrix}
0 \\ I \\ -X(\la)
\end{bmatrix}
= \begin{bmatrix}
0 \\ Q(\la) \\ 0
\end{bmatrix},
\]
where the sizes of the identity and zero blocks are conformable with the partition of the last matrix in \eqref{eq:XYZminlin}. By using the structure of $U_1(\lambda)^{-1} \oplus I_{m_2}$ (recall \eqref{eq:twounimodembed}), the multiplication of the last two factors in the left-hand side of the previous equation leads to
\begin{equation} \label{eq:ins3forrecovery}
(U_2(\lambda)^{-T} \oplus I_{m_1}) \, \EL{} \,
\begin{bmatrix}
N_1(\la)^T \\ -X(\la)
\end{bmatrix}
= \begin{bmatrix}
0 \\ Q(\la) \\ 0
\end{bmatrix}.
\end{equation}
This equation implies that $z(\la) \in \mathcal{N}_r (\mathcal{L})$ if $h(\la) \in \mathcal{N}_r (Q)$, and also that $z(\la)$ is a vector polynomial if $h(\la)$ is, because $N_1(\la)$ and $X(\la)$ are matrix polynomials.

It only remains to prove the degree shift property \eqref{eq:degreeshift1} to conclude the proof of part (a). First, take into account that all the row degrees of the minimal basis $N_1 (\la)$ are equal and that its highest degree coefficient has full row rank. Therefore,
\begin{equation} \label{eq:auxdegreeshift1}
\deg( N_1 (\la)^T \, g(\la)) = \deg( N_1 (\la)) +  \deg(g(\la)) \, ,
\end{equation}
for any vector polynomial $g(\la) \ne 0$. The same argument applied to the minimal basis $K_2 (\la)$ proves that
\begin{equation} \label{eq:auxdegreeshift2}
\deg( K_2 (\la)^T \, y(\la)) = \deg( K_2 (\la)) +  \deg(y(\la)) = 1 +  \deg(y(\la)) \, ,
\end{equation}
for any vector polynomial $y(\la) \ne 0$. Next, observe that
\begin{equation} \label{eq:maxdegrees}
\deg (z(\la)) = \max \{\deg( N_1 (\la)^T h(\la)) \, , \, \deg (X(\la) h(\la))\} \, .
\end{equation}
Therefore \eqref{eq:degreeshift1} follows trivially if $X(\la) h(\la) = 0$. Finally, assume that $X(\la) h(\la) \ne 0$ and $h(\la) \in \mathcal{N}_r (Q)$. Then use $\EL {} z(\la) = 0$, and perform the multiplication corresponding to the first block of $\EL {} z(\la)$, using the expressions of $z(\la)$ in \eqref{eq:defz} and $\EL{}$ in \eqref{eq:minbaspencil}, to get
\[
M(\la) N_1 (\la)^T h(\la) = K_2(\la)^T X(\la) h(\la).
\]
This equality implies, together with \eqref{eq:auxdegreeshift2}, that
\begin{align*}
1 + \deg (X(\la) h(\la)) & = \deg (K_2(\la)^T X(\la) h(\la)) \leq \deg(M(\la)) + \deg( N_1 (\la)^T h(\la)) \\
& \leq 1 + \deg( N_1 (\la)^T h(\la)),
\end{align*}
and, so, $\deg (X(\la) h(\la)) \leq \deg( N_1 (\la)^T h(\la))$. This inequality, together with \eqref{eq:auxdegreeshift1} and \eqref{eq:maxdegrees} thus prove \eqref{eq:degreeshift1}.

\smallskip

(b) Let us consider the matrix product
\[ B(\la):=
\begin{bmatrix}
N_1 (\la)^T \\-\widehat{N}_2 (\la) M(\la) N_1 (\la)^T
\end{bmatrix} [h_1(\la) \cdots h_p (\la)],
\]
and let us prove that their columns are a minimal basis of the rational subspace they span by applying a version of Theorem \ref{thm:minimal_basis} for columns. Note that for all $\lambda_0 \in \overline{\FF}$, $B(\la_0)$ has full column rank since $N_1 (\la_0)^T$ and $[h_1(\la_0) \cdots h_p (\la_0)]$ have both full column rank, since the columns of $N_1 (\la)^T $ and $[h_1(\la) \cdots h_p (\la)]$ are minimal bases. Next, observe that \eqref{eq:degreeshift1} implies that the highest column degree coefficient matrix $B_{hc}$ of $B(\la)$ has as a submatrix the highest column degree coefficient matrix $C_{hc}$ of $C(\la) := N_1 (\la)^T [h_1(\la) \cdots h_p (\la)]$. Since the column degrees of $N_1 (\la)^T$ are all equal, we have that $C_{hc}$ is the product of the highest column degree coefficient matrices of $N_1 (\la)^T$ and  $[h_1(\la) \cdots h_p (\la)]$, which have both full column rank because the columns of both matrices are minimal bases. So $C_{hc}$ has full column rank, as well as $B_{hc}$. This implies that the columns of $B(\la)$ are a minimal basis of a rational subspace $\mathcal{S}$. In addition, $\mathcal{S} \subseteq \mathcal{N}_r (\EL{})$ by part (a). Finally, note that $\mathcal{S} = \mathcal{N}_r (\mathcal{L})$ because $\dim(\mathcal{N}_r (Q)) = \dim (\mathcal{N}_r (\mathcal{L}))$, since $\EL{}$ is a strong linearization of $Q(\la)$ by Theorem \ref{thm:blockminlin}(b) and, then, Theorem 4.1 in \cite{de2014spectral} holds.
\end{proof}

\section{Proof of Lemma \ref{lemma_min_singular_value}} \label{sec:proof-sigmamin} In this appendix, we assume that $\e \ne 0$ and $\eta \ne 0$ according to Remark \ref{rem:2empty}.
We first reduce in Lemma \ref{lemm:appT1} the problem of computing $\sigma_{\min}(T)$  to the problem of computing the minimum singular value of a matrix of size $2 \e \eta \times (2 \e \eta + \e + \eta)$, which is much smaller than the size of $T$.

\begin{lemma} \label{lemm:appT1} Let $T$ be the matrix defined in \eqref{eq:linear_operator_equation} and
\begin{equation} \label{eq:defThat}
    \widehat{T} :=
    \left[
      \begin{array}{c|c}
        I_{\e}\otimes E_{\eta} &E_{\e}\otimes I_{\eta}\\\hline
        I_{\e}\otimes F_{\eta}&F_{\e}\otimes I_{\eta}
      \end{array}
    \right]\>,
  \end{equation}
where $\la F_k - E_k := L_k(\la)$ is the pencil in \eqref{eq:Lk}.
Then $\sigma_{\min} (T) = \sigma_{\min} (\widehat{T})$.
\end{lemma}
\begin{proof}
Since the Kronecker product is associative \cite[Chapter 4]{Horn}, we may write the matrix $T$ as
  \begin{align}\label{eq:T_reiter}
  \begin{split}
    T=&
    \left[
      \begin{array}{c|c}
        E_{\eta}\otimes I_{\rowdim}\otimes I_{\e}\otimes I_{\coldim} &I_{\eta}\otimes I_{\rowdim}\otimes E_{\e}\otimes I_{\coldim}\\\hline
        F_{\eta}\otimes I_{\rowdim}\otimes I_{\e}\otimes I_{\coldim}&I_{\eta}\otimes I_{\rowdim}\otimes F_{\e}\otimes I_{\coldim}
      \end{array}
    \right]\\
     =&
    \left[
      \begin{array}{c|c}
        (E_{\eta}\otimes I_{\rowdim})\otimes I_{\e}&I_{\eta\rowdim}\otimes E_{\e}\\\hline
        (F_{\eta}\otimes I_{\rowdim})\otimes I_{\e}&I_{\eta\rowdim}\otimes F_{\e}
      \end{array}
    \right]\otimes I_{\coldim}=: \widetilde{T}\otimes I_{\coldim}.
    \end{split}
  \end{align}
Thus, $\sigma_{\min}(T) = \sigma_{\min} (\widetilde{T})$ by \cite[Theorem 4.2.15]{Horn}. Let us perform a perfect shuffle on the matrix $\widetilde{T}$ on the right of \eqref{eq:T_reiter} to swap the order of the Kronecker products of its blocks.
Following Van Loan \cite{VanLoan2000}, there exist permutation matrices $S$, $R_1^T$ and $R_2^T$ of sizes $\e \eta m\times \e \eta m$, $\e(\eta+1)m\times \e (\eta +1) m$ and $(\e +1)\eta m\times (\e +1)\eta m$, respectively, such that
  \begin{align*}
    &\left[
      \begin{array}{c|c}
        S&\\\hline
        &S
      \end{array}
    \right]
    \left[
      \begin{array}{c|c}
        (E_{\eta}\otimes I_{\rowdim})\otimes I_{\e}&I_{\eta \rowdim}\otimes E_{\e}\\\hline
        (F_{\eta}\otimes I_{\rowdim})\otimes I_{\e}&I_{\eta \rowdim}\otimes F_{\e}
      \end{array}
    \right]
    \left[
      \begin{array}{c|c}
        R_1^{T}&\\\hline
        &R_2^{T}
      \end{array}
    \right]\\
    =&
    \left[
      \begin{array}{c|c}
        I_{\e}\otimes (E_{\eta}\otimes I_{\rowdim})&E_{\e}\otimes I_{\eta \rowdim}\\\hline
         I_{\e}\otimes (F_{\eta}\otimes I_{\rowdim})&F_{\e}\otimes I_{\eta \rowdim}
      \end{array}
    \right]
    =
    \left[
      \begin{array}{c|c}
        I_{\e}\otimes E_{\eta} &E_{\e}\otimes I_{\eta}\\\hline
         I_{\e}\otimes F_{\eta}&F_{\e}\otimes I_{\eta}
      \end{array}
    \right]\otimes I_{\rowdim}=\widehat{T}\otimes I_{\rowdim}.
  \end{align*}
Using again \cite[Theorem 4.2.15]{Horn}, we get $\sigma_{\min} (T)= \sigma_{\min} (\widetilde{T}) = \sigma_{\min} (\widehat{T})$.
\end{proof}

Lemma \ref{lemm:appT2} reduces the problem of computing the minimum singular value of $\widehat{T}$ in \eqref{eq:defThat} to compute the largest singular value of a matrix smaller than $\widehat{T}$, essentially with half its size, and with a simpler structure.
\begin{lemma} \label{lemm:appT2} Let $\widehat{T}$ be the matrix in \eqref{eq:defThat}. Then
\begin{equation}
\label{eq:sv}
\sigma_{\rm min}(\widehat{T}) = \sqrt{2-\sigma_{\rm max}(W_{\e,\eta}}) \, ,
\end{equation}
where $W_{\e,\eta} = I_\e \otimes E_\eta F_\eta^T+E_\e F_\e^T\otimes I_\eta \in \mathbb{R}^{\e \eta \times \e \eta}$ and $\sigma_{\rm max}(W_{\e,\eta})$ denotes its maximum singular value.
\end{lemma}
\begin{proof}
The singular values of $\widehat{T}$ are the square roots of the eigenvalues of
\[
\widehat{T} \widehat{T}^T =
\begin{bmatrix}
2I_{\e \eta} & W_{\e ,\eta} \\
W_{\e , \eta}^T & 2I_{\e \eta}
\end{bmatrix} =
2 \, I_{2 \e \eta}
+
\begin{bmatrix}
0 & W_{\e ,\eta} \\
W_{\e , \eta}^T & 0
\end{bmatrix},
\]
where $W_{\e,\eta} = I_\e \otimes E_\eta F_\eta^T+E_\e F_\e^T\otimes I_\eta$. It is well known (see, for instance, \cite[Theorem I.4.2]{stewartsunbook}) that the eigenvalues of $[0 \, , \, W_{\e ,\eta} \, ; \, W_{\e , \eta}^T \, , \, 0]$ are $\pm \sigma_1 (W_{\e,\eta}), \ldots, \pm \sigma_{\e \eta} (W_{\e,\eta})$, where   $\sigma_1 (W_{\e,\eta}) \geq \cdots \geq \sigma_{\e \eta} (W_{\e,\eta})$ are the singular values of $W_{\e,\eta}$. Therefore, the eigenvalues of $\widehat{T} \widehat{T}^T$ are $2 \pm \sigma_1 (W_{\e,\eta}), \ldots, 2 \pm \sigma_{\e \eta} (W_{\e,\eta})$, which implies the result. Observe that $\widehat{T} \widehat{T}^T$ is positive semidefinite and, thus, its eigenvalues are nonnegative.
\end{proof}

The advantage of the matrix $W_{\e,\eta}$ is that has a bidiagonal block Toeplitz structure with very simple blocks. This comes from the fact that
\[
E_k F_k^T = \begin{bmatrix}
0 \\
1 & 0 \\
& 1 & 0 \\
&& \ddots & \ddots \\
&& & 1 & 0
\end{bmatrix} =: J_k \in\mathbb{R}^{k \times k} \quad \mbox{(with $J_1 := 0_{1\times 1}$)},
\]
which implies
\begin{equation} \label{eq:Weeta}
W_{\e,\eta} = I_\e \otimes E_\eta F_\eta^T+E_\e F_\e^T\otimes I_\eta =
\underbrace{
\begin{bmatrix}
J_\eta \\
I_\eta & J_\eta \\
& I_\eta & J_\eta \\
& & \ddots & \ddots \\
& & & I_\eta & J_\eta
\end{bmatrix}}_{\displaystyle \e \; \mbox{block columns}}
\left. \phantom{\begin{array}{l} l\\l\\l\\l\\l \end{array}} \! \! \!\!\! \!\!\!\right\}
\e \mbox{ block rows} \,.
\end{equation}
This structure will allow us to compute explicitly the largest singular value of $W_{\e,\eta}$.
Without loss of generality, we assume that $\e \geq \eta$, since, otherwise, $W_{\e,\eta}$ is transformed  into $W_{\eta ,\e}$ with a perfect shuffle permutation , i.e., by interchanging the order of the Kronecker products in the summands of $W_{\e ,\eta}$.
In this situation, note that if $\eta =1$, then $W_{1,1} = 0_{1\times 1}$ and $W_{\e,1} = J_\e$ for $\e > \eta =1$. Therefore,
\begin{equation} \label{eq:eta=1}
\sigma_{\max}( W_{1,1})  = 0 \quad \mbox{and} \quad \sigma_{\max}( W_{\e,1})  = 1, \quad \mbox{if $\e > \eta = 1$}.
\end{equation}
If $\eta > 1$, then $\sigma_{\max}( W_{\e,\eta})$ can be computed with the help of  Lemma \ref{lemma:direct_sum}, where we show that $W_{\e,\eta}$ is permutationally equivalent to a direct sum involving the following two types of matrices
\begin{equation} \label{eq:GkMk}
M_k :=
\begin{bmatrix}
1 & 1 \\ & 1 & 1 \\ & & \ddots & \ddots \\ & & & 1 & 1 \\ & & & & 1
\end{bmatrix}\in\mathbb{R}^{k\times k}\quad \mbox{and}\quad
G_k :=
\begin{bmatrix}
1 \\
1 & 1 \\
& \ddots & \ddots \\
& & 1 & 1\\
& & & 1
\end{bmatrix}\in\mathbb{R}^{(k+1)\times k}.
\end{equation}

\begin{lemma}\label{lemma:direct_sum} Let $W_{\e , \eta}$ be the matrix in \eqref{eq:Weeta}, let $M_k$ and $G_k$ be the matrices in \eqref{eq:GkMk}, and assume that $\e \geq \eta$.
Then, there exist two permutation matrices $P_1$ and $P_2$ such that
\begin{equation}
\label{eq:direct_sum}
P_1W_{\e,\eta}P_2 =
\underbrace{(M_\eta\oplus M_\eta \oplus \cdots \oplus M_\eta)}_{\e-\eta\mbox{ times}}\oplus (G_{\eta-1}\oplus G_{\eta-1}^T)\oplus \cdots \oplus (G_1\oplus G_1^T) \oplus 0_{1\times 1}.
\end{equation}
\end{lemma}

\begin{proof} If $\eta =1$, then the result follows trivially from the discussion in the two lines above \eqref{eq:eta=1} with the convention $G_0 \oplus G_0^T := 0_{1 \times 1}$. Therefore, we assume in the rest of the proof that $\eta > 1$.
Observe that the $0_{1\times 1}$ block is a consequence of the fact that the first row and the last column of $W_{\e,\eta}$ are both zero. Thus, permuting the first row to the last row position produces the $0_{1\times 1}$ block. A complete formal proof is rather technical, but the key ideas are easy to follow. Therefore, we restrict ourselves to describe such ideas. In order to do this in a concise way we use in this proof the following notation: the column $2^{(3)}$ of $W_{\e , \eta}$ stands for the $2$nd column in the $3$rd block column of  $W_{\e , \eta}$. An analogous notation is used for rows and both notations are combined with the standard MATLAB's notation for submatrices.

Observe that $G_1$ is the submatrix of nonzero rows of $W_{\e,\eta} (:, 1^{(1)})$, which correspond to rows of $W_{\e,\eta}$ with the remaining entries equal to zero and, thus, this submatrix can be transformed via permutations into an explicit direct summand. $G_2$ is the submatrix of nonzero rows of $W_{\e,\eta} (:, [2^{(1)},1^{(2)}])$, which correspond to rows of $W_{\e,\eta}$ with the remaining entries equal to zero. $G_3$ is the submatrix of nonzero rows of $W_{\e,\eta} (:, [3^{(1)},2^{(2)},1^{(3)}])$, which correspond to rows of $W_{\e,\eta}$ with the remaining entries equal to zero. This process continues until we find that $G_{\eta -1}$ is the submatrix of nonzero rows of $W_{\e,\eta} (:, [(\eta-1)^{(1)},(\eta-2)^{(2)}, \ldots, 1^{(\eta-1)}])$, which correspond to rows of $W_{\e,\eta}$ with the remaining entries equal to zero. Note that we have started each of the previous submatrices with the $1$st, $2$nd, ..., $(\eta-1)$th columns of the first block column of $W_{\e , \eta}$. Since $W_{\e , \eta}$ is symmetric with respect to the main antidiagonal, $G_1^T, G_2^T, \ldots , G_{\eta -1}^T$ are obtained starting from the bottom with the $\eta$th, $(\eta-1)$th , ..., $2$nd rows of the last block row of $W_{\e , \eta}$. More precisely, $G_1^T$ comes from $W_{\e,\eta} (\eta^{(\e)},:)$, $G_2^T$ comes from $W_{\e,\eta} ([(\eta^{(\e-1)},(\eta-1)^{(\e)}],:)$, and so on until one gets $G_{\eta -1}^T$, which comes from $W_{\e,\eta} ([\eta^{(\e-\eta+2)} , \ldots , 3^{(\e-1)},2^{(\e)}],:)$. In this way, the direct summands $0_{1\times 1},G_1,G_1^T, G_2,G_2^T, \ldots , G_{\eta -1} , G_{\eta -1}^T$ have been identified for any $\e \geq \eta$. This leads directly to the proof in the case $\e = \eta$, because in this case the size and the number of entries equal to $1$ of $(G_{\eta-1}\oplus G_{\eta-1}^T)\oplus \cdots \oplus (G_1\oplus G_1^T) \oplus 0_{1\times 1}$ are equal to those of $W_{\eta,\eta}$.

Next, we prove the case $\e = \eta + 1$. The proof we propose is based on the fact that $W_{\eta,\eta}$ is the submatrix of $W_{\eta+1,\eta}$ lying in its $2,\ldots, \eta+1$ block rows and columns. This implies that the submatrices of $W_{\eta,\eta}$ are submatrices of $W_{\eta+1,\eta}$. Observe that the top rows of the submatrices of $W_{\eta,\eta}$ corresponding to $G_1, G_2, \ldots , G_{\eta-1}$, that is, the submatrices formed by the nonzero rows of $W_{\eta,\eta} (:, 1^{(1)}), W_{\eta,\eta} (:, [2^{(1)},1^{(2)}]), \ldots , W_{\eta,\eta} (:, [(\eta-1)^{(1)},(\eta-2)^{(2)}, \ldots, 1^{(\eta-1)}])$, when viewed as submatrices of $W_{\eta+1,\eta}$ correspond to rows of $W_{\eta+1,\eta}$ that have nonzero entries to the left of such submatrices and, thus, these submatrices cannot be transformed via permutations into block summands. In fact, $W_{\eta+1,\eta}(:,2^{(1)})$ combined with the $G_1$ of $W_{\eta,\eta}$ (not permuted) leads to the $G_2$ in $W_{\eta+1,\eta}$, $W_{\eta+1,\eta}(:,3^{(1)})$ combined with the $G_2$ of $W_{\eta,\eta}$ leads to the $G_3$ in $W_{\eta+1,\eta}$, and so on until we get that $W_{\eta+1,\eta}(:,(\eta-1)^{(1)})$ combined with the $G_{\eta-2}$ of $W_{\eta,\eta}$ leads to the $G_{\eta-1}$ in $W_{\eta+1,\eta}$. The column $W_{\eta+1,\eta}(:,\eta^{(1)})$ is different than the previous ones of $W_{\eta+1,\eta}$, since it has exactly one entry equal to $1$, while the previous ones have two entries equal to $1$. Therefore, $W_{\eta+1,\eta}(:,\eta^{(1)})$ combined with the $G_{\eta-1}$ of $W_{\eta,\eta}$ leads to an $M_{\eta}$ block in $W_{\eta+1,\eta}$.
The missing $G_1$ direct summand of $W_{\eta+1,\eta}$ comes from the nonzero rows of $W_{\eta+1,\eta}(:,1^{(1)})$. To summarize, we have identified the direct sum $M_\eta \oplus (G_{\eta-1}\oplus G_{\eta-1}^T)\oplus \cdots \oplus (G_1\oplus G_1^T) \oplus 0_{1\times 1}$ inside $W_{\eta+1,\eta}$. The proof is completed by noting that the sizes and the numbers of $1$s of these two matrices are equal.

The last part of the proof is an induction argument that follows exactly the steps explained in the previous paragraph. Let us assume that the result is true for any $W_{\e,\eta}$ with $\e > \eta$ and let us prove it for $W_{\e + 1,\eta}$.
As in the previous paragraph $W_{\e , \eta}$ is viewed as the submatrix of $W_{\e + 1, \eta}$ lying in its $2, \ldots , \e+1$ block rows and columns. Also
as in the previous paragraph, $W_{\e+1,\eta}(:,2^{(1)})$ combined with the $G_1$ of $W_{\e,\eta}$ leads to the $G_2$ in $W_{\e+1,\eta}$, $W_{\e+1,\eta}(:,3^{(1)})$ combined with the $G_2$ of $W_{\e,\eta}$ leads to the $G_3$ in $W_{\e+1,\eta}$, ... ,$W_{\e +1,\eta}(:,(\eta-1)^{(1)})$ combined with the $G_{\eta-2}$ of $W_{\e,\eta}$ leads to the $G_{\eta-1}$ in $W_{\e+1,\eta}$, and $W_{\e+1,\eta}(:,\eta^{(1)})$ combined with the $G_{\eta-1}$ of $W_{\e,\eta}$ leads to an $M_{\eta}$ block in $W_{\e+1,\eta}$. The $G_1$ direct summand of $W_{\e+1,\eta}$ comes from the nonzero rows of $W_{\e+1,\eta}(:,1^{(1)})$. The rest of submatrices of $W_{\e,\eta}$ producing direct summands lie in its $2,...,\e$ block rows, i.e., in the $3,\ldots , \e+1$ block rows and $2, \ldots, \e+1$ block columns of $W_{\e+1,\eta}$, and, thus, do not interact with other nonzero entries of $W_{\e+1,\eta}$, which implies that they remain as direct summands of $W_{\e+1,\eta}$, To summarize, we have identified the direct sum
\[
\underbrace{(M_\eta\oplus M_\eta \oplus \cdots \oplus M_\eta)}_{\e+1-\eta\mbox{ times}}\oplus (G_{\eta-1}\oplus G_{\eta-1}^T)\oplus \cdots \oplus (G_1\oplus G_1^T) \oplus 0_{1\times 1}
\]
inside $W_{\e+1,\eta}$. The proof concludes by noting that the sizes and the numbers of $1$s of this direct sum and $W_{\e+1,\eta}$ are equal.
\end{proof}

Now, we are in the position of computing $\sigma_{\max}(W_{\e,\eta})$.

\begin{proposition}\label{prop:normW}
Let $W_{\e,\eta}$ be the matrix in \eqref{eq:Weeta}. Then
\begin{equation}\label{eq:normW}
\sigma_{\max}(W_{\e,\eta}) = \left\{
\begin{array}{ll}
2\cos  \frac{\pi}{2\min\{\e,\eta\}+1},  & \mbox{ if }\e\neq \eta,  \\
2\cos \frac{\pi}{2\eta},  & \mbox{ if }\e = \eta .
\end{array}
\right.
\end{equation}
\end{proposition}
\begin{proof} As explained after the equation \eqref{eq:Weeta}, we may assume without loss of generality that $\e \geq \eta$. In addition, if $\eta =1$, then the result follows immediately from \eqref{eq:eta=1}. Thus, the rest of the proof assumes $\e \geq \eta > 1$.

Let us consider first the case $\e =\eta > 1$. Lemma \ref{lemma:direct_sum} implies that $\sigma_{\max}(W_{\eta,\eta}) = \max \{ \sigma_{\max}(G_{\eta-1}) ,..., \sigma_{\max}(G_{2}), \sigma_{\max}(G_{1}) \}$. In addition, since $G_k$ is a submatrix of $G_{k+1}$, we have that $\sigma_{\max}(G_{\eta-1}) \geq \cdots \geq \sigma_{\max}(G_{2}) \geq \sigma_{\max}(G_{1})$ \cite[Corollary 3.1.3]{Horn}. Therefore, $\sigma_{\max}(W_{\eta,\eta}) = \sigma_{\max}(G_{\eta-1})$. The singular values of $G_{\eta-1}$ are the square roots of the eigenvalues of
\[
G_{\eta-1}^T G_{\eta -1} =
\begin{bmatrix}
2 & 1 \\
1 & 2 & 1 \\
& 1 & \ddots & \ddots \\
& & \ddots & 2 & 1\\
& & & 1 & 2
\end{bmatrix} \in \mathbb{R}^{(\eta-1) \times (\eta-1)},
\]
which are known at least from the 1940s \cite[p. 111]{gantmacher-krein}. They are
\[
\lambda_j = 2\left(1 - \cos \frac{\pi j}{\eta} \right), \quad \mbox{for }j=1,2,\hdots,\eta-1.
\]
Therefore the maximum of these eigenvalues is
\[
\lambda_{\eta -1} = 2\left(1 - \cos \frac{\pi (\eta-1)}{\eta} \right) = 2\left(1 + \cos \frac{\pi}{\eta} \right) = 4 \cos^2 \frac{\pi}{2\eta} \, .
\]
The result follows from $\sigma_{\max}(W_{\eta,\eta}) = \sigma_{\max}(G_{\eta-1}) = \sqrt{\lambda_{\eta -1}}$.

Next, we consider the case $\e >\eta > 1$. In this situation, Lemma \ref{lemma:direct_sum} implies that $\sigma_{\max}(W_{\e,\eta}) = \max \{ \sigma_{\max}(M_{\eta}), \sigma_{\max}(G_{\eta-1}) ,..., \sigma_{\max}(G_{1}) \} = \sigma_{\max}(M_{\eta}),$
where we have used again that $G_k$ is a submatrix of $G_{k+1}$ and that $G_{\eta-1}$ is a submatrix of $M_\eta$. The singular values of $M_\eta$ are the square roots of the eigenvalues of $M_\eta M_\eta^T$, i.e., the square roots of the roots of the characteristic equation
\[
\det(\lambda I- M_\eta M_\eta^T) = \det
\begin{bmatrix}
(\lambda-2) & -1 \\
-1 & (\lambda-2) & -1 \\
& -1 & \ddots & \ddots \\
& & \ddots & (\lambda-2) & -1\\
& & & -1 &(\lambda-1)
\end{bmatrix} =0.
\]
With the change of variable $\lambda=2\mu+2$, the equation above becomes
\begin{equation}\label{eq:poly_mu}
\det
\begin{bmatrix}
2\mu & -1 \\
-1 & 2\mu & -1 \\
& -1 & \ddots & \ddots \\
& & \ddots & 2\mu & -1\\
& & & -1 &2\mu+1
\end{bmatrix}=U_\eta(\mu)+U_{\eta-1}(\mu) =0,
\end{equation}
where $U_\ell(\mu)$ is the degree-$\ell$ Chebyshev polynomial of the second kind. The first equality in \eqref{eq:poly_mu} can be obtained directly from \cite[eq. (11)]{kulkarni1999} by applying the recurrence relation of the Chebyshev polynomials of the second kind\footnote{The reader should take into account that in \cite{kulkarni1999} the characteristic polynomial is defined as $\det(M_\eta M_\eta^T-\la I)$ and the change of variable is slightly different.}. It can also be easily established from results in \cite{GOOD01011961}. Observe that Gershgorin Circle Theorem \cite[Theorem 7.2.1]{golubvanloan4} implies that the eigenvalues of $M_\eta M_\eta^T$ satisfy
$0\leq \lambda \leq 4$. Therefore, the roots of \eqref{eq:poly_mu} satisfy $-1\leq \mu \leq 1$. Moreover, we also have that $1$ and $-1$ are not roots of \eqref{eq:poly_mu} since $U_\eta(1)+U_{\eta-1}(1) = 2\eta+1\neq 0$ and $U_\eta(-1)+U_{\eta-1}(-1)=(-1)^\eta\neq 0$.
Thus, the roots of \eqref{eq:poly_mu} satisfy $-1< \mu < 1$.
With the change of variable $\mu = \cos\theta$, we get the equation
\[
U_\eta(\mu)+U_{\eta-1}(\mu) = \frac{1}{\sin\theta}(\sin(\eta+1)\theta+\sin \eta\theta) = 2  \,
\frac{\cos\frac{\theta}{2}}{\sin\theta} \, \sin\frac{(2\eta+1)\theta}{2} =0,
\]
whose roots are $\theta_j = 2\pi j/(2\eta + 1)$, $j=1,\ldots, \eta$ in the interval $0< \theta < \pi$. We finally obtain that the eigenvalues of $M_\eta M_\eta^T$ are
\begin{equation} \label{eq:laj}
\la_j = 2 + 2 \cos\frac{2j\pi}{2\eta+1},\quad \mbox{for }j=1,2,\hdots,\eta.
\end{equation}
The largest one is $\la_1$, which implies
\[ \sigma_{\max}(W_{\e,\eta}) =
\sigma_{\max} (M_\eta)
= \sqrt{ 2 + 2 \cos\frac{2\pi}{2\eta+1} }
=
2\cos\frac{\pi}{2\eta+1}.
\]

\end{proof}

Finally, Lemma \ref{lemma_min_singular_value} follows from combining Lemmas \ref{lemm:appT1} and \ref{lemm:appT2}, Proposition \ref{prop:normW} and a elementary trigonometric identity. Observe that $\sigma_{\min} (T) \ne 0$, which implies that $T$ has full row rank.

\section{Proof of Theorem \ref{thm:unperturbed21block}} \label{sec:proof-sigmamin2}
Taking into account that $L_\e (\la) \otimes I_n$ and $\Lambda_\e (\la)^T \otimes I_n$ are dual minimal bases with all their row degrees equal, respectively, to $1$ and $\e$, part (a) is an immediate consequence of Theorem \ref{thm:pencilminconv}. Part (b) can also be seen as a consequence of Theorem \ref{thm:polyminconv} (except the obvious equality $C_{0}(\Lambda_\e (\la)^T \otimes I_n) = I_{(\e + 1) n}$), although it can be deduced directly because the matrices
$C_{0}(\Lambda_\e (\la)^T \otimes I_n)$ and $C_{1}(\Lambda_\e (\la)^T \otimes I_n)$
are very simple.

In order to prove part (c), we first note that $C_{\e -1}(L_\e (\la) \otimes I_n)=C_{\e -1}(L_\e (\la) )\otimes I_n$ and $C_{\e}(L_\e (\la) \otimes I_n)=C_{\e}(L_\e (\la))\otimes I_n$. So, it suffices to look at $C_{\e -1}(L_\e (\la))$ and $C_{\e}(L_\e (\la))$. We then
point out that there exist diagonal sign scalings, $S_1, S_2, S_3, S_4,$ (and hence orthogonal matrices) which get rid of all negative signs in $C_{\e -1}(L_\e (\la))$ and $C_{\e}(L_\e (\la))$, and that with the notation at the beginning of Section \ref{sec:firststep} lead to:
$$
S_1C_{\e - 1} (L_\e (\la))S_2 =: \hat C_{\e - 1} (L_\e (\la))
= \underbrace{
	\begin{bmatrix}
	F_{\e} \\
	E_{\e} & \ddots \\
	& \ddots & F_{\e} \\
	& & E_{\e}
	\end{bmatrix}}_{\displaystyle \e \; \mbox{block columns}}
\left. \phantom{\begin{array}{l} l\\l\\l\\l\\l \end{array}} \! \! \!\!\! \!\!\!\right\}
\e+1 \; \mbox{block rows} \, ,
$$
and
$$
S_3C_{\e} (L_\e (\la))S_4 =: \hat C_{\e} (L_\e (\la)) =
\underbrace{
	\begin{bmatrix}
	F_{\e } \\
	E_{\e } & \ddots \\
	& \ddots & F_{\e } \\
	& & E_{\e }
	\end{bmatrix}}_{\displaystyle \e+1 \; \mbox{block columns}}
\left. \phantom{\begin{array}{l} l\\l\\l\\l\\l \end{array}} \! \! \!\!\! \!\!\!\right\}
\e+2 \mbox{ block rows} \,.
$$

Clearly we can as well look at the singular values of the matrices $\hat C_{\e - 1} (L_\e (\la))$ and $\hat C_{\e} (L_\e (\la))$ since they are orthogonally equivalent to $C_{\e - 1} (L_\e (\la))$ and $C_{\e} (L_\e (\la))$, respectively. We then show that there exist row and column permutations (and hence orthogonal transformations) that put $\hat C_{\e - 1} (L_\e (\la))$ and $\hat C_{\e} (L_\e (\la))$ in the following block diagonal forms
\begin{align}
\Pi_1 \hat C_{\e - 1} (L_\e (\la)) \Pi_2 & =
M_{\e}\oplus M_{\e}^T\oplus \cdots \oplus M_1\oplus M_1^T ,  \label{eq:lastpermut1}
\\
\Pi_3 \hat C_{\e} (L_\e (\la)) \Pi_4 & =
M_{\e}\oplus M_{\e}^T\oplus \cdots \oplus M_1\oplus M_1^T \oplus G^T_\e ,
\label{eq:lastpermut2}
\end{align}
where $M_k$ and $G_k$ were defined in \eqref{eq:GkMk}. Since a formal proof of \eqref{eq:lastpermut1} and \eqref{eq:lastpermut2} is long, we simply sketch the main ideas.
Notice that each of the matrices  $\hat C_{\e - 1}(L_\e (\la))$ and
$\hat C_{\e}(L_\e (\la))$ have one or two 1's in each column or row. Moreover,
note that $\hat C_{\e - 1}(L_\e (\la))$ has exactly $2 \e$ columns with only one ``1'' and exactly $2 \e$ rows with only one ``1'', while $\hat C_{\e}(L_\e (\la))$ has exactly $2 (\e +1)$ columns with only one ``1'' and exactly $2 \e$ rows with only one ``1''. Then, starting from the leading column in $\hat C_{\e - 1} (L_\e (\la))$ with a single ``1'', one can then reconstruct a staircase $M_\e$ and starting from its trailing column with a single ``1'', one can reconstruct a staircase $M_\e^T$. The corresponding index selection in Matlab notation for these two submatrices is:
$$ M_\e = \hat C_{\e-1}(L_\e (\la))(\e+1:\e+1:\e^2+\e, 1:\e+2:\e^2+\e-1) , $$
$$ M_\e^T = \hat C_{\e-1}(L_\e (\la))(1:\e+1:\e^2, 2:\e+2:\e^2+\e).
$$
After permuting these two blocks out of
$\hat C_{\e - 1} (L_\e (\la))$ one continues in a similar way to recover all other
blocks $M_k$ and $M_k^T$, for $k=\e-1, \ldots, 1$.
For the matrix $\hat C_{\e} (L_\e (\la))$, the procedure is similar, except that in the first step, one extracts
$$ G_\e^T = \hat C_{\e} (L_\e (\la)) (\e+1:\e+1:\e^2+\e, 1:\e+2:\e^2+2\e+1) $$
starting from the ``1'' in the leading column.
The rest of the extraction is similar to the one for the matrix $\hat C_{\e -1} (L_\e (\la))$.

So the smallest singular values of $\hat C_{\e - 1} (L_\e (\la))$ and $\hat C_{\e} (L_\e (\la))$ are those of the diagonal blocks with the smallest singular values. This turns out to be $M_\e$ for both matrices, since the smallest singular value of the full-row rank matrix $G_\e^T=\left[M_\e | e_\e \right]$ is larger than that of $M_\e$ \cite[Corollary 3.1.3]{Horn} and, according to \eqref{eq:laj}, $\sigma_{\rm min} (M_\e) < \sigma_{\rm min} (M_{\e -1}) < \cdots < \sigma_{\rm min} (M_{1})$.
The smallest singular value of $M_\e$ is the square root
of the smallest eigenvalue given in \eqref{eq:laj}~:
$$\sigma_{\min}(M_\e) = \sqrt{2+2\cos \left(\frac{2\e\pi}{2\e+1} \right)}=2 \sin\left(\frac{\pi}{4\e+2} \right).$$
The inequality $2 \sin(\frac{\pi}{4\e+2})\ge \frac{3}{2\e+1} \ge \frac{3}{2(\e+1)}$ follows then from the inequality $\sin (x) \geq 3x/\pi$ for $0 \leq x \leq \pi/6$ since we assumed $\e\ge 1$.

The proof of part (d) follows from the equality $C_{0}(\Lambda_\e (\la)^T \otimes I_n) = I_{(\e + 1) n}$ and the fact that an obvious column permutation $\Pi$ allows us to prove that $C_{1}(\Lambda_\e (\la)^T \otimes I_n) \, \Pi = I_n \oplus (I_{\e n} \otimes [1,1]) \oplus I_n$. Therefore, the singular values of $C_{1}(\Lambda_\e (\la)^T \otimes I_n)$ are $1$ (with multiplicity $2n$) and $\sqrt{2}$ (with multiplicity $\e n$).

\bibliographystyle{siam}
\bibliography{shortstrings,biblio}

\end{document}